\documentclass[lettersize,journal]{IEEEtran}
\usepackage{amsmath,amsfonts,amssymb}
\usepackage{balance}
\usepackage{algorithmic}
\usepackage{algorithm}
\usepackage{array}
\usepackage{textcomp}
\usepackage{stfloats}
\usepackage{url}
\usepackage{verbatim,footnote}
\usepackage{epsfig,epstopdf}
\usepackage{graphicx}
\usepackage{cite}
\usepackage{color,xcolor,makecell,multirow,diagbox}
\usepackage[justification=centering,font=small]{caption}
\usepackage[caption=false,font=small]{subfig}
\usepackage{CJK}
\usepackage{multicol}
\usepackage{tabularx}
\usepackage{paralist} 
\usepackage{enumitem}

\setlist[enumerate]{leftmargin=*,itemsep=0pt,parsep=0pt,topsep=0pt,partopsep=0pt}
\setlist[itemize]{leftmargin=*,itemsep=0pt,parsep=0pt,topsep=0pt,partopsep=0pt}

\usepackage{autobreak}
\newtheorem{theorem}{Theorem}
\newtheorem{myassumptionprimal}{Assumption}

\newtheorem{myassumptiondual}{Assumption}

\newtheorem{statement}{Statement}
\newtheorem{definition}{Definition}
\newtheorem{lemma}{Lemma}
\newtheorem{example}{Example}

\newtheorem{assumption}{Assumption}

\makeatletter
\renewcommand{\maketag@@@}[1]{\hbox{\m@th\normalsize\normalfont#1}}%
\makeatother

\makeatletter
\let\NAT@parse\undefined
\makeatother
\usepackage{hyperref}  
\hypersetup{hidelinks} 

\usepackage{etoolbox}
\makeatletter
\patchcmd{\@makecaption}
{\\}
{.\ }
{}
{}

\def\BibTeX{{\rm B\kern-.05em{\sc i\kern-.025em b}\kern-.08em
		T\kern-.1667em\lower.7ex\hbox{E}\kern-.125emX}}

\begin{document}
\title{Decomposition and Successive Decomposition Methods and Algorithms for Nonconvex Optimization}

\author{Yiqing Zhai, Ying Cui,~\IEEEmembership{Member,~IEEE,} and Danny H. K. Tsang,~\IEEEmembership{Life Fellow,~IEEE} 
\thanks{
	The authors are with IoT Thrust, The Hong Kong University of Science and Technology (Guangzhou), Guangzhou 511400, China (e-mail: yzhai837@connect.hkust-gz.edu.cn; yingcui@ust.hk; eetsang@ust.hk). This paper was presented in part at IEEE GLOBECOM 2025~\cite{zhaizfprimal}.
}
}

\maketitle

\begin{abstract} 
Existing results on decomposition methods and algorithms for nonconvex problems are minimal. 
Parallel decomposition algorithms do not exist for nonconvex~problems with coupling nonlinear equality constraints. 
Besides, decomposition structures (i.e., coupling variables and constraints) are not fully exploited in designing decomposition methods and algorithms. 
In this paper, we consider nonconvex problems with decomposition structures that are more general than those handled in the existing literature. 
We propose primal and dual decomposition and successive primal and successive dual decomposition methods and algorithms for these nonconvex problems,  
which exploit decomposition structures, allow for parallel and distributed~implementations,  produce the original nonconvex problems' stationary points, and offer good opportunities to achieve superior tradeoff between convergence performance and computation~time.
Finally, we compare the proposed methods and algorithms, extend them to indirect and two-level decomposition methods and algorithms, and provide examples and numerical results to demonstrate their respective values. 
Notably, the proposed decomposition and~successive decomposition methods and algorithms generalize the basic ones for convex problems, and the proposed successive decomposition methods and algorithms extend the existing ones for nonconvex problems, together enriching the decomposition theory.
\end{abstract}


\section{Introduction}

\IEEEPARstart{L}{arge}-scale optimization problems have various applications in several engineering fields, such as information processing, control, communication networks, machine learning, and transportation \cite{scutari2018parallel}.
Solving large-scale problems requires substantial computational resources (time and storage) and is particularly challenging.
Decomposition methods and algorithms have been developed to alleviate such requirements. 
The basic idea is to decompose a large-scale problem into smaller subproblems that can be solved separately (in parallel or sequentially and in a centralized or distributed  manner) and are coordinated by a master problem.
Decomposition inherently reduces the overall computational complexity and hence computation time for solving a large-scale problem.\footnote{The overall computational complexity for solving a problem usually grows more than linearly with the problem size~\cite{boyd2007notes}.} 
If the subproblems are solved in parallel, further computation time reduction can be achieved. 
Decomposition first appeared in early work on large-scale linear problems in the 1960s \cite{dantzig1960decomposition,bender1962partitioning} and has been extended to more general convex problems \cite{lasdon1970optimization,bertsekas1997parallel,bertsekas2016nonlinear,boyd2007notes}.

Existing decomposition methods for convex problems can generally be classified into primal and dual decomposition methods that naturally handle coupling variables and constraints, respectively.
In primal decomposition, the original (primal) problem is decomposed into smaller (primal) subproblems that can be separately solved by fixing the coupling variables, and the coupling variables are handled by solving a master (primal) problem.
In dual decomposition, the dual problem of the original problem, formed by relaxing the coupling constraints, is decomposed into smaller (dual) subproblems, and the coupling constraints are dealt with by solving a master (dual) problem. 
Based on different algorithms for solving the master (primal or dual) problem, (primal or dual) decomposition methods can lead to various (primal or dual) decomposition algorithms.
For example, 
in the (primal and dual) decomposition algorithms in \cite[Section~7.6.2]{bertsekas2016nonlinear}, the cutting plane method is used to solve the master (primal and dual) problems. 
In the (primal and dual) decomposition algorithms in \cite{boyd2007notes}, the master (primal and dual) problems are solved using the subgradient method.
Successful applications of these decomposition algorithms can be found in many diverse fields, such as communication networks \cite{palomar2005convex,palomar2005minimum,johansson2006mathematical,han2017backhaul,pennanen2011decentralized},  
computer vision \cite{torresani2013dual,komodakis2007mrf,komodakis2011mrf,jojic2010accelerated},
and natural language processing \cite{rush2010dual,rush2012tutorial}. 
However, these decomposition methods and algorithms mentioned above are applicable only for convex problems,
relying on the subproblems' globally optimal points and the master problem’s convexity.

Some recent work \cite{scutari2017parallel,scutari2014decomposition,alvarado2014new,shi2020penalty} investigates decomposition algorithms for nonconvex problems. 
Expressly, \cite{scutari2017parallel,scutari2014decomposition,alvarado2014new} propose Successive Convex Approximation (SCA)-based decomposition algorithms for nonconvex problems with separable convex constraints \cite{scutari2014decomposition, alvarado2014new} or generally nonconvex coupling inequality constraints \cite{scutari2017parallel}, which allow for parallel implementations. 
The idea is to convert the original nonconvex problem into a sequence of successively refined convex approximate problems using the SCA method \cite{beck2010sequential} and then directly separate each convex approximate problem into smaller subproblems that can be solved in parallel \cite{scutari2014decomposition, alvarado2014new} or solve each convex approximate problem using the standard (primal and dual) decomposition algorithms for convex problems \cite{scutari2017parallel}. 
Notice that the SCA-based Primal Decomposition (SCAPD) algorithm  in \cite{scutari2017parallel} does not fully exploit the decomposition structure of the original nonconvex problem in convex approximation and  solves each convex approximate problem's equivalent problem with a larger size (including slack variables and additional constraints), yielding possibly worse convergence performance and higher computational complexity, and the SCA-based Dual Decomposition (SCADD) algorithm in \cite{scutari2017parallel} does not consider coupling linear equality constraints.
In \cite{shi2020penalty}, the authors propose a Penalty Dual Decomposition (PDD) algorithm for nonconvex problems with coupling variables and constraints by integrating the penalty method \cite{bertsekas2016nonlinear} and the Augmented Lagrangian (AL) method \cite{hestenes1969multiplier} and utilizing a Block-Successive-Upper-Bound-Minimization (BSUM)-based algorithm,  i.e., rBSUM \cite{shi2020penalty}, to decompose the AL problems into smaller subproblems. 
Nevertheless, PDD does not allow for parallel implementations (due to rBSUM's sequential update mechanism) and therefore cannot effectively reduce the computation time.

To sum up, the existing decomposition methods and algorithms for nonconvex problems \cite{scutari2017parallel,scutari2014decomposition,alvarado2014new,shi2020penalty} cannot handle nonlinear equality constraints, allow for parallel implementations, or exploit complete decomposition structures, possibly limiting their effectiveness and efficiency. 
To address these issues, in this paper, we investigate nonconvex problems with decomposition structures (i.e., coupling variables and~constraints) and design new (primal and dual) decomposition methods and algorithms for nonconvex problems, which exploit decomposition structures, allow for parallel and distributed implementations, and produce the original nonconvex problems' stationary points.
More detailed contributions are summarized as~follows.

1) We consider two nonconvex problems with coupling variables in the objective functions and constraints.
Specifically, the equality constraints of the first problem are generally nonlinear, while those of the second problem are linear. 
For the first nonconvex problem, we propose a Primal Decomposition Method (PD-M), which no longer relies on the requirements for the subproblems' optimal points and the master problem's convexity.
Based on PD-M, we propose a Primal Decomposition Algorithm (PD-A), which utilizes SCA-based algorithms to solve the master problem.
For the second nonconvex problem, we propose a Successive Primal Decomposition Method (SPD-M), which converts the original nonconvex problem into a sequence of successively refined convex subproblems and master (primal) problems.
Based on SPD-M, we propose a Successive Primal Decomposition Algorithm (SPD-A), which solves each convex approximate problem without increasing its problem size, possibly achieving better convergence performance and shorter computation time than SCAPD \cite{scutari2017parallel}.

2) We consider two nonconvex problems with coupling equality and inequality constraints.
In particular, the coupling equality constraints of the first problem are generally nonlinear, whereas those of the second problem are linear. 
For the first nonconvex problem, we propose a Dual Decomposition Method (DD-M), which no longer relies on the original problem’s strong duality, the subproblems' optimal points, and the master problem's convexity.
Based on \mbox{DD-M}, we propose a Dual Decomposition Algorithm (DD-A), which utilizes SCA-based algorithms to solve the master problem with its variables updated in closed form.
For the second nonconvex problem, we propose a Successive Dual Decomposition Method \mbox{(SDD-M)}, which converts the original nonconvex problem into a sequence of successively refined convex subproblems and master (dual) problems and does not rely on the original problem’s strong duality.
Based on SDD-M, we present a Successive Dual Decomposition Algorithm (SDD-A), which slightly extends SCADD \cite{scutari2017parallel} to nonconvex problems with coupling linear equality constraints.

3) We compare the proposed methods and algorithms (i.e., \mbox{PD-M/A}, \mbox{SPD-M/A}, \mbox{DD-M/A}, and \mbox{SDD-M/A}) for nonconvex problems,  extend them to indirect and two-level decomposition methods and algorithms, and provide examples as well as numerical results to demonstrate their strengths.
Specifically, \mbox{PD-M/A} and \mbox{DD-M/A} apply to a broader range of nonconvex problems, exploit more decomposition structures, and yield possibly simpler algorithm structures but require stronger conditions than \mbox{SPD-M/A} and \mbox{SDD-M/A}. 
In addition, the proposed four algorithms exhibit their own strengths in the tradeoff between convergence performance and computation time in respective nonconvex problems.
On the other hand, \mbox{PD-M} and \mbox{DD-M} can reduce to the basic primal and dual decomposition methods for convex problems~\cite{boyd2007notes,palomar2006tutorial}, respectively, and PD-A, DD-A, SPD-M/A, and SDD-M/A can serve as new decomposition methods and algorithms for convex problems, all naturally producing optimal points. 
Furthermore, we extend the proposed methods and algorithms for nonconvex problems with coupling variables or constraints to indirect decomposition methods and algorithms for nonconvex problems without explicit coupling variables or constraints and two-level decomposition methods and algorithms for nonconvex problems with explicit coupling variables and constraints.

\begin{figure}[t]
	\centering
	\includegraphics[width=0.4\textwidth]{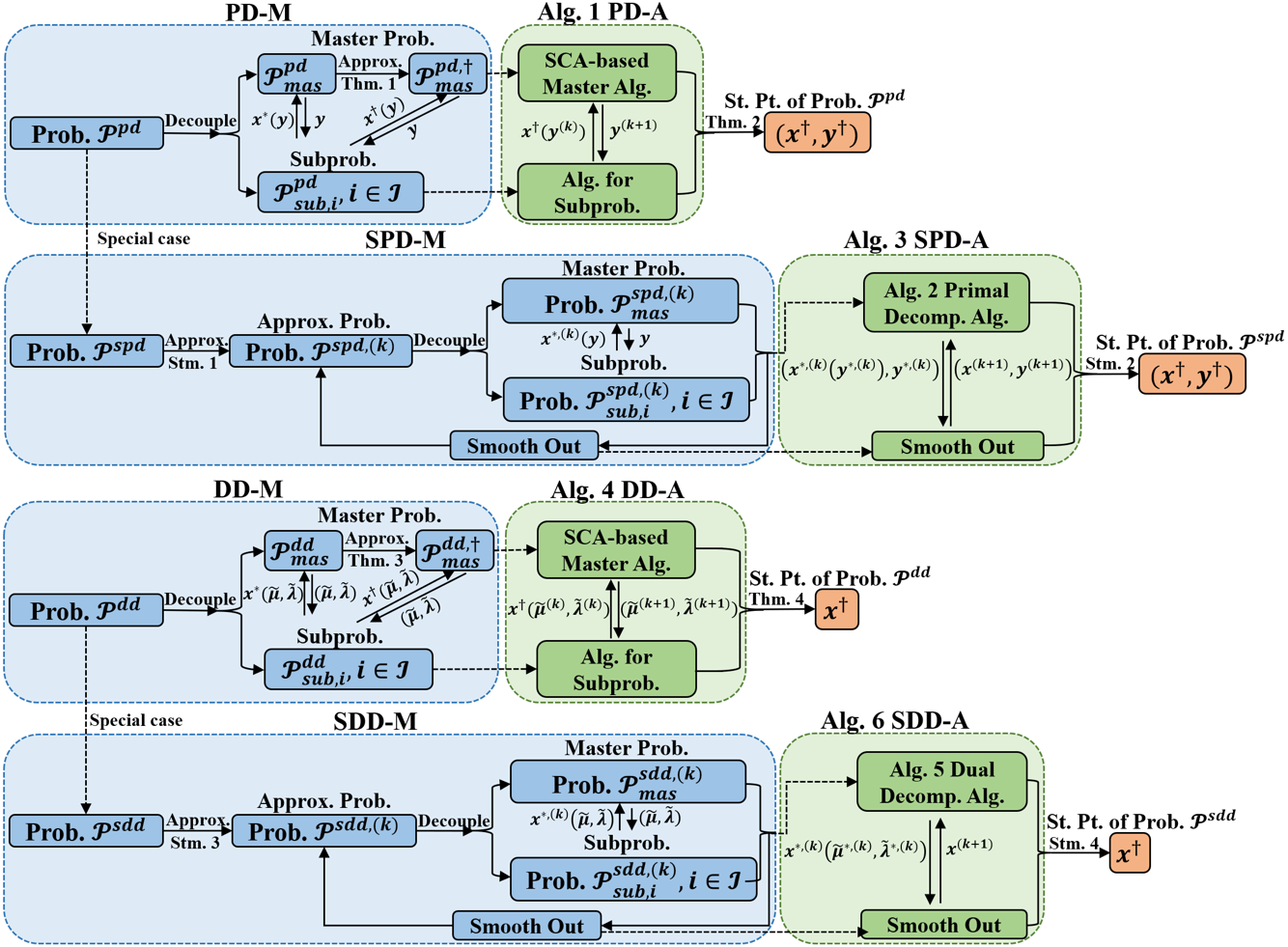}
	\caption{Solution framework.}
	\label{fig framework}
\end{figure}

Fig.~\ref{fig framework} illustrates the solution framework of this paper.   
The rest of this paper is organized as follows. 
Section~\ref{sec pre} introduces some fundamental definitions and assumptions employed in this paper.
Sections~\ref{primal section} and \ref{sp section} investigate nonconvex problems with coupling variables and propose PD-M/A and SPD-M/A, respectively.
Sections~\ref{dual section} and \ref{sd section} study nonconvex problems with coupling constraints and develop DD-M/A and \mbox{SDD-M/A}, respectively.
Section~\ref{sec compare} provides comparisons of the proposed methods and algorithms for nonconvex problems and elucidates their connections with the basic decomposition methods and algorithms for convex problems.
Section~\ref{section extension} extends the proposed methods and algorithms to indirect and two-level decomposition methods and algorithms. 
Section~\ref{sec exm} and \ref{sec numerical results} provide examples of the proposed algorithms and their numerical results, respectively.
Finally, Section~\ref{sec conclusion} draws some conclusions and discusses potential future research directions.

\textbf{Notation}: The set of positive integers smaller than or equal to a positive integer $I$ is denoted by $\mathcal{I}\triangleq\{1, \cdots, I\}$.
The identity matrix of size $n$ is denoted by $\mathbf{I}_{n}$, and the $n\times m$ zero matrix is denoted by $\mathbf{0}_{n\times m}$.
The complement of a set $\mathcal{A}\subseteq \mathcal{C}$ is denoted by $\mathcal{A}^c$.
The interior of a set $\mathcal{A}$ is denoted by $\operatorname{int}(\mathcal{A})$.
Superscripts $(\cdot)^{T}$ denote the transpose.
The Hadamard product of two vectors $\mathbf{x}$ and $\mathbf{y}$ is denoted by $\mathbf{x} \odot \mathbf{y}$. 
The $i$-th coordinate of a vector $\mathbf{x}$ is denoted by $\langle\mathbf{x}\rangle_i$.
$\mathbf{x}\triangleq(\mathbf{x}_{1},\cdots,\mathbf{x}_{I})$ with $\mathbf{x}_{i}\in\mathbb{R}^{n_{i}}$, $i\in\mathcal{I}$ represents a $\sum_{i\in \mathcal{I}}n_{i}$-dimensional vector.
$\mathbf{A}\triangleq[\mathbf{x}_{1}\cdots\mathbf{x}_{I}]$ with $\mathbf{x}_{i}\in\mathbb{R}^{n}$, $i\in\mathcal{I}$ represents a $n\times I$ matrix.
The $\ell_{2}$-norm of a vector $\mathbf{x}$ is denoted by $\|\mathbf{x}\|_2$. 
Operator $\operatorname{P}_{\mathcal{C}}: \mathbb{R}^{n} \rightarrow \mathcal{C}$ represents the Euclidean projection of a vector in $\mathbb{R}^{n}$ onto the nonempty closed convex set $\mathcal{C}$ and $\operatorname{P}_{\mathbb{R}_{+}^{n}}$ is denoted by $[\cdot]_{+}$ for short. 
The submatrix consisting of the $i$-th to $l$-th rows and the $j$-th to $k$-th columns of a matrix $\mathbf{A}$ is denoted by $\mathbf{A}[i:l,  j:k]$.
Function $f:\mathbb{R}^{n}\rightarrow \mathbb{R}$ represents a real-valued function, and $\nabla f(\mathbf{x})\in \mathbb{R}^{n}$ and $\partial f(\mathbf{x})\in \mathbb{R}^{n}$ represent its gradient and subgradient at $\mathbf{x}\in\mathbb{R}^{n}$, respectively.
Function $\mathbf{f}:\mathbb{R}^{n}\rightarrow \mathbb{R}^{m}$ represents a vector-valued function, and $\nabla \mathbf{f}(\mathbf{x})\in\mathbb{R}^{n\times m}$ represents its gradient, i.e., the transpose of its Jacobian matrix, at $\mathbf{x}\in\mathbb{R}^{n}$.
Unless otherwise specified, the superscript $(\cdot)^{(k)}$ in $\mathbf{x}^{(k)}$ is reserved for the algorithm iteration and $\mathbf{x}^{(k)}$ represents the value of $\mathbf{x}$ at the $k$-th iteration.
$\mathcal{N}(\mu,\sigma^2)$ and $\mathcal{U}(a,b)$ represent  normal distribution with mean $\mu$ and variance $\sigma^2$ and uniform distribution on $[a,b]$, respectively.
Other key notations used in this paper are listed in Table~\ref{table notation}.

\begin{table}[t]
	\centering
	\caption{Key notation.}
	\begin{tabular}{|p{3.15cm}|p{4.75cm}|} 
		\hline
		Notation & Explanation  \\ 
		\hline
		$\mathcal{P}$, M, A & Problem, method, algorithm  \\ 
		\hline	
		PD ($pd$), SPD ($spd$) & Primal decomposition, successive primal decomposition\\ 
		\hline	
		DD ($dd$), SDD ($sdd$) & Dual decomposition, successive dual decomposition\\ 
		\hline	
		$sub$, $mas$ & Subproblem, master problem\\
		\hline	
		$(\cdot)^*$, $(\cdot)^\dagger$ & Related to optimal point, stationary point\\
		\hline	
		$\mathbf{x}$, $\mathbf{y}$, $\mathbf{x}_{i}$ & Primal variables\\
		\hline	
		$f$ ($f_{i}$) & Objective function (its component) \\
		\hline	
		$\mathbf{g}$, $\tilde{\mathbf{g}}_{i}$, $\mathbf{g}_{i}$, $\mathbf{h}$, $\tilde{\mathbf{h}}_{i}$ & Constraint functions or their components\\
		\hline	
		$\boldsymbol{\mu}$,  $\boldsymbol{\mu}_{i}$, $\tilde{\boldsymbol{\mu}}$, $\tilde{\boldsymbol{\mu}}_{i}$, $\boldsymbol{\lambda}$,  $\tilde{\boldsymbol{\lambda}}$, $\tilde{\boldsymbol{\lambda}}_{i}$ & Dual variables\\
		\hline	
		$L$ ($L_{i}$), $q$ & Lagrangian (its component), dual function \\
		\hline	
		$F$,  $F_{i}$, $F^{\dagger}$, $\mathbf{G}$, $\mathbf{G}_{i}$, $\tilde{\mathbf{G}}_{i}$,  $Q^{\dagger}$ & Approximate functions\\
		\hline	
		$\gamma$, $\gamma_{in}$ /  $k$, $t$ & Stepsize / iteration index\\
		\hline	
	\end{tabular}
	\label{table notation}
\end{table}

\textbf{General Statements}: 
Optimization problems are feasible and, in general, nonconvex. 
Objective and inequality constraint functions are generally nonconvex. The goal of solving each generally nonconvex problem is to find its stationary points.\footnote{Most iterative algorithms for convex and nonconvex problems are designed to solve the KKT conditions and return stationary points. For convex problems, any stationary point is a  globally optimal point. For nonconvex problems, a stationary point can be a local/global minimum point, local/global maximum point, or saddle point. } 
All subproblems are separately solved (in parallel or sequentially and in a centralized or distributed  manner).

\section{Preliminaries}\label{sec pre}
In this section, we present some fundamental definitions and assumptions employed in this paper.
Consider a nonconvex problem $\mathcal{P}(\boldsymbol{\theta})$ parameterized by $\boldsymbol{\theta}\in \Theta$:
\begin{align}
	\mathcal{P}(\boldsymbol{\theta}):\ \min_{\mathbf{x}} \ &f(\mathbf{x},\boldsymbol{\theta})\nonumber\\  
	s.t. \ 	&\mathbf{g}(\mathbf{x},\boldsymbol{\theta})\preceq \mathbf{0},\label{pre ineq-cons}\\
	&\mathbf{h}(\mathbf{x},\boldsymbol{\theta})=\mathbf{0},\\
	&\mathbf{x}\in \mathcal{X},
\end{align}
where $\mathbf{x}\in \mathbb{R}^{n}$,
$\mathcal{X}$ is a nonempty, closed, and convex set contained in an open set $\mathcal{U}$, and $f:\mathcal{U}\times \Theta\rightarrow \mathbb{R}$, $\mathbf{g}:\mathcal{U}\times \Theta\rightarrow \mathbb{R}^{r}$, and $\mathbf{h}:\mathcal{U}\times \Theta\rightarrow \mathbb{R}^{m}$ are continuously differentiable on $\mathcal{U}\times\Theta$.
Notably, $\mathbf{h}$ is generally not affine in $\mathbf{x}$.

First, we introduce some basic definitions that are used throughout this paper.

\begin{definition}[Karush-Kuhn-Tucker (KKT) Conditions {\cite[Section 5.3.3]{boyd2004convex}}]
	\label{definition kkt-conditions}
	The KKT conditions for Problem $\mathcal{P}(\boldsymbol{\theta})$ for fixed $\boldsymbol{\theta}$ are given by:
	\begin{gather}
		\begin{gathered}
			(\nabla_{\mathbf{x}} f(\mathbf{x},\boldsymbol{\theta})\! +\! \nabla_{\mathbf{x}}\mathbf{g}(\mathbf{x},\!\boldsymbol{\theta})\boldsymbol{\mu}\!+\! \nabla_{\mathbf{x}}\mathbf{h}(\mathbf{x},\boldsymbol{\theta})\boldsymbol{\lambda})^{T}(\mathbf{x}'\!-\!\mathbf{x})
			\!\geq\!0, \\
			\forall \mathbf{x}'\in\mathcal{X},\label{KKT-1}
		\end{gathered}
		\\
\boldsymbol{\mu}\odot\mathbf{g}(\mathbf{x},\boldsymbol{\theta})=\mathbf{0},\label{KKT-2}\\
		\boldsymbol{\mu}\succeq\mathbf{0}, \label{KKT-3}
		\\
		\mathbf{g}(\mathbf{x},\boldsymbol{\theta})\preceq \mathbf{0},
		\mathbf{h}(\mathbf{x},\boldsymbol{\theta})=\mathbf{0}, 
		\mathbf{x}\in \mathcal{X}.\label{KKT-4}
	\end{gather}
\end{definition}

If the point $\mathbf{x}$ in the KKT conditions is an interior point of $\mathcal{X}$, i.e., $\mathbf{x}\in\operatorname{int}(\mathcal{X})$, then \eqref{KKT-1} reduces to:
\begin{align}
	\nabla_{\mathbf{x}} f(\mathbf{x},\boldsymbol{\theta})\! + \! \nabla_{\mathbf{x}}\mathbf{g}(\mathbf{x},\boldsymbol{\theta})\boldsymbol{\mu}\!+\! \nabla_{\mathbf{x}}\mathbf{h}(\mathbf{x},\boldsymbol{\theta})\boldsymbol{\lambda} =\mathbf{0}.\label{KKT-5}
\end{align}

\begin{definition}[KKT system]\label{def KKT-system}
	Given a point $\mathbf{x}\in\mathcal{U}$, its KKT system for Problem $\mathcal{P}(\boldsymbol{\theta})$ for fixed $\boldsymbol{\theta}$ is given by:\footnote{The KKT system incorporates the KKT conditions in (\ref{KKT-2}) and (\ref{KKT-5}).}
	\begin{equation*}\label{KKT sys}
		\begin{aligned}
			\begin{cases}
				\nabla_\mathbf{x} \mathbf{g}(\mathbf{x}, \boldsymbol{\theta})\boldsymbol{\mu}+\nabla_\mathbf{x} \mathbf{h}(\mathbf{x}, \boldsymbol{\theta})\boldsymbol{\lambda}=-\nabla_\mathbf{x} f(\mathbf{x}, \boldsymbol{\theta}),\\
				\langle\boldsymbol{\mu}\rangle_{j}=0, \quad j\notin \mathcal{A}(\mathbf{x}, \boldsymbol{\theta}),
			\end{cases}
		\end{aligned}
	\end{equation*}
	where $\mathcal{A}(\mathbf{x}, \boldsymbol{\theta})\triangleq\{j\in\{1, \cdots, r\}  |  \langle\mathbf{g}(\mathbf{x}, \boldsymbol{\theta})\rangle_{j}=0\}$ represents the index set of the inequality constraints in \eqref{pre ineq-cons} of Problem $\mathcal{P}(\boldsymbol{\theta})$ that are active at $\mathbf{x}$.
\end{definition}

\begin{definition}[KKT Function]\label{definition kkt-function}
	A vector function $\mathbf{k}:\mathcal{U}\times\mathbb{R}^{r}\times\mathbb{R}^{m}\times\Theta\rightarrow \mathbb{R}^{n+r+m}$, given by:\footnote{The KKT function incorporates the functions in the equality conditions of the KKT conditions in (\ref{KKT-2}), (\ref{KKT-4}), and (\ref{KKT-5}). It is introduced for notation simplicity and will be used in Assumptions~\ref{primal assump-stationary-2} and~\ref{dual assump-stationary-2}.} 
	\begin{align*}
		\mathbf{k}(\mathbf{x},\boldsymbol{\mu},\boldsymbol{\lambda},\boldsymbol{\theta})
		\triangleq\begin{pmatrix}
			\nabla_\mathbf{x} f(\mathbf{x}, \boldsymbol{\theta})+\nabla_\mathbf{x} \mathbf{g}(\mathbf{x}, \boldsymbol{\theta})\boldsymbol{\mu}+\nabla_\mathbf{x} \mathbf{h}(\mathbf{x}, \boldsymbol{\theta})\boldsymbol{\lambda}
			\\
			\boldsymbol{\eta}(\mathbf{x},\boldsymbol{\mu},\boldsymbol{\theta})
			\\
			\mathbf{h}(\mathbf{x}, \boldsymbol{\theta})
		\end{pmatrix},
	\end{align*}
	is called the KKT function of Problem $\mathcal{P}(\boldsymbol{\theta})$ for $\boldsymbol{\theta}\in\Theta$, where
	\begin{align*}
		\langle\boldsymbol{\eta}(\mathbf{x},\boldsymbol{\mu},\boldsymbol{\theta})\rangle_{j}\triangleq\begin{cases}
			\langle\mathbf{g}(\mathbf{x}, \boldsymbol{\theta})\rangle_{j} & j\in \mathcal{A}(\mathbf{x}, \boldsymbol{\theta})\\
			\langle\boldsymbol{\mu}\rangle_{j} & j\notin \mathcal{A}(\mathbf{x}, \boldsymbol{\theta})
		\end{cases}, \ j=1, \cdots, r.
	\end{align*}
\end{definition}

\begin{definition}[Stationary Points {\cite[Definition 2]{scutari2017parallel}}]\label{def stationary point}
	A point $\mathbf{x}^{\dagger}(\boldsymbol{\theta})$ is a stationary point of Problem $\mathcal{P}(\boldsymbol{\theta})$ for fixed $\boldsymbol{\theta}$, if there are Lagrange multipliers $\boldsymbol{\mu}(\boldsymbol{\theta})$ and $\boldsymbol{\lambda}(\boldsymbol{\theta})$, that together with $\mathbf{x}^{\dagger}(\boldsymbol{\theta})$, satisfy the KKT conditions in \eqref{KKT-1}-\eqref{KKT-4}. 
\end{definition}

Given a stationary point $\mathbf{x}^{\dagger}(\boldsymbol{\theta})$, if $\mathbf{x}^{\dagger}(\boldsymbol{\theta})\in\operatorname{int}(\mathcal{X})$, we can find the corresponding Lagrange multipliers $\boldsymbol{\mu}(\boldsymbol{\theta})$ and $\boldsymbol{\lambda}(\boldsymbol{\theta})$ by solving its KKT system (involving only linear equalities) for $\boldsymbol{\mu}\succeq\mathbf{0}$ and $\boldsymbol{\lambda}$; otherwise, we need to solve the KKT conditions (consisting of inequalities).\footnote{This will be utilized in Step~5 of Algorithm~\ref{primal alg:alg1}, Step~5 of Algorithm~\ref{sp alg:alg1-PD}, and Step~5 of Algorithm~\ref{dual alg:alg1}.} 
Besides, stationary points of a nonconvex problem can be local minimum points.
However, in some rare cases \cite[Examples 4.1.1 and 4.3.6]{bertsekas2016nonlinear}, no local minimum points are stationary~points.

\begin{definition}[Regular Points {\cite[Definition 1]{scutari2017parallel}}]\label{def regularity}
	A feasible point $\mathbf{x}^{\ddagger}(\boldsymbol{\theta})$ of Problem $\mathcal{P}(\boldsymbol{\theta})$ for fixed $\boldsymbol{\theta}$ is called regular if the Mangasarian-Fromovitz Constraint Qualification (MFCQ) holds at $\mathbf{x}^{\ddagger}(\boldsymbol{\theta})$, i.e., if the following equations:
	\begin{gather*}
		(\nabla \mathbf{g}(\mathbf{x}^{\ddagger}(\boldsymbol{\theta}),\boldsymbol{\theta})\boldsymbol{\mu}\!+\!\nabla\mathbf{h}(\mathbf{x}^{\ddagger}(\boldsymbol{\theta}),\boldsymbol{\theta})\boldsymbol{\lambda})^{T}\!(\mathbf{x}-\mathbf{x}^{\ddagger}(\boldsymbol{\theta}))\!\geq\!0, \forall \mathbf{x}\in\mathcal{X},\\
		\boldsymbol{\mu}\succeq \mathbf{0}, \ \langle\boldsymbol{\mu}\rangle_{j}=0, \ \forall j\notin\mathcal{A}(\mathbf{x}^{\ddagger}(\boldsymbol{\theta}),\boldsymbol{\theta}),
	\end{gather*}
	imply $\boldsymbol{\mu}=\mathbf{0}$ and $\boldsymbol{\lambda}=\mathbf{0}$.
\end{definition}

Every regular local minimum point of a nonconvex problem is a stationary point \cite[Proposition 4.3.6]{bertsekas2016nonlinear}.

\begin{definition}[Equivalent Nonconvex Problems]\label{definition equivalent-problem}
	Two nonconvex problems are equivalent if from any stationary point of one, a stationary point of the other is readily found, and vice versa.
\end{definition}

Definition~\ref{definition equivalent-problem} extends the definition of equivalent problem for convex problems and is more reasonable for nonconvex problems (see Appendix~\ref{app equivalent nonconvex prob} for some general transformations that yield equivalent nonconvex problems).

Next, we present some assumptions that are used throughout the paper.
To simplify the presentation, we assume the following regularity condition for each problem in the paper.
\begin{assumption}\label{regularity}
	All feasible points of a problem are regular.\footnote{In practice, Assumption~\ref{regularity} is generally satisfied by a large class of problems.
		Even if it is not satisfied, one could relax this assumption and require regularity only at specific points \cite{scutari2017parallel}.} 
\end{assumption}

Assumption~\ref{regularity} ensures that every local minimum point of a problem is regular and hence stationary.
Under Assumption~\ref{regularity}, we can obtain local minimum points of nonconvex problems from stationary points.

Consider fixed $\boldsymbol{\theta}$, for simplicity, we omit $\boldsymbol{\theta}$ in $f(\mathbf{x},\boldsymbol{\theta})$ and $\mathbf{g}(\mathbf{x},\boldsymbol{\theta})$ and write them as $f(\mathbf{x})$ and $\mathbf{g}(\mathbf{x})$, respectively.
Let $F:\mathcal{U}\times \mathcal{X}\rightarrow \mathbb{R}$ denote the approximate function of the objective function $f$, 
and let $\mathbf{G}:\mathcal{U}\times \mathcal{X}\rightarrow \mathbb{R}^{r}$ denote the approximate function of the inequality constraint function $\mathbf{g}$.
These approximate functions are basic components of SCA-based algorithms \cite{scutari2017parallel,scutari2018parallel}, which will be used in this paper to design methods and algorithms.
We make the following assumptions on $F$ and $\mathbf{G}$, respectively.\footnote{Please see \cite{scutari2017parallel} for detailed explanations of Assumptions~\ref{SCA-assump-surrogate-object} and \ref{SCA-assump-surrogate-constriant} and examples of qualified approximate functions.}

\begin{assumption}[{Approximate Function $F$ \cite[Assumption~2]{scutari2017parallel}}]\label{SCA-assump-surrogate-object}
	The function $F:\mathcal{U}\times \mathcal{X}\rightarrow \mathbb{R}$ is continuously~differentiable with respect to (w.r.t.) the first argument and such~that:
	\begin{enumerate}
		\item $F(\cdot; \mathbf{y})$ is uniformly strongly convex on $\mathcal{X}$ with constant $c_{F}>0$, i.e. for all $\mathbf{x}, \mathbf{x}'\in \mathcal{X}$, and $\mathbf{y}\in \mathcal{X}$,
		\begin{align*}
			(\mathbf{x}- \mathbf{x}')^{T}\left(\nabla_{\mathbf{x}}F(\mathbf{x}; \mathbf{y})- \nabla_{\mathbf{x}}F(\mathbf{x}'; \mathbf{y})\right)\geq c_{F}\left\rVert\mathbf{x}- \mathbf{x}'\right\rVert_{2}^{2};
		\end{align*}
		
		\item $\nabla_{\mathbf{x}}F(\mathbf{y}; \mathbf{y})=\nabla f(\mathbf{y})$ for all $\mathbf{y}\in \mathcal{X}$;
		
		\item $\nabla_{\mathbf{x}}F(\cdot; \cdot)$ is continuous on $\mathcal{U}\times \mathcal{X}$.
	\end{enumerate}
\end{assumption}

\begin{assumption}[{Approximate Function $\mathbf{G}$ \cite[Assumption~3]{scutari2017parallel}}]\label{SCA-assump-surrogate-constriant}
	The vector function $\mathbf{G}:\mathcal{U}\times \mathcal{X}\rightarrow \mathbb{R}^{r}$ satisfies:
	\begin{enumerate}
		\item $\mathbf{G}(\cdot; \mathbf{y})$ is convex on $\mathcal{X}$ for all $\mathbf{y}\in \mathcal{X}$;
		
		\item $\mathbf{G}(\mathbf{y}; \mathbf{y})= \mathbf{g}(\mathbf{y})$ for all $\mathbf{y}\in \mathcal{X}$;
		
		\item $\mathbf{G}(\mathbf{x}; \mathbf{y})- \mathbf{g}(\mathbf{x})\succeq \mathbf{0}$ for all $\mathbf{x}, \mathbf{y}\in \mathcal{X}$;
		
		\item $\mathbf{G}(\cdot; \cdot)$ is continuous on $\mathcal{X}\times \mathcal{X}$;
		
		\item $\nabla_{\mathbf{x}}\mathbf{G}(\mathbf{y}; \mathbf{y})=\nabla_{\mathbf{x}}\mathbf{g}(\mathbf{y})$ for all $\mathbf{y}\in \mathcal{X}$;
		
		\item $\nabla_{\mathbf{x}}\mathbf{G}(\cdot; \cdot)$ is continuous on $\mathcal{X}\times \mathcal{X}$.
	\end{enumerate}
\end{assumption}

\section{Primal Decomposition}\label{primal section}
In this section, we consider a nonconvex problem with coupling variables and propose a new primal decomposition method and a corresponding algorithm for this~problem, referred to as PD-M and \mbox{PD-A}, respectively.

\subsection{Problem Formulation}\label{primal section-problem}
Consider the following nonconvex problem:
\begin{align}
	\mathcal{P}^{pd}:\ \min_{\mathbf{x}, \mathbf{y}} \ &f(\mathbf{x},\mathbf{y})\triangleq f_0(\mathbf{y})+\sum_{i\in \mathcal{I}}f_i(\mathbf{x}_i,\mathbf{y}) \nonumber\\  
	s.t. \ 
	&\tilde{\mathbf{g}}_i(\mathbf{x}_i,\mathbf{y})\preceq \mathbf{0},\ i\in \mathcal{I},  \label{primal coup-ineq-cons}\\
	&\tilde{\mathbf{h}}_i(\mathbf{x}_i,\mathbf{y})= \mathbf{0},\ i\in \mathcal{I},  \label{primal coup-eq-cons}\\
	&\mathbf{g}_{i}(\mathbf{x}_{i})\preceq \mathbf{0},\ i\in \mathcal{I},  &&\label{primal decoup-ineq-cons-x}\\
	&\mathbf{x}_{i}\in \mathcal{X}_i,\ i\in \mathcal{I},  \label{primal decoup-cvx-cons-x}\\
	&\mathbf{g}_0(\mathbf{y})\preceq \mathbf{0}, \label{primal decoup-ineq-cons-y}\\ 
	&\mathbf{y}\in\mathcal{Y},  \label{primal decoup-cvx-cons-y}       
\end{align}        
where $\mathbf{x}\triangleq(\mathbf{x}_1,\mathbf{x}_2,\cdots,\mathbf{x}_{I})\in \mathbb{R}^{n}$ with $\mathbf{x}_i\in \mathbb{R}^{n_i}$, $i\in \mathcal{I}$, 
$\mathbf{y}\in \mathbb{R}^{n_0}$, 
$f_i:\mathcal{U}_i\times \mathcal{V}\rightarrow \mathbb{R}$,
$\tilde{\mathbf{g}}_i:\mathcal{U}_i\times \mathcal{V}\rightarrow \mathbb{R}^{\tilde{r}_i}$, 
$\tilde{\mathbf{h}}_i:\mathcal{U}_i\times \mathcal{V}\rightarrow \mathbb{R}^{\tilde{m}_i}$
$\mathbf{g}_{i}:\mathcal{U}_{i}\rightarrow \mathbb{R}^{r_i}$,
$i\in \mathcal{I}$,
$f_0:\mathcal{V}\rightarrow \mathbb{R}$,
and $\mathbf{g}_0:\mathcal{V}\rightarrow \mathbb{R}^{r_0}$. Note that $\tilde{\mathbf{h}}_i$ for all $i\in\mathcal{I}$ is generally not affine in $(\mathbf{x}_{i},\mathbf{y})$.\footnote{If $\tilde{\mathbf{h}}_{i}$ for all $i\in\mathcal{I}$ is affine in $(\mathbf{x}_{i},\mathbf{y})$, then the generally nonlinear equality constraints in (\ref{primal coup-eq-cons}) reduce to the linear equality constraints in (\ref{sp coup-eq-cons}).}

\begin{assumption}[Assumptions on Problem $\mathcal{P}^{pd}$]\label{primal assump-p}
1) For all $i\in \mathcal{I}$, $\mathcal{X}_i$ is a nonempty, closed, and convex set that belongs to the open set $\mathcal{U}_i\subseteq\mathbb{R}^{n_i}$; 
2) $\mathcal{Y}$ is a nonempty, closed, and convex set that belongs to the open set $\mathcal{V}\subseteq\mathbb{R}^{n_0}$; 
3) For all $i\in \mathcal{I}$, $f_i$ is continuously differentiable on $\mathcal{U}_i\times \mathcal{V}$,  
and $\nabla f_i$ is Lipschitz continuous on $\mathcal{X}_i\times \mathcal{Y}$; 
4) $f_0$ is continuously differentiable on $\mathcal{V}$,   
and $\nabla f_0$ is Lipschitz continuous on $\mathcal{Y}$; 
5) For all $i\in \mathcal{I}$, $\tilde{\mathbf{g}}_i$ is continuously differentiable on $\mathcal{U}_i\times \mathcal{V}$;
6) For all $i\in \mathcal{I}$, $\tilde{\mathbf{h}}_i$ is continuously differentiable on $\mathcal{U}_i\times \mathcal{V}$; 
7) For all $i\in \mathcal{I}$, $\mathbf{g}_{i}$ is continuously differentiable on $\mathcal{U}_{i}$;  
8) $\mathbf{g}_0$ is continuously differentiable on $\mathcal{V}$;  
9) $f$ is bounded below.
\end{assumption}

Once $\mathbf{y}$ is fixed, both the objective and constraints of Problem $\mathcal{P}^{pd}$ decouple in $\mathbf{x}_{i}$, $i\in\mathcal{I}$.
Thus, $\mathbf{y}$ is the coupling variable of Problem $\mathcal{P}^{pd}$.
Following the primal decomposition method initially proposed for convex problems \cite{palomar2006tutorial, boyd2007notes}, we fix $\mathbf{y}$ and define the (primal) subproblems as follows: for $i\in\mathcal{I}$,
\begin{align}
		\mathcal{P}_{sub, i}^{pd}:\ \min_{\mathbf{x}_i} \ &f_i(\mathbf{x}_i,\mathbf{y}) \nonumber\\  
		s.t. \
		&\tilde{\mathbf{g}}_i(\mathbf{x}_i,\mathbf{y})\preceq \mathbf{0},  \label{primal sub-coup-ineq-cons}\\
		&\tilde{\mathbf{h}}_i(\mathbf{x}_i,\mathbf{y})= \mathbf{0},  \label{primal sub-coup-eq-cons}\\
		&\mathbf{g}_{i}(\mathbf{x}_{i})\preceq \mathbf{0}, \label{primal sub-decoup-ineq-cons-x}\\
		&\mathbf{x}_{i}\in \mathcal{X}_i.  \label{primal sub-decoup-cvx-cons-x}       
	\end{align} 
Then, Problem $\mathcal{P}^{pd}$ is equivalent to the following master (primal) problem:
\begin{align}
		\mathcal{P}^{pd}_{mas}:\ \min_{\mathbf{y}} \ &f^*(\mathbf{y})\triangleq f_0(\mathbf{y})+\sum_{i\in \mathcal{I}}f_i(\mathbf{x}^*_i(\mathbf{y}), \mathbf{y})\nonumber\\
		s.t. \
		&\text{\eqref{primal decoup-ineq-cons-y}, \eqref{primal decoup-cvx-cons-y}},\nonumber
	\end{align} 
where $\mathbf{x}^*_i(\mathbf{y})$ represents an optimal point of the subproblem $\mathcal{P}_{sub, i}^{pd}$ for $\mathbf{y}\in \operatorname{dom} f^*\triangleq\{\mathbf{y}\in \mathcal{V}\ | \ 
	\tilde{\mathbf{g}}_i(\mathbf{x}_i,\mathbf{y})\preceq \mathbf{0},
	\tilde{\mathbf{h}}_i(\mathbf{x}_i,\mathbf{y})= \mathbf{0}, \mathbf{g}_{i}(\mathbf{x}_{i})\preceq \mathbf{0},
	\ \text{for some}\ \mathbf{x}_{i} \in \mathcal{X}_{i}, \ \text{for all} \ i\in\mathcal{I}
	\}$.
Note that $\mathbf{y}\in\operatorname{dom} f^*$ guarantees that for all $i\in\mathcal{I}$, the subproblem $\mathcal{P}_{sub, i}^{pd}$ is feasible, and hence $\mathbf{x}_{i}^{*}(\mathbf{y})$ exists and $f^*(\mathbf{y})<+\infty$.
Each subproblem $\mathcal{P}_{sub, i}^{pd}$ for $\mathbf{y}\in \operatorname{dom} f^*$ may have multiple optimal points but a unique optimal value, and hence $f^*(\mathbf{y})$ is unique.

Since Problem $\mathcal{P}^{pd}$ is nonconvex, all subproblems $\mathcal{P}_{sub, i}^{pd}$, $i\in\mathcal{I}$ 
are generally nonconvex (or cannot be shown to be convex).
Thus, it is typically challenging to obtain an optimal point $\mathbf{x}^*_i(\mathbf{y})$ of the subproblem $\mathcal{P}_{sub, i}^{pd}$ for all $i\in\mathcal{I}$ and the objective value $f^*(\mathbf{y})$ of the master problem $\mathcal{P}^{pd}_{mas}$ at $\mathbf{y}\in\operatorname{dom} f^*$. 
Besides, the nonuniqueness of a subproblem $\mathcal{P}_{sub, i}^{pd}$'s optimal Lagrange multipliers associated with the constraints in \eqref{primal sub-coup-ineq-cons} and \eqref{primal sub-coup-eq-cons}
and the nonconvexity of $f_{i}$, $i\in\mathcal{I}\cup\{0\}$ can lead to the nondifferentiability and nonconvexity of $f^*$, respectively.
Thus, gradient-based methods and subgradient-based methods (which are applicable only for convex problems) no longer work for the generally nonconvex master problem $\mathcal{P}^{pd}_{mas}$.
Consequently, the primal decomposition method (given by the subproblems $\mathcal{P}_{sub, i}^{pd}$, $i\in\mathcal{I}$ and the master problem $\mathcal{P}^{pd}_{mas}$) and the subgradient-based primal decomposition algorithm for convex problems in \cite{boyd2007notes,palomar2006tutorial} are not applicable to nonconvex Problem $\mathcal{P}^{pd}$.

\subsection{PD-M and Its Theoretical Analysis}\label{primal section-method}
To address the issues discussed above, we propose PD-M for Problem $\mathcal{P}^{pd}$. Specifically, we approximate the master problem $\mathcal{P}^{pd}_{mas}$ with the following problem:
\begin{equation*}\label{primal mas appro}
	\begin{aligned}
		\mathcal{P}^{pd, \dag}_{mas}:\ \min_{\mathbf{y}} \ &f^{\dag}(\mathbf{y})\triangleq f_0(\mathbf{y})+		
		f_{\mathcal{I}}^{\dagger}(\mathbf{y})
		\\ 
		s.t. \
		&\text{\eqref{primal decoup-ineq-cons-y}, \eqref{primal decoup-cvx-cons-y}},
	\end{aligned} 
\end{equation*}
where  $f_{\mathcal{I}}^{\dagger}(\mathbf{y})\triangleq\sum_{i\in \mathcal{I}} f_{i}^{\dagger}(\mathbf{y})$ with $f_{i}^{\dagger}(\mathbf{y})\triangleq f_i(\mathbf{x}^{\dag}_{i}(\mathbf{y}), \mathbf{y})$ and $\mathbf{x}^{\dag}_i(\mathbf{y})$ being a (selected) stationary point of the subproblem $\mathcal{P}_{sub, i}^{pd}$ for $\mathbf{y}\in\operatorname{dom}f^{\dagger}=\operatorname{dom}f^*$.
Note that each subproblem $\mathcal{P}_{sub, i}^{pd}$ for  $\mathbf{y}\in\operatorname{dom}f^{\dagger}$ may have multiple stationary points, and hence $f^{\dag}(\mathbf{y})$ may change with the selected stationary points, which is different from $f^*(\mathbf{y})$.
We also call Problem $\mathcal{P}^{pd, \dag}_{mas}$ the master problem.
The proposed PD-M for Problem $\mathcal{P}^{pd}$ is given by the subproblems $\mathcal{P}_{sub, i}^{pd}$, $i\in\mathcal{I}$ and the master problem $\mathcal{P}^{pd, \dag}_{mas}$.

Next, we verify the effectiveness of PD-M by analyzing the relationship between the stationary points of the master problem $\mathcal{P}^{pd, \dag}_{mas}$ and Problem $\mathcal{P}^{pd}$ under the following assumptions.

\begin{assumption}\label{primal assump-stationary-1}
	There exists an open set $\mathcal{N}\subseteq\operatorname{dom} f^{\dagger}$ and single-valued continuously differentiable functions $\mathbf{X}^{\dag}_i:\mathcal{N}\rightarrow \mathcal{U}_{i}$, $\widetilde{\mathbf{M}}_i:\mathcal{N}\rightarrow \mathbb{R}_+^{\tilde{r}_i}$,
	$\tilde{\boldsymbol{\Lambda}}_i:\mathcal{N}\rightarrow \mathbb{R}^{\tilde{m}_i}$,
	$\mathbf{M}_i:\mathcal{N}\rightarrow \mathbb{R}_+^{r_i}$,  $i\in\mathcal{I}$ such that for all $i\in\mathcal{I}$ and $\mathbf{y}\in \mathcal{N}$, $\mathbf{X}^{\dag}_i(\mathbf{y})\in\operatorname{int}(\mathcal{X}_{i})$ is a stationary point of the subproblem $\mathcal{P}_{sub, i}^{pd}$, and $\widetilde{\mathbf{M}}_i(\mathbf{y})$, $\tilde{\boldsymbol{\Lambda}}_i(\mathbf{y})$,  and $\mathbf{M}_i(\mathbf{y})$ are the Lagrange multipliers associated with the constraints in \eqref{primal sub-coup-ineq-cons}, \eqref{primal sub-coup-eq-cons}, and \eqref{primal sub-decoup-ineq-cons-x}, respectively.
\end{assumption}

If for all $i\in\mathcal{I}$ and $\mathbf{y}\in\operatorname{dom}f^{\dagger}$, we can obtain a closed-form stationary point $\mathbf{x}_{i}^{\dagger}(\mathbf{y})$ of the subproblem $\mathcal{P}^{pd}_{sub,i}$ and the related closed-form Lagrange multipliers $\tilde{\boldsymbol{\mu}}_i(\mathbf{y})$,  $\tilde{\boldsymbol{\lambda}}_i(\mathbf{y})$, and $\boldsymbol{\mu}_i(\mathbf{y})$, then it is relatively easy to check Assumption~\ref{primal assump-stationary-1}.
If not, we provide an alternative assumption in the following, for which it is sufficient to obtain a numerical stationary point of Problem $\mathcal{P}^{pd}$ and its numerical Lagrange multipliers.

\begin{myassumptionprimal}\label{primal assump-stationary-2}
1) $f_i$, $\tilde{\mathbf{g}}_i$, $\tilde{\mathbf{h}}_i$, and $\mathbf{g}_i$, $i\in\mathcal{I}$ are twice continuously differentiable;
2) There exists a stationary point $(\mathbf{x}^{\ddagger},\mathbf{y}^{\ddagger})\in \prod_{i\in\mathcal{I}}\operatorname{int}(\mathcal{X}_{i})\times \mathcal{Y}$ of Problem $\mathcal{P}^{pd}$ with Lagrange multipliers $\tilde{\boldsymbol{\mu}}^{\ddagger}_i$,   $\tilde{\boldsymbol{\lambda}}^{\ddagger}_i$, $\boldsymbol{\mu}^{\ddagger}_i$, $i\in\mathcal{I}$ associated with the constraints in \eqref{primal coup-ineq-cons}, \eqref{primal coup-eq-cons}, and \eqref{primal decoup-ineq-cons-x}, respectively, such that for all $i\in\mathcal{I}$, the gradient of the subproblem $\mathcal{P}^{pd}_{sub,i}$'s KKT function $\mathbf{k}_{i}$ for $\mathbf{y}\in\operatorname{dom}f^{\dagger}$ 
w.r.t. $(\mathbf{x}_i,\tilde{\boldsymbol{\mu}}_i, \tilde{\boldsymbol{\lambda}}_i, \boldsymbol{\mu}_i)$ at $(\mathbf{x}_i^{\ddagger},  \tilde{\boldsymbol{\mu}}^{\ddagger}_i,   \tilde{\boldsymbol{\lambda}}^{\ddagger}_i,  \boldsymbol{\mu}^{\ddagger}_i, \mathbf{y}^{\ddagger})$, i.e., $\nabla_{(\mathbf{x}_i,\tilde{\boldsymbol{\mu}}_i,  \tilde{\boldsymbol{\lambda}}_i, \boldsymbol{\mu}_i)} \mathbf{k}_i(\mathbf{x}_i^{\ddagger},  \tilde{\boldsymbol{\mu}}^{\ddagger}_i,    \tilde{\boldsymbol{\lambda}}^{\ddagger}_i, \boldsymbol{\mu}^{\ddagger}_i, \mathbf{y}^{\ddagger})$, is invertible.
\end{myassumptionprimal}

In practice, we usually do not explicitly verify Assumption~\ref{primal assump-stationary-2}.2 and observe the effectiveness of PD-M based on numerical results.\footnote{A similar situation occurs in the local convergence analysis of Newton’s method \cite[Section 1.4]{bertsekas2016nonlinear}.} 
Now, we present the relationship between Assumptions~\ref{primal assump-stationary-1} and \ref{primal assump-stationary-2} below.

\begin{lemma}\label{primal lem1}
	Suppose that Assumptions~\ref{primal assump-p} and \ref{primal assump-stationary-2} are satisfied.
	Then there exists a neighborhood $\mathcal{N}_{\mathbf{y}^{\ddagger}}$ of $\mathbf{y}^{\ddagger}$ and neighborhoods $\mathcal{N}_{\mathbf{x}_{i}^{\ddagger}}$, $\mathcal{N}_{\tilde{\boldsymbol{\mu}}_{i}^{\ddagger}}$, 
	$\mathcal{N}_{\tilde{\boldsymbol{\lambda}}_{i}^{\ddagger}}$,
	$\mathcal{N}_{\boldsymbol{\mu}_{i}^{\ddagger}}$,  $i\in\mathcal{I}$ of $\mathbf{x}_{i}^{\ddagger}$, $\tilde{\boldsymbol{\mu}}_{i}^{\ddagger}$,  
	$\tilde{\boldsymbol{\lambda}}_{i}^{\ddagger}$,
	$\boldsymbol{\mu}_{i}^{\ddagger}$,  $i\in\mathcal{I}$, respectively, such that
	\begin{enumerate}
		\item Assumption~\ref{primal assump-stationary-1} holds with $\mathcal{N}$ being $\mathcal{N}_{\mathbf{y}^{\ddagger}}$;
		
		\item For all $i\in\mathcal{I}$ and $\mathbf{y}\in \mathcal{N}_{\mathbf{y}^{\ddagger}}$, 
		the subproblem $\mathcal{P}^{pd}_{sub,i}$ has a unique stationary point in $\mathcal{N}_{\mathbf{x}_{i}^{\ddagger}}$ that has Lagrange multipliers associated with the constraints in \eqref{primal sub-coup-ineq-cons}, \eqref{primal sub-coup-eq-cons}, and \eqref{primal sub-decoup-ineq-cons-x}  in $\mathcal{N}_{\tilde{\boldsymbol{\mu}}_{i}^{\ddagger}}$, 
		$\mathcal{N}_{\tilde{\boldsymbol{\lambda}}_{i}^{\ddagger}}$,
		and $\mathcal{N}_{\boldsymbol{\mu}_{i}^{\ddagger}}$, respectively, and the Lagrange multipliers are unique.
	\end{enumerate}
\end{lemma}
\begin{IEEEproof}
	See Appendix~\ref{appendix primal lem1}.
\end{IEEEproof}

Lemma~\ref{primal lem1}.1 shows that Assumption~\ref{primal assump-stationary-2} together with Assumption~\ref{primal assump-p} implies Assumption~\ref{primal assump-stationary-1}.
Lemma~\ref{primal lem1}.2 indicates the uniqueness of each subproblem $\mathcal{P}^{pd}_{sub,i}$'s stationary point and its Lagrange multipliers, implying that for all $i\in\mathcal{I}$ and $\mathbf{y}\in \mathcal{N}_{\mathbf{y}^{\ddagger}}$, Linear Independence Constraint Qualification (LICQ) holds at the stationary point $\mathbf{X}_{i}^{\dagger}(\mathbf{y})$ of the subproblem $\mathcal{P}^{pd}_{sub,i}$ with $\mathbf{X}_{i}^{\dagger}$ specified in Lemma~\ref{primal lem1}.1.
Under LICQ, the KKT system of the stationary point $\mathbf{X}_{i}^{\dagger}(\mathbf{y})\in\operatorname{int}(\mathcal{X}_{i})$ has a unique solution, providing an efficient method for calculating Lagrange multipliers of a given stationary point.

\begin{theorem}\label{primal thm1}
	Suppose that Assumptions~\ref{primal assump-p} and \ref{primal assump-stationary-1} (or \ref{primal assump-stationary-2}) are satisfied.
	Let $\mathcal{N}$ (or $\mathcal{N}_{\mathbf{y}^{\ddagger}}$) denote the open set specified in Assumption~\ref{primal assump-stationary-1} (or Lemma~\ref{primal lem1}.1).
	Then the following results~hold.
	\begin{enumerate}
		\item  $f^{\dag}_{\mathcal{I}}$, where $\mathbf{x}^{\dag}_{i}(\mathbf{y})=\mathbf{X}^{\dag}_{i}(\mathbf{y})$ with $\mathbf{X}_i^{\dag}$ specified in Assumption \ref{primal assump-stationary-1} (or Lemma~\ref{primal lem1}.1), is continuously differentiable on $\mathcal{N}$ (or $\mathcal{N}_{\mathbf{y}^{\ddagger}}$), and for all $\mathbf{y}\in \mathcal{N}$ (or $\mathcal{N}_{\mathbf{y}^{\ddagger}}$), the gradient $\nabla f^{\dagger}_{\mathcal{I}}(\mathbf{y})=\sum_{i\in\mathcal{I}} \nabla f_{i}^{\dagger}(\mathbf{y})$,  
		where 
		$\nabla f_{i}^{\dagger}(\mathbf{y})$ is given by 
		\begin{equation}\label{primal mas-obj-grad-y-i}
		\hspace{-0.4cm}	\begin{aligned}
				\nabla f_{i}^{\dagger}(\mathbf{y})= &\nabla_{\mathbf{y}}f_i(\mathbf{X}^{\dagger}_i(\mathbf{y}), \mathbf{y})
				+ \nabla_{\mathbf{y}}\tilde{\mathbf{g}}_i(\mathbf{X}^{\dagger}_i(\mathbf{y}), \mathbf{y})\widetilde{\mathbf{M}}_{i}(\mathbf{y})\\
				&+ \nabla_{\mathbf{y}}\tilde{\mathbf{h}}_i(\mathbf{X}^{\dagger}_i(\mathbf{y}), \mathbf{y})\tilde{\boldsymbol{\Lambda}}_{i}(\mathbf{y}),
			\end{aligned}
		\end{equation}
		with $\widetilde{\mathbf{M}}_i$ and $\tilde{\boldsymbol{\Lambda}}_i$ specified in Assumption~\ref{primal assump-stationary-1} (or Lemma~\ref{primal lem1}.1);

		\item If $\mathbf{y}^{\dag}\in\mathcal{N}$ (or $\mathcal{N}_{\mathbf{y}^{\ddagger}}$) is a stationary point of the master problem $\mathcal{P}^{pd, \dag}_{mas}$, then 	$(\mathbf{X}^{\dag}(\mathbf{y}^{\dag}), \mathbf{y}^{\dag})$ is a stationary point of Problem $\mathcal{P}^{pd}$, where $\mathbf{X}^{\dag}(\mathbf{y}^{\dag})=(\mathbf{X}_1^{\dag}(\mathbf{y}^{\dag}), \cdots, \mathbf{X}^{\dag}_I(\mathbf{y}^{\dag}))$ with $\mathbf{X}^{\dag}_i(\mathbf{y}^{\dag})$ being the stationary point of the subproblem $\mathcal{P}^{pd}_{sub, i}$ and $\mathbf{X}^{\dag}_i$ specified in Assumption \ref{primal assump-stationary-1} (or Lemma~\ref{primal lem1}.1). 
	\end{enumerate}
\end{theorem}
\begin{IEEEproof}
	See Appendix~\ref{appendix primal thm1}.
\end{IEEEproof}

Theorem~\ref{primal thm1}.1 provides a method for calculating the gradient $\nabla f_{\mathcal{I}}^{\dag}(\mathbf{y})$, which does not require the analytical expression of $f_{\mathcal{I}}^{\dag}$. 
In addition, Theorem~\ref{primal thm1}.1 also implies that $f^{\dag}$ is continuously differentiable on $\mathcal{N}$ (or $\mathcal{N}_{\mathbf{y}^{\ddagger}}$), which is a necessary condition for gradient-based algorithms that can be used to solve the master problem $\mathcal{P}^{pd, \dag}_{mas}$.
Theorem~\ref{primal thm1}.2 indicates that 
a stationary point of Problem $\mathcal{P}^{pd}$ can be obtained from a stationary point $\mathbf{y}^{\dag}$ of the master problem $\mathcal{P}^{pd, \dag}_{mas}$ and stationary points $\mathbf{x}^{\dag}_i(\mathbf{y}^{\dag})$, $i\in\mathcal{I}$ of the subproblems $\mathcal{P}_{sub, i}^{pd}$, $i\in\mathcal{I}$, showcasing the effectiveness of PD-M in producing stationary points of Problem $\mathcal{P}^{pd}$.

\subsection{PD-A and Its Local Convergence Analysis}\label{primal section-algorithm}
By Theorem~\ref{primal thm1}, the key to designing a corresponding primal decomposition algorithm for Problem $\mathcal{P}^{pd}$ is to design a master algorithm for the master problem $\mathcal{P}^{pd,\dagger}_{mas}$. 
Generally, the closed-form expression for $f_{\mathcal{I}}^{\dag}$ (i.e., the component of the master problem $\mathcal{P}^{pd,\dagger}_{mas}$'s objective function $f^{\dag}$) is hard to obtain, and only local information about $f_{\mathcal{I}}^{\dag}$ (e.g., its value and gradient at a point) is accessible. 
Thus, we need to iteratively solve the subproblems $\mathcal{P}^{pd}_{sub, i}$, $i\in\mathcal{I}$ and update the coupling variable $\mathbf{y}$ based on the  local information accessible from $f_{\mathcal{I}}^{\dag}$.

Based on this, we propose PD-A for Problem $\mathcal{P}^{pd}$, which uses the SCA-based algorithm \cite{scutari2017parallel} to
solve the master problem $\mathcal{P}^{pd,\dagger}_{mas}$.
Specifically, at iteration $k$, for fixed $\mathbf{y}^{(k)}$ (which is obtained at iteration $k-1$), we can divide the algorithm process into two parts.
In the first part, we focus on solving the subproblems $\mathcal{P}_{sub, i}^{pd}$, $i\in\mathcal{I}$ separately. 
First, we obtain an arbitrary stationary point $\mathbf{x}^{\dag}_i(\mathbf{y}^{(k)})$ of the subproblem $\mathcal{P}_{sub, i}^{pd}$ using algorithms for generally nonconvex problems (e.g., Majorization-Minimization (MM) algorithms \cite{sun2017mm}, SCA algorithms \cite{scutari2018parallel}, and Sequential Quadratic Programming (SQP)\cite{gill2011sequential}).\footnote{For a nonconvex problem, once a stationary point is obtained via an iterative algorithm, its performance can be assessed to determine whether it satisfies practical needs. To further enhance the performance, one can find multiple stationary points by running the iterative algorithm multiple times with different initial points, chosen randomly or based on some heuristics, and select the best stationary point among the obtained ones.} 
Then, we can obtain the corresponding Lagrange multipliers $\tilde{\boldsymbol{\mu}}_i(\mathbf{y}^{(k)})$,  $\tilde{\boldsymbol{\lambda}}_i(\mathbf{y}^{(k)})$, and $\boldsymbol{\mu}_i(\mathbf{y}^{(k)})$ by solving the KKT system or KKT conditions (see the discussion below Definition~\ref{def stationary point}).
Note that $\boldsymbol{\mu}_i(\mathbf{y}^{(k)})$ is not utilized in the followed algorithm process but is essential for the local convergence analysis. 
Next, we compute $\nabla f_{i}^{\dag}(\mathbf{y}^{(k)})$ according to \eqref{primal mas-obj-grad-y-i}~where $\mathbf{X}^{\dagger}_i(\mathbf{y}^{(k)})\!=\!\mathbf{x}_{i}^{\dag}(\mathbf{y}^{(k)})$, $\widetilde{\mathbf{M}}_{i}(\mathbf{y}^{(k)})\!=\!\tilde{\boldsymbol{\mu}}_i(\mathbf{y}^{(k)})$, and $\tilde{\boldsymbol{\Lambda}}_{i}(\mathbf{y}^{(k)})\!=\!\tilde{\boldsymbol{\lambda}}_i(\mathbf{y}^{(k)})$.

In the second part, we focus on solving the master problem $\mathcal{P}^{pd, \dag}_{mas}$ by the SCA-based algorithm \cite{scutari2017parallel}. 
First, we compute $\nabla f_{\mathcal{I}}^{\dag}(\mathbf{y}^{(k)})=\sum_{i\in \mathcal{I}}\nabla f_{i}^{\dag}(\mathbf{y}^{(k)})$ 
and choose the following approximate function of $f^{\dag}$ at $\mathbf{y}^{(k)}$:
\begin{align*}\label{primal mas-appro-appro-object}
	F^{\dag}(\mathbf{y}; \mathbf{y}^{(k)})\triangleq F_{0}(\mathbf{y}; \mathbf{y}^{(k)}) + F_{\mathcal{I}}^{\dag}(\mathbf{y}; \mathbf{y}^{(k)}),
\end{align*}
where $F_{0}(\mathbf{y}; \mathbf{y}^{(k)})$ is any approximate function of $f_0$ at $\mathbf{y}^{(k)}$ satisfying Assumption~\ref{SCA-assump-surrogate-object}, and $F_{\mathcal{I}}^{\dag}(\mathbf{y}; \mathbf{y}^{(k)})$ is the approximate function of $f^{\dagger}_{\mathcal{I}}$ at $\mathbf{y}^{(k)}$ (i.e., the second-order Taylor approximation of $f_{\mathcal{I}}^{\dag}$ near $\mathbf{y}^{(k)}$ \cite{scutari2017parallel}), given by:
\begin{align*}
	F_{\mathcal{I}}^{\dag}(\mathbf{y}; \mathbf{y}^{(k)})\triangleq\frac{\tau}{2}\left\Vert\mathbf{y}-\mathbf{y}^{(k)}\right\Vert_2^2 + 
	\nabla f_{\mathcal{I}}^{\dag}(\mathbf{y}^{(k)})^{T}(\mathbf{y}-\mathbf{y}^{(k)}),
\end{align*}
where $\tau\!\geq\!0$.\footnote{
The strong convexity of $F_{0}$ in Assumption~\ref{SCA-assump-surrogate-object}.1 can reduce to convexity when $\tau>0$, since the strong convexity of $F^{\dagger}$ is already guaranteed by the strong convexity of $F^\dagger_{\mathcal{I}}$.}
Notably, $F^{\dagger}$ satisfies Assumption~\ref{SCA-assump-surrogate-object}.
In addition, we choose an approximation of $\mathbf{g}_0$ at $\mathbf{y}^{(k)}$, denoted by $\mathbf{G}_0(\mathbf{y}; \mathbf{y}^{(k)})$, which satisfies Assumption~\ref{SCA-assump-surrogate-constriant}.\footnote{The approximate function $F_{\mathcal{I}}^{\dag}$ of $f_{\mathcal{I}}^{\dag}$ is specified, 
whereas the approximate functions $F_{0}$ of $f_{0}$ and $\mathbf{G}_{0}$ of $\mathbf{g}_{0}$ can be chosen more flexibly.
The reason is that only $f^{\dag}_{\mathcal{I}}$'s local information (e.g., its value and gradient at a point) is known, whereas the expressions of $f_{0}$ and $\mathbf{g}_{0}$ are available.} 
The corresponding approximate problem of the master problem $\mathcal{P}^{pd, \dag}_{mas}$~is:
\begin{align}
	\mathcal{P}^{pd, \dag, (k)}_{mas}:\ \min_{\mathbf{y}} \ &F^{\dag}(\mathbf{y}; \mathbf{y}^{(k)})\nonumber\\
	s.t. \ &\mathbf{G}_0(\mathbf{y}; \mathbf{y}^{(k)})\preceq \mathbf{0},\\
	\ &\text{\eqref{primal decoup-cvx-cons-y}}.\nonumber               
\end{align}
This problem is strongly convex since $F^{\dagger}$ and $\mathbf{G}_{0}$ satisfy Assumptions~\ref{SCA-assump-surrogate-object}.1 and~\ref{SCA-assump-surrogate-constriant}.1, respectively.
Then, we solve Problem $\mathcal{P}^{pd, \dag, (k)}_{mas}$ to obtain its unique optimal point $\mathbf{y}^*(\mathbf{y}^{(k)})$ using algorithms for convex problems (e.g., interior-point methods). 
Next, the coupling variable $\mathbf{y}$ is updated according to:
\begin{equation}\label{primal y-smooth}
	\mathbf{y}^{(k+1)}=\mathbf{y}^{(k)}+\gamma^{(k)}(\mathbf{y}^*(\mathbf{y}^{(k)}) - \mathbf{y}^{(k)}),
\end{equation} 
where $\gamma^{(k)}$ is a step size satisfying \cite{scutari2017parallel}:
\begin{align}\label{sca step size}
	\gamma^{(k)}\in(0,1],\ \lim_{k\rightarrow\infty}\gamma^{(k)} = 0,\ \text{and}\ \sum_{k=0}^{\infty}\gamma^{(k)}=\infty.
\end{align}

\begin{algorithm}[t]
	\caption{PD-A}\label{primal alg:alg1}
	\begin{algorithmic}[1]\label{alg1}\small
		\STATE \textbf{initialization}: Set $k=0$ and choose any feasible point $\mathbf{y}^{(0)}$, $\tau \geq 0$, and  $\{\gamma^{(k)}\}_{k\in\mathbb{N}}\subseteq(0,1]$ satisfying (\ref{sca step size}).
		\REPEAT
		\FOR{all $i\in\mathcal{I}$}
		\STATE Get an arbitrary stationary point $\mathbf{x}^{\dag}_i(\mathbf{y}^{(k)})$ of the generally nonconvex subproblem $\mathcal{P}_{sub, i}^{pd}$.
		\STATE Obtain the corresponding Lagrange multipliers $\tilde{\boldsymbol{\mu}}_i(\mathbf{y}^{(k)})$, 
		$\tilde{\boldsymbol{\lambda}}_i(\mathbf{y}^{(k)})$, and $\boldsymbol{\mu}_i(\mathbf{y}^{(k)})$.
		\STATE Compute $\nabla f_{i}^{\dag}(\mathbf{y}^{(k)})$ according to (\ref{primal mas-obj-grad-y-i}).
		\ENDFOR
		\STATE Compute $\nabla f_{\mathcal{I}}^{\dag}(\mathbf{y}^{(k)})=\sum_{i\in \mathcal{I}}\nabla f_{i}^{\dag}(\mathbf{y}^{(k)})$. 
		\STATE Obtain the unique optimal point $\mathbf{y}^*(\mathbf{y}^{(k)})$ of the convex approximate problem $\mathcal{P}^{pd, \dag, (k)}_{mas}$.
		\STATE Update $\mathbf{y}$ according to (\ref{primal y-smooth}).
		\STATE Set $k \gets k+1$.
		\UNTIL{Some termination criterion is met.} 
	\end{algorithmic}
\end{algorithm}

The detailed procedure is summarized in Algorithm~\ref{primal alg:alg1}.\footnote{In Steps~6, 8, and 10, parallel computations can be applied to conduct matrix (vector) multiplications and additions.}
Note that it is possible that $\mathbf{y}^{(k)}\notin\operatorname{dom} f^{\dagger}$ for some $k$, i.e., there exist infeasible subproblems for some $k$.\footnote{This issue seldom happens in Algorithm~\ref{primal alg:alg1} for Problem $\mathcal{P}^{pd}$ or in the standard primal decomposition algorithm for convex problems \cite{boyd2007notes}. In \cite{boyd2007notes}, the issue is also addressed by additional procedure, which is not presented in the formal algorithm description for ease of exposition.} 
In this case, we skip Steps~3-10 and update $\mathbf{y}^{(k)}$ by $\mathbf{y}^{(k+1)}\!=\!\mathbf{y}^{(k-1)}\!+\!\delta^{(k)}(\mathbf{y}^{(k)}\!-\!\mathbf{y}^{(k-1)})$ with $\delta^{(k)}\!\in\!(0,1]$ such that  $\mathbf{y}^{(k+1)}\!\in\!\operatorname{dom} f^{\dagger}$.

To prove the local convergence of Algorithm~\ref{primal alg:alg1} under Assumptions~\ref{primal assump-p} and \ref{primal assump-stationary-2}, we make the following assumption on $\nabla f^{\dag}(\mathbf{y})$ which exists by Theorem~\ref{primal thm1}.1.

\begin{assumption}\label{primal assump-mas-appro-object}
	$\nabla f^{\dag}(\mathbf{y})$ is Lipschitz continuous on a neighborhood of $\mathbf{y}^{\ddagger}$ with $\mathbf{y}^{\ddagger}$ specified in Assumption~\ref{primal assump-stationary-2}.2.\footnote{
			If functions $f_i$, $\tilde{\mathbf{g}}_i$, $\tilde{\mathbf{h}}_i$, $\mathbf{g}_i$, $i\in\mathcal{I}$ are three times continuously differentiable, 
			then Assumption~\ref{primal assump-mas-appro-object} can be concluded from Assumptions~\ref{primal assump-p}~and~\ref{primal assump-stationary-2}.}
\end{assumption}

Then, based on Theorem~\ref{primal thm1} and \cite[Theorem~2]{scutari2017parallel}, we can show the local convergence of Algorithm~\ref{primal alg:alg1}.

\begin{theorem}\label{primal thm2}
	Suppose that Assumptions~\ref{regularity}-\ref{primal assump-p}, \ref{primal assump-stationary-2}, and \ref{primal assump-mas-appro-object} are satisfied.
	The points $\mathbf{y}^{\ddagger}$, $\mathbf{x}_{i}^{\ddagger}$, $\tilde{\boldsymbol{\mu}}_{i}^{\ddagger}$,  
	$\tilde{\boldsymbol{\lambda}}_{i}^{\ddagger}$, $\boldsymbol{\mu}_{i}^{\ddagger}$,  $i\in\mathcal{I}$ are specified in Assumption~\ref{primal assump-stationary-2}.2.
	The sequences $\{\mathbf{y}^{(k)}\}_{k\in\mathbb{N}}$, 
	$\{\mathbf{x}^{\dag}_i(\mathbf{y}^{(k)})\}_{k\in\mathbb{N}}$,
	$\{\tilde{\boldsymbol{\mu}}_i(\mathbf{y}^{(k)})\}_{k\in\mathbb{N}}$,  
	$\{\tilde{\boldsymbol{\lambda}}_i(\mathbf{y}^{(k)})\}_{k\in\mathbb{N}}$,
	$\{\boldsymbol{\mu}_i(\mathbf{y}^{(k)})\}_{k\in\mathbb{N}}$,  $i\in\mathcal{I}$,
	are generated by Algorithm~\ref{primal alg:alg1}.
	Then, there exists a neighborhood $\mathcal{N}_{\mathbf{y}^{\ddagger}}$ of $\mathbf{y}^{\ddagger}$ and neighborhoods $\mathcal{N}_{\mathbf{x}_{i}^{\ddagger}}$, $\mathcal{N}_{\tilde{\boldsymbol{\mu}}_{i}^{\ddagger}}$, 
	$\mathcal{N}_{\tilde{\boldsymbol{\lambda}}_{i}^{\ddagger}}$,
	$\mathcal{N}_{\boldsymbol{\mu}_{i}^{\ddagger}}$,  $i\in\mathcal{I}$ of $\mathbf{x}_{i}^{\ddagger}$, $\tilde{\boldsymbol{\mu}}_{i}^{\ddagger}$, 
	$\tilde{\boldsymbol{\lambda}}_{i}^{\ddagger}$, $\boldsymbol{\mu}_{i}^{\ddagger}$,  $i\in\mathcal{I}$ such that if $\{\mathbf{y}^{(k)}\}_{k\in\mathbb{N}}\subseteq\mathcal{N}_{\mathbf{y}^{\ddagger}}$,  
	$\{\mathbf{x}^{\dag}_i(\mathbf{y}^{(k)})\}_{k\in\mathbb{N}}\subseteq\mathcal{N}_{\mathbf{x}_{i}^{\ddagger}}$,
	$\{\tilde{\boldsymbol{\mu}}_i(\mathbf{y}^{(k)})\}_{k\in\mathbb{N}}\subseteq\mathcal{N}_{\tilde{\boldsymbol{\mu}}_{i}^{\ddagger}}$, $\{\tilde{\boldsymbol{\lambda}}_i(\mathbf{y}^{(k)})\}_{k\in\mathbb{N}}\subseteq\mathcal{N}_{\tilde{\boldsymbol{\lambda}}_{i}^{\ddagger}}$,  $\{\boldsymbol{\mu}_i(\mathbf{y}^{(k)})\}_{k\in\mathbb{N}}\subseteq\mathcal{N}_{\boldsymbol{\mu}_{i}^{\ddagger}}$,  $i\in\mathcal{I}$,
	then the following results hold:
	\begin{enumerate}
		\item There exists at least one convergent subsequence $\{\mathbf{y}^{(k)}\}_{k\in\mathcal{K}}$ with $\mathcal{K}\subseteq\mathbb{N}$  whose limit point is a stationary point of the master problem $\mathcal{P}^{pd, \dag}_{mas}$.\footnote{We can also show the asymptotic behaviour of the norm between successively generated points, i.e.,
		$\lim_{k\rightarrow \infty} \|\mathbf{y}^{(k+1)}-\mathbf{y}^{(k)}\|_{2}=0$.}
		If, in addition, $\mathcal{Y}$ is compact, $\mathbf{G}_0(\cdot; \cdot)$ is Lipschitz continuous on $\mathcal{Y}\times \mathcal{Y}$, and 
		$\nabla_{\mathbf{y}}F_0(\cdot;\mathbf{z})$ and $\nabla_{\mathbf{y}}F_0(\mathbf{y};\cdot)$ are  uniformly Lipschitz continuous on $\mathcal{Y}$,
		then the limit point of any convergent~subsequence is a stationary point of the master problem~$\mathcal{P}^{pd, \dag}_{mas}$.\footnote{The compactness condition is usually satisfied in practice since the constraint sets of most practical problems are bounded. Besides, examples of $\mathbf{G}_{0}$ satisfying the Lipschitz continuity condition and examples of $F_{0}$ whose gradients $\nabla_{\mathbf{y}}F_0(\cdot;\mathbf{z})$ and $\nabla_{\mathbf{y}}F_0(\mathbf{y};\cdot)$ satisfy the uniformly Lipschitz continuity condition can be found in \cite{scutari2017parallel}.}
		
		\item For all $i\in\mathcal{I}$, the subsequence 
		$\{\mathbf{x}^{\dag}_i(\mathbf{y}^{(k)})\}_{k\in\mathcal{K}}$ with $\mathcal{K}$ specified in 1) converges, and  $(\mathbf{x}^{(\infty)},\mathbf{y}^{(\infty)})$ is a stationary point of Problem $\mathcal{P}^{pd}$, where $\mathbf{x}^{(\infty)}=(\mathbf{x}^{(\infty)}_{1},\cdots, \mathbf{x}^{(\infty)}_{I})$, $\mathbf{x}^{(\infty)}_{i}\in \operatorname{int}(\mathcal{X}_{i})$ is the limit point of $\{\mathbf{x}^{\dag}_i(\mathbf{y}^{(k)})\}_{k\in\mathcal{K}}$, and $\mathbf{y}^{(\infty)}$ is the limit point of $\{\mathbf{y}^{(k)}\}_{k\in\mathcal{K}}$.
	\end{enumerate}	
\end{theorem}
\begin{IEEEproof}
	See Appendix~\ref{primal proof-thm2}.
\end{IEEEproof}

\section{Successive Primal Decomposition}\label{sp section}
In this section, we consider a nonconvex problem with coupling variables and propose a successive primal decomposition method and a corresponding algorithm for this problem, referred to as SPD-M and \mbox{SPD-A}, respectively.

\subsection{Problem Formulation}
Consider the following nonconvex problem:
\begin{align}
	\mathcal{P}^{spd}:\ \min_{\mathbf{x}, \mathbf{y}} \ &f(\mathbf{x}, \mathbf{y})\triangleq f_0(\mathbf{y})+ \sum_{i\in\mathcal{I}}f_i(\mathbf{x}_i, \mathbf{y})\nonumber\\
	s.t. \ 
	&\mathbf{A}_{i,\mathbf{x}}\mathbf{x}_i+\mathbf{A}_{i,\mathbf{y}}\mathbf{y}+ \mathbf{b}_{i}=\mathbf{0},
	\ i\in \mathcal{I}, \label{sp coup-eq-cons}\\
	&\text{\eqref{primal coup-ineq-cons}, \eqref{primal decoup-ineq-cons-x}-\eqref{primal decoup-cvx-cons-y}},      
	\nonumber
\end{align}
where $\mathbf{x}$, $\mathbf{y}$, $f_i$, $\tilde{\mathbf{g}}_i$, $\mathbf{g}_{i}$, $i\in\mathcal{I}$, $f_0$, and $\mathbf{g}_{0}$ are the same as in Problem $\mathcal{P}^{pd}$, 
$\mathbf{A}_{i,\mathbf{x}}\in\mathbb{R}^{\tilde{m}_{i}\times n_i}$, $\mathbf{A}_{i,\mathbf{y}}\in\mathbb{R}^{\tilde{m}_{i}\times n_0}$,  $\mathbf{b}_i\in\mathbb{R}^{\tilde{m}_{i}}$, $i\in\mathcal{I}$.
In contrast, the coupling linear equality constraints of Problem $\mathcal{P}^{spd}$ in \eqref{sp coup-eq-cons} are less general than the coupling (generally) nonlinear equality constraints of Problem $\mathcal{P}^{pd}$ in \eqref{primal coup-eq-cons}.
Thus, Problem $\mathcal{P}^{spd}$ is a special case of Problem $\mathcal{P}^{pd}$. 
Accordingly, by removing 6) in Assumption~\ref{primal assump-p}, we have the following assumptions about Problem $\mathcal{P}^{spd}$.
\begin{assumption}[Assumptions on Problem $\mathcal{P}^{spd}$]\label{sp p}
	Assumptions~\ref{primal assump-p}.1-\ref{primal assump-p}.5 and \ref{primal assump-p}.7-\ref{primal assump-p}.9 hold true.
\end{assumption}

As discussed in Section~\ref{primal section-problem}, the basic primal decomposition method and algorithm for convex problems in \cite{boyd2007notes,palomar2006tutorial} are also not applicable to nonconvex Problem $\mathcal{P}^{spd}$.

\subsection{SPD-M and Its Theoretical Analysis}\label{sp section-method}
Now, we present SPD-M.
Specifically, first, we approximate nonconvex Problem $\mathcal{P}^{spd}$ with a sequence of~successively~refined convex approximate problems $\mathcal{P}^{spd,(k)}, k\in\mathbb{N}$.
Each convex problem is obtained by approximating the objective function and constraints of Problem $\mathcal{P}^{spd}$ at $(\mathbf{x}^{(k)}, \mathbf{y}^{(k)})$:
	\begin{align}
			\mathcal{P}^{spd,(k)}:\ 
			\min_{\mathbf{x}, \mathbf{y}} \ &F(\mathbf{x}, \mathbf{y}; \mathbf{x}^{(k)}, \mathbf{y}^{(k)})\nonumber\\
			&\triangleq F_0(\mathbf{y}; \mathbf{y}^{(k)})+
			\sum_{i\in\mathcal{I}}F_{i}(\mathbf{x}_i, \mathbf{y}; \mathbf{x}^{(k)}_{i}, \mathbf{y}^{(k)})\nonumber\\ 
			s.t. \ 
			&\widetilde{\mathbf{G}}_i(\mathbf{x}_i, \mathbf{y}; \mathbf{x}^{(k)}_{i}, \mathbf{y}^{(k)})
			\preceq \mathbf{0}, i\in\mathcal{I}, \\
			&\mathbf{G}_{i}(\mathbf{x}_{i}; \mathbf{x}_{i}^{(k)})\preceq \mathbf{0},\ i\in\mathcal{I},  \\
			&\mathbf{G}_0(\mathbf{y}; \mathbf{y}^{(k)})\preceq \mathbf{0},\label{sp appro-cons-decoup-ineq-y}\\
			&\text{\eqref{sp coup-eq-cons},  \eqref{primal decoup-cvx-cons-x}, \eqref{primal decoup-cvx-cons-y}}, \nonumber
		\end{align}  
where $F_{0}$ and $F_{i}$ are strongly convex approximations of $f_0$ and $f_i$, respectively, satisfying Assumption~\ref{SCA-assump-surrogate-object},
and $\widetilde{\mathbf{G}}_i$, $\mathbf{G}_{i}$, and $\mathbf{G}_{0}$ are convex~approximations of $\widetilde{\mathbf{g}}_i$, $\mathbf{g}_{i}$, and $\mathbf{g}_{0}$, respectively, satisfying~Assumption~\ref{SCA-assump-surrogate-constriant}.\footnote{
When $F_{i}$ for some $i\in\mathcal{I}\cup\{0\}$ satisfies Assumption~\ref{SCA-assump-surrogate-object}.1, the strong convexity in $\mathbf{y}$ in Assumption~\ref{SCA-assump-surrogate-object}.1 for $F_{j}$ can reduce to convexity, for all $j\in\mathcal{I}\cup\{0\}-\{i\}$.}

Then, we apply the basic primal decomposition method for convex problems \cite{boyd2007notes,palomar2006tutorial} to Problem $\mathcal{P}^{spd,(k)}$.
Specifically, we fix the coupling variable $\mathbf{y}$ and decouple Problem~$\mathcal{P}^{spd,(k)}$ into $I$ (primal) subproblems, one for each $i\in\mathcal{I}$, as follows:
	\begin{align}
		\mathcal{P}^{spd, (k)}_{sub, i}:\ \min_{\mathbf{x}_{i}} \ &F_{i}(\mathbf{x}_i, \mathbf{y}; \mathbf{x}^{(k)}_{i}, \mathbf{y}^{(k)})\nonumber\\
		s.t. \ 
		&\widetilde{\mathbf{G}}_i(\mathbf{x}_i, \mathbf{y}; \mathbf{x}^{(k)}_{i}, \mathbf{y}^{(k)})\preceq \mathbf{0}, \label{sp appro-sub-coup-ineq}\\
		&\mathbf{A}_{i,\mathbf{x}}\mathbf{x}_i+\mathbf{A}_{i,\mathbf{y}}\mathbf{y}+ \mathbf{b}_{i}=\mathbf{0},\label{sp appro-sub-coup-eq}\\
		&\mathbf{G}_{i}(\mathbf{x}_{i}; \mathbf{x}_{i}^{(k)})\preceq \mathbf{0},  \label{sp appro-sub-decoup-ineq}\\
		&\mathbf{x}_{i}\in \mathcal{X}_i.  \label{sp appro-sub-decoup-cvx}    
	\end{align}
Then, Problem $\mathcal{P}^{spd,(k)}$ is equivalent to the following master (primal) problem:
	\begin{align}
		\mathcal{P}^{spd,(k)}_{mas}:\ \min_{\mathbf{y}} \ &F^*(\mathbf{y}; \mathbf{x}^{(k)}, \mathbf{y}^{(k)})\nonumber\\
		&\triangleq \!\! F_0(\mathbf{y}; \mathbf{y}^{(k)})+\sum_{i\in\mathcal{I}}\! F_{i}^*(\mathbf{y}; \mathbf{x}^{(k)}, \mathbf{y}^{(k)})
		\nonumber\\
		s.t. \ 
		&\text{\eqref{primal decoup-cvx-cons-y}, \eqref{sp appro-cons-decoup-ineq-y}},\nonumber     
	\end{align}  
where $F_{i}^*(\mathbf{y}; \mathbf{x}^{(k)}, \mathbf{y}^{(k)})\triangleq F_{i}(\mathbf{x}^{*,(k)}_{i}(\mathbf{y}), \mathbf{y}; \mathbf{x}^{(k)}_{i}, \mathbf{y}^{(k)})$ with $\mathbf{x}^{*,(k)}_{i}(\mathbf{y})$ being the unique optimal point of the convex subproblem $\mathcal{P}_{sub, i}^{spd, (k)}$ for $\mathbf{y}\!\in\!\operatorname{dom}F^*(\cdot; \mathbf{x}^{(k)}, \mathbf{y}^{(k)})$. 
Denote $\mathbf{y}^{*, (k)}$ as the unique optimal point of the convex master problem $\mathcal{P}^{spd, (k)}_{mas}$.
Then, $(\mathbf{x}^{*,(k)}(\mathbf{y}^{*, (k)}), \mathbf{y}^{*, (k)})$ is the unique optimal point of convex Problem $\mathcal{P}^{spd, (k)}$.

Next, taking a step from the given $(\mathbf{x}^{(k)}, \mathbf{y}^{(k)})$ towards $(\mathbf{x}^{*,(k)}(\mathbf{y}^{*, (k)}), \mathbf{y}^{*, (k)})$, we set:
\begin{equation}\label{sp smooth}
	\begin{gathered}
		\hspace{-3.2cm}(\mathbf{x}^{(k+1)}, \mathbf{y}^{(k+1)})=(\mathbf{x}^{(k)}, \mathbf{y}^{(k)})+\\
		\qquad\quad\gamma^{(k)}\left( (\mathbf{x}^{*,(k)}(\mathbf{y}^{*, (k)}), \mathbf{y}^{*, (k)})-(\mathbf{x}^{(k)}, \mathbf{y}^{(k)})\right),
	\end{gathered}
\end{equation}
where $\gamma^{(k)}$ is a step size satisfying \eqref{sca step size}. 
The proposed \mbox{SPD-M} for Problem $\mathcal{P}^{spd}$ is given by the convex approximate problems $\mathcal{P}^{spd,(k)}$,  the subproblems $\mathcal{P}_{sub, i}^{spd, (k)}$, $i\in\mathcal{I}$, the master problems $\mathcal{P}^{spd, (k)}_{mas}$, and the updates in \eqref{sp smooth}, $k\in\mathbb{N}$.

The effectiveness of SPD-M can be easily concluded based on \cite[Theorem~2]{scutari2017parallel}.
\begin{statement}\label{sp thm-1}
	If Assumptions~\ref{regularity}-\ref{SCA-assump-surrogate-constriant} and \ref{sp p} are satisfied, then at least one limit point of $\{(\mathbf{x}^{(k)}, \mathbf{y}^{(k)})\}_{k\in\mathbb{N}}$ generated by SPD-M is a stationary point of Problem $\mathcal{P}^{spd}$.
\end{statement}

\subsection{SPD-A and Its Convergence Analysis}\label{sp section-alg}
In SPD-M, only the value of $F^*_i$ (i.e., the component of the master problem $\mathcal{P}^{spd,(k)}_{mas}$'s objective function $F^{*}$) is accessible. 
To solve the master problem $\mathcal{P}^{spd, (k)}_{mas}$, more information of $F^*_i$ (e.g., its subgradients w.r.t. $\mathbf{y}$) is required. 
To obtain such information, we need the following assumptions.

\begin{assumption}\label{sp assump-surrogate-object-i-2}
	For all $i\in\mathcal{I}$, $F_{i}$ is separable on $\mathbf{x}_{i}$ and $\mathbf{y}$, i.e.,  $F_{i}(\mathbf{x}_i, \mathbf{y}; \mathbf{z}, \mathbf{w})=F_{i,\mathbf{x}}(\mathbf{x}_i; \mathbf{z}, \mathbf{w})+F_{i,\mathbf{y}}(\mathbf{y}; \mathbf{z}, \mathbf{w})$, where $F_{i,\mathbf{x}}:\mathcal{U}_i\times (\mathcal{X}_i\times\mathcal{Y})\rightarrow \mathbb{R}$ and $F_{i,\mathbf{y}}:\mathcal{V}\times (\mathcal{X}_i\times\mathcal{Y})\rightarrow \mathbb{R}$ are continuously differentiable w.r.t. the first argument.\footnote{Please see \cite[Examples 5-7]{scutari2017parallel} for examples of $F_i$  satisfying Assumptions~\ref{SCA-assump-surrogate-object} and \ref{sp assump-surrogate-object-i-2}.}
\end{assumption}

\begin{assumption}\label{sp assump-surrogate-constriant-i-c-2}
	For all $i\in\mathcal{I}$, $\widetilde{\mathbf{G}}_{i}$ is separable on $\mathbf{x}_{i}$ and $\mathbf{y}$ and linear on $\mathbf{y}$, i.e., $\widetilde{\mathbf{G}}_{i}(\mathbf{x}_{i}, \mathbf{y}; \mathbf{z},\mathbf{w})=\widetilde{\mathbf{G}}_{i,\mathbf{x}}(\mathbf{x}_{i}; \mathbf{z},\mathbf{w})+\mathbf{C}_{i}( \mathbf{z},\mathbf{w})\mathbf{y}$, where $\widetilde{\mathbf{G}}_{i,\mathbf{x}}:\mathcal{U}_{i}\times (\mathcal{X}_{i}\times\mathcal{Y})\rightarrow \mathbb{R}^{r_{i}}$ and $\mathbf{C}_{i}:\mathcal{X}_{i}\times\mathcal{Y}\rightarrow\mathbb{R}^{\tilde{r}_{i}\times n_0}$ are continuous.
\end{assumption}

Two examples of $\widetilde{\mathbf{G}}_i$ that satisfies Assumptions~\ref{SCA-assump-surrogate-constriant} and~\ref{sp assump-surrogate-constriant-i-c-2} are: 
1) If $\tilde{\mathbf{g}}_i$ is separable on $\mathbf{x}_i$ and $\mathbf{y}$ and linear on $\mathbf{y}$, i.e., $\tilde{\mathbf{g}}_i(\mathbf{x}_i,\mathbf{y})=\tilde{\mathbf{g}}_{i,\mathbf{x}}(\mathbf{x}_i)+\mathbf{C}_{i}\mathbf{y}$, then $\widetilde{\mathbf{G}}_{i}$ can be chosen by constructing approximate function $\widetilde{\mathbf{G}}_{i,\mathbf{x}}$ of $\tilde{\mathbf{g}}_{i,\mathbf{x}}$ that satisfies Assumption~\ref{SCA-assump-surrogate-constriant} and keeping $\mathbf{C}_{i}\mathbf{y}$ unchanged, i.e., $\widetilde{\mathbf{G}}_{i}(\mathbf{x}_{i},\mathbf{y};\mathbf{z},\mathbf{w})=\widetilde{\mathbf{G}}_{i,\mathbf{x}}(\mathbf{x}_{i};\mathbf{z})+\mathbf{C}_{i}\mathbf{y}$;
2) If $\tilde{\mathbf{g}}_i$ has a Difference of Convex (DC) structure, i.e., $\tilde{\mathbf{g}}_i(\mathbf{x}_i,\mathbf{y})=\tilde{\mathbf{g}}_{i}^{+}(\mathbf{x}_i)-\tilde{\mathbf{g}}_{i}^{-}(\mathbf{x}_i,\mathbf{y})$ with two convex and continuously differentiable functions $\tilde{\mathbf{g}}_{i}^{+}$ and $\tilde{\mathbf{g}}_{i}^{-}$, then  $\widetilde{\mathbf{G}}_{i}$ can be chosen by linearizing the concave part $-\tilde{\mathbf{g}}_{i}^{-}$ and keeping the convex part $\tilde{\mathbf{g}}_{i}^{+}$ unchanged, i.e., $\widetilde{\mathbf{G}}_{i}(\mathbf{x}_{i}, \mathbf{y}; \mathbf{z},\mathbf{w})\!=\!\tilde{\mathbf{g}}_{i}^{+}(\mathbf{x}_{i})\!-\!\tilde{\mathbf{g}}_{i}^{-}(\mathbf{z},\mathbf{w})\!-\!\nabla_{\mathbf{x}_{i}} \tilde{\mathbf{g}}_{i}^{-}(\mathbf{z},\mathbf{w})^{T}(\mathbf{x}_{i}\!-\!\mathbf{z})\!-\!\nabla_{\mathbf{y}} \tilde{\mathbf{g}}_{i}^{-}(\mathbf{z},\mathbf{w})^{T}(\mathbf{y}\!-\!\mathbf{w})$.

Under Assumption~\ref{sp assump-surrogate-constriant-i-c-2}, the constraints in \eqref{sp appro-sub-coup-ineq} reduce to 
\begin{align}
	\widetilde{\mathbf{G}}_{i,\mathbf{x}}(\mathbf{x}_i; \mathbf{x}^{(k)}_{i}, \mathbf{y}^{(k)})+\mathbf{C}_{i}(\mathbf{x}^{(k)}_{i}, \mathbf{y}^{(k)})\mathbf{y}\preceq \mathbf{0}.\label{sp appro-sub-coup-ineq-2}
\end{align}

\begin{assumption}\label{sp assump-slater and feasible}
	For all $i\in\mathcal{I}$ and $k\in\mathbb{N}$, the subproblem $\mathcal{P}^{spd, (k)}_{sub, i}$ 
	satisfies Slater's condition.
\end{assumption}

Assumption~\ref{sp assump-slater and feasible} ensures that for all $i\!\in\!\mathcal{I}$ and $\mathbf{y}\!\in\!\operatorname{dom}F^*(\cdot; \mathbf{x}^{(k)}, \mathbf{y}^{(k)})$, $\mathbf{x}^{*,(k)}_{i}(\mathbf{y})$ is a stationary point of the subproblem $\mathcal{P}^{spd, (k)}_{sub, i}$, and its Lagrange multipliers $\tilde{\boldsymbol{\mu}}_{i}^{(k)}(\mathbf{y})$,  $\tilde{\boldsymbol{\lambda}}_{i}^{(k)}(\mathbf{y})$, and $\boldsymbol{\mu}_{i}^{(k)}(\mathbf{y})$ associated with the constraints in \eqref{sp appro-sub-coup-ineq-2}, \eqref{sp appro-sub-coup-eq}, and \eqref{sp appro-sub-decoup-ineq}, are optimal Lagrange multipliers \cite[pp. 244]{boyd2004convex}.

\begin{algorithm}[t]
	\caption{Primal Decomposition Algorithm \cite{boyd2007notes,palomar2006tutorial} for Problem $\mathcal{P}^{spd,(k)}$}\label{sp alg:alg1-PD}
	\begin{algorithmic}[1]\label{alg1}\small
		\STATE \textbf{initialization}: 
		Set $t=0$ and $\mathbf{y}^{(0)}=\mathbf{y}^{(k)}$; 
		choose
		$\{\gamma_{in}^{(t)}\}_{t\in\mathbb{N}}\subseteq(0,+\infty)$ satisfying  $\sum_{t=0}^{\infty}\gamma_{in}^{(t)}=\infty$ and  $\sum_{t=0}^{\infty}(\gamma_{in}^{(t)})^2<\infty$.
		\REPEAT
		\FOR{all $i\in\mathcal{I}$}
		\STATE Get the unique optimal point $\mathbf{x}_{i}^{*,(k)}(\mathbf{y}^{(t)})$ of the convex subproblem $\mathcal{P}^{spd,(k)}_{sub, i}$.
		\STATE Obtain the corresponding Lagrange multipliers $\tilde{\boldsymbol{\mu}}_{i}^{(k)}(\mathbf{y}^{(t)})$  and $\tilde{\boldsymbol{\lambda}}_{i}^{(k)}(\mathbf{y}^{(t)})$.
		\STATE Compute $\partial_{\mathbf{y}}F_{i}^{*}(\mathbf{y}^{(t)}; \mathbf{x}^{(k)}, \mathbf{y}^{(k)})$ according to:
		\begin{equation*}
			\begin{aligned}
				&\partial_{\mathbf{y}}F_{i}^{*}(\mathbf{y}^{(t)}; \mathbf{x}^{(k)}, \mathbf{y}^{(k)})=\nabla_{\mathbf{y}}F_{i,\mathbf{y}}(\mathbf{y}^{(t)}; \mathbf{x}^{(k)}, \mathbf{y}^{(k)})\\
				&\mathbf{C}_{i}(\mathbf{x}^{(k)}_{i}, \mathbf{y}^{(k)})^{T}\tilde{\boldsymbol{\mu}}^{(k)}_{i}(\mathbf{y}^{(t)})+\mathbf{A}_{i,\mathbf{y}}^{T}\tilde{\boldsymbol{\lambda}}^{(k)}_{i}(\mathbf{y}^{(t)}).
			\end{aligned}
		\end{equation*}
		\ENDFOR
		\STATE Compute $\partial_{\mathbf{y}}F^{*}(\mathbf{y}^{(t)}; \mathbf{x}^{(k)}, \mathbf{y}^{(k)})$ according to:
		\begin{equation*}
			\begin{aligned}
				\partial_{\mathbf{y}}F^{*}(\mathbf{y}^{(t)}; \mathbf{x}^{(k)}, \mathbf{y}^{(k)})=&\nabla_{\mathbf{y}}F_0(\mathbf{y}^{(t)}; \mathbf{y}^{(k)})\\
				&+\sum_{i\in\mathcal{I}}
				\partial_{\mathbf{y}}F_{i}^{*}(\mathbf{y}^{(t)}; \mathbf{x}^{(k)}, \mathbf{y}^{(k)}).
			\end{aligned}
		\end{equation*}
		\STATE Update $\mathbf{y}$ according to: 
		$$
		\mathbf{y}^{(t+1)}= \operatorname{P}_{\mathcal{C}^{(k)}}\left(\mathbf{y}^{(t)} - \gamma_{in}^{(t)} \partial_{\mathbf{y}}F^{*}(\mathbf{y}^{(t)}; \mathbf{x}^{(k)}, \mathbf{y}^{(k)})\right),
		$$
		where $\mathcal{C}^{(k)}$ is the feasible set of the master problem $\mathcal{P}^{spd, (k)}_{mas}$.
		\STATE Set $t \gets t+1$.
		\UNTIL{Some termination criterion is met.}
	\end{algorithmic}
	\label{alg1}
\end{algorithm}

\begin{algorithm}[t]
	\caption{SPD-A}\label{sp alg:alg1-SCA}
	\begin{algorithmic}[1]\label{alg1}\small
		\STATE \textbf{initialization}: Set $k=0$ and choose any feasible point $(\mathbf{x}^{(0)},\mathbf{y}^{(0)})$ and $\{\gamma^{(k)}\}_{k\in\mathbb{N}}\subseteq(0,1]$ satisfying (\ref{sca step size}).
		\REPEAT
		\STATE Obtain the unique optimal point $(\mathbf{x}^{*,(k)}(\mathbf{y}^{*, (k)}), \mathbf{y}^{*, (k)})$ of the convex approximate problem $\mathcal{P}^{spd,(k)}$ by Algorithm~\ref{sp alg:alg1-PD}.
		\STATE Update $(\mathbf{x}, \mathbf{y})$ by (\ref{sp smooth}).
		\STATE Set $k \gets k+1$.
		\UNTIL{Some termination criterion is met.}
	\end{algorithmic}
	\label{alg1}
\end{algorithm}

Now, we are ready to propose SPD-A for Problem $\mathcal{P}^{spd}$. 
Specifically, first, for fixed $k$, we adopt the basic primal decomposition algorithm for convex problems \cite{boyd2007notes,palomar2006tutorial} to solve Problem $\mathcal{P}^{spd,(k)}$. 
The detailed procedure is summarized in 
Algorithm~\ref{sp alg:alg1-PD}.\footnote{In Steps~6, 8, and 9, parallel computations can be applied to conduct matrix (vector) multiplications and additions.}
In Step 4, we can solve the convex subproblem $\mathcal{P}^{spd,(k)}_{sub, i}$ using algorithms for convex problems, e.g., interior-point methods.
In Step 5, 
we can obtain the corresponding Lagrange multipliers $\tilde{\boldsymbol{\mu}}_{i}^{(k)}(\mathbf{y}^{(t)})$ and $\tilde{\boldsymbol{\lambda}}_{i}^{(k)}(\mathbf{y}^{(t)})$ 
by solving the KKT system or the KKT conditions. 
Then, we present the complete SPD-A in Algorithm~\ref{sp alg:alg1-SCA}.
Notably, in each iteration, the convex approximate problem in SPD-A is directly solved by Algorithm~\ref{sp alg:alg1-PD}, while that in SCAPD \cite{scutari2017parallel} is solved by first introducing slack variables and additional constraints and then applying the standard primal decomposition algorithm for convex problems \cite{boyd2007notes,palomar2006tutorial}. 
Thus, SPD-A may achieve better convergence performance and shorter computation time than SCAPD \cite{scutari2017parallel}, which will be shown in Section~\ref{sec numerical results}.

Based on the convergence results of the projected subgradient algorithm \cite[Proposition 8.2.6]{bertsekas2003convex} and Statement~\ref{sp thm-1},  
we can give the convergence of Algorithm~\ref{sp alg:alg1-SCA}.
\begin{statement}\label{sp thm-2}
	If Assumptions~\ref{regularity}-\ref{SCA-assump-surrogate-constriant} and~\ref{sp p}-\ref{sp assump-slater and feasible} are satisfied, then at least one limit point of $\{(\mathbf{x}^{(k)}, \mathbf{y}^{(k)})\}_{k\in\mathbb{N}}$ generated by Algorithm~\ref{sp alg:alg1-SCA} is a stationary point of Problem $\mathcal{P}^{spd}$.
\end{statement}

\section{Dual Decomposition}\label{dual section}
In this section, we consider a nonconvex problem with coupling constraints and propose a new dual decomposition method and a corresponding algorithm for this problem, referred to as DD-M and \mbox{DD-A}, respectively.

\subsection{Problem Formulation}\label{dual section-problem}
Consider the following nonconvex problem:
\begin{align}
	\mathcal{P}^{dd}:\ \min_{\mathbf{x}} \ &f(\mathbf{x})\triangleq\sum_{i\in\mathcal{I}}f_i(\mathbf{x}_i)&&\nonumber\\  
	s.t. \ 
	&\sum_{i\in\mathcal{I}}\tilde{\mathbf{g}}_i(\mathbf{x}_i)\preceq\mathbf{0},&&\label{dual coup-ineq-cons}\\
	&\sum_{i\in\mathcal{I}}\tilde{\mathbf{h}}_i(\mathbf{x}_i)=\mathbf{0},&&\label{dual coup-eq-cons}\\
	&\mathbf{g}_{i}(\mathbf{x}_{i})\preceq \mathbf{0},\ i\in\mathcal{I},  &&\label{dual decoup-ineq-cons}\\
	&\mathbf{x}_{i}\in \mathcal{X}_i,\ i\in\mathcal{I}, && \label{dual decoup-cvx-cons}       
\end{align}    
where $\mathbf{x}\triangleq(\mathbf{x}_1,\mathbf{x}_2,\cdots,\mathbf{x}_{I})\in \mathbb{R}^{n}$ with $\mathbf{x}_i\in \mathbb{R}^{n_i}$, $i\in \mathcal{I}$, $f_i:\mathcal{U}_i\rightarrow \mathbb{R}$,
$\tilde{\mathbf{g}}_i:\mathcal{U}_i\rightarrow \mathbb{R}^{\tilde{r}}$,
$\tilde{\mathbf{h}}_i:\mathcal{U}_i\rightarrow \mathbb{R}^{\tilde{m}}$
$\mathbf{g}_{i}:\mathcal{U}_{i}\rightarrow \mathbb{R}^{r_i}$, $i\in\mathcal{I}$. 
Notably, $\tilde{\mathbf{h}}_i$ for all $i\in\mathcal{I}$ is generally not affine in $\mathbf{x}_{i}$.\footnote{If $\tilde{\mathbf{h}}_{i}$ for all $i\in\mathcal{I}$ is affine in $\mathbf{x}_{i}$, then the generally nonlinear equality constraints in (\ref{dual coup-eq-cons}) reduce to the linear equality constraints in (\ref{sd coup-eq-cons}).}

\begin{assumption}[Assumptions on Problem $\mathcal{P}^{dd}$]\label{dual assump-p}
1) For all $i\in \mathcal{I}$, $\mathcal{X}_i$ is a nonempty, closed, and convex set that belongs to the open set $\mathcal{U}_i\subseteq\mathbb{R}^{n_i}$;
2) For all $i\in \mathcal{I}$, $f_i$ is continuously differentiable on $\mathcal{U}_i$,  
	and $\nabla f_i$ is Lipschitz continuous on $\mathcal{X}_i$; 
3) For all $i\in \mathcal{I}$, $\tilde{\mathbf{g}}_i$ and $\mathbf{g}_{i}$ are continuously differentiable  on $\mathcal{U}_i$; 
4) For all $i\in \mathcal{I}$, $\tilde{\mathbf{h}}_i$ is continuously differentiable  on $\mathcal{U}_i$; 
5) $f$ is bounded below.
\end{assumption}

Observe that if the coupling constraints in \eqref{dual coup-ineq-cons} and \eqref{dual coup-eq-cons} are relaxed, then Problem $\mathcal{P}^{dd}$ decouples.
Now, we form the partial Lagrangian by augmenting the objective function with a weighted sum of the constraints in \eqref{dual coup-ineq-cons} and \eqref{dual coup-eq-cons}:
$L(\mathbf{x}, \tilde{\boldsymbol{\mu}}, \tilde{\boldsymbol{\lambda}} )\triangleq \sum_{i\in\mathcal{I}}L_{i}(\mathbf{x}_{i}, \tilde{\boldsymbol{\mu}}, \tilde{\boldsymbol{\lambda}})$,
where $L_{i}(\mathbf{x}_{i}, \tilde{\boldsymbol{\mu}}, \tilde{\boldsymbol{\lambda}} )\triangleq f_i(\mathbf{x}_i)+\tilde{\boldsymbol{\mu}}^T\tilde{\mathbf{g}}_i(\mathbf{x}_i)+\tilde{\boldsymbol{\lambda}}^T\tilde{\mathbf{h}}_i(\mathbf{x}_i)$, and the weights $\tilde{\boldsymbol{\mu}}\in \mathbb{R}^{\tilde{r}}_{+}$ and $\tilde{\boldsymbol{\lambda}}\in \mathbb{R}^{\tilde{m}}$ are the Lagrange multipliers associated with the constraints in \eqref{dual coup-ineq-cons} and \eqref{dual coup-eq-cons}, respectively. 
$L$ is separable in $\mathbf{x}_i$, $i\in\mathcal{I}$.
Thus, the problem for obtaining the dual function, i.e.,
$q(\tilde{\boldsymbol{\mu}}, \tilde{\boldsymbol{\lambda}})=\inf\{L(\mathbf{x}, \tilde{\boldsymbol{\mu}}, \tilde{\boldsymbol{\lambda}})\ | \ \mathbf{g}_{i}(\mathbf{x}_{i})\preceq \mathbf{0}, \mathbf{x}_{i}\in\mathcal{X}_i, i\in\mathcal{I}\}$, 
can be separated into $I$ (dual) subproblems, 
one for each $i\in\mathcal{I}$:
	\begin{align}
		\mathcal{P}^{dd}_{sub, i}:\ \min_{\mathbf{x}_i} \ &L_{i}(\mathbf{x}_{i}, \tilde{\boldsymbol{\mu}}, \tilde{\boldsymbol{\lambda}})\nonumber\\ 
		s.t. \ 
		&\mathbf{g}_{i}(\mathbf{x}_{i})\preceq \mathbf{0},\label{dual sub-decoup-ineq}\\
		&\mathbf{x}_{i}\in \mathcal{X}_i.\label{dual sub-decoup-cvx}
	\end{align}
Then the master (dual) problem (i.e., dual problem) is:
	\begin{align}
		\mathcal{P}^{dd}_{mas}:\ \max_{\tilde{\boldsymbol{\mu}}, \tilde{\boldsymbol{\lambda}}} \ &q(\tilde{\boldsymbol{\mu}}, \tilde{\boldsymbol{\lambda}}) = \sum_{i\in \mathcal{I}}L_{i}(\mathbf{x}^{*}_{i}(\tilde{\boldsymbol{\mu}}, \tilde{\boldsymbol{\lambda}}), \tilde{\boldsymbol{\mu}}, \tilde{\boldsymbol{\lambda}} )\nonumber\\
		s.t. \ 
		&\tilde{\boldsymbol{\mu}}\succeq\mathbf{0},\label{dual mas-decoup}
	\end{align}
where $\mathbf{x}^*_i(\tilde{\boldsymbol{\mu}}, \tilde{\boldsymbol{\lambda}})$ represents an optimal point of the subproblem $\mathcal{P}_{sub, i}^{dd}$ for $(\tilde{\boldsymbol{\mu}}, \tilde{\boldsymbol{\lambda}})\in\mathbb{R}_{+}^{\tilde{r}}\times\mathbb{R}^{\tilde{m}}$.
Each subproblem $\mathcal{P}_{sub, i}^{dd}$ for $(\tilde{\boldsymbol{\mu}}, \tilde{\boldsymbol{\lambda}})\in\mathbb{R}_{+}^{\tilde{r}}\times\mathbb{R}^{\tilde{m}}$ may have multiple optimal points but a unique optimal value, and hence $q(\tilde{\boldsymbol{\mu}}, \tilde{\boldsymbol{\lambda}})$ is unique.
Besides, the master problem $\mathcal{P}^{dd}_{mas}$ is always convex.\footnote{A Lagrange dual problem is always convex whether or not the original problem is convex.}

Since Problem $\mathcal{P}^{dd}$ is nonconvex, all subproblems are generally nonconvex (or cannot be shown to be convex).
Thus, it is usually difficult to obtain an optimal point of the subproblem $\mathcal{P}_{sub, i}^{dd}$ for all $i\in\mathcal{I}$ and the objective value $q(\tilde{\boldsymbol{\mu}}, \tilde{\boldsymbol{\lambda}})$ of the master problem $\mathcal{P}^{dd}_{mas}$ at $(\tilde{\boldsymbol{\mu}}, \tilde{\boldsymbol{\lambda}})\in\mathbb{R}_{+}^{\tilde{r}}\times\mathbb{R}^{\tilde{m}}$.  
Besides, strong duality does not generally hold for nonconvex Problem $\mathcal{P}^{dd}$.
Hence, the optimal point of the master problem $\mathcal{P}^{dd}_{mas}$ may not lead to a stationary point of Problem $\mathcal{P}^{dd}$.
That is, the dual decomposition method (given by the subproblems $\mathcal{P}_{sub, i}^{dd}$, $i\in\mathcal{I}$ and the master problem $\mathcal{P}^{dd}_{mas}$) for convex problems \cite{boyd2007notes,palomar2006tutorial} cannot produce a stationary point of nonconvex Problem $\mathcal{P}^{dd}$.
Consequently, the dual decomposition algorithm for convex problems \cite{boyd2007notes,palomar2006tutorial}, which is based on the dual decomposition method, is not applicable to nonconvex Problem $\mathcal{P}^{dd}$.

\subsection{DD-M and Its Theoretical Analysis}\label{dual section-method}
To address the issues discussed above, we propose \mbox{DD-M} for Problem $\mathcal{P}^{dd}$. 
Specifically, we approximate the master problem $\mathcal{P}^{dd}_{mas}$ with the following problem:
\begin{equation*}
	\begin{aligned}
		\mathcal{P}^{dd,\dag}_{mas}:\ \max_{\tilde{\boldsymbol{\mu}}, \tilde{\boldsymbol{\lambda}}} \ &q^{\dag}(\tilde{\boldsymbol{\mu}}, \tilde{\boldsymbol{\lambda}})\triangleq \sum_{i\in \mathcal{I}}L_{i}(\mathbf{x}^{\dag}_{i}(\tilde{\boldsymbol{\mu}}, \tilde{\boldsymbol{\lambda}}), \tilde{\boldsymbol{\mu}}, \tilde{\boldsymbol{\lambda}})\\
		s.t. \ 
		&\text{\eqref{dual mas-decoup}}, 
	\end{aligned} 
\end{equation*}
where $\mathbf{x}^{\dag}_i(\tilde{\boldsymbol{\mu}}, \tilde{\boldsymbol{\lambda}})$ represents a (selected) stationary point of the subproblem $\mathcal{P}_{sub, i}^{dd}$ for $(\tilde{\boldsymbol{\mu}}, \tilde{\boldsymbol{\lambda}})\in\mathbb{R}_{+}^{\tilde{r}}\times\mathbb{R}^{\tilde{m}}$. 
Note that each subproblem $\mathcal{P}_{sub, i}^{dd}$ for $(\tilde{\boldsymbol{\mu}}, \tilde{\boldsymbol{\lambda}})\in\mathbb{R}_{+}^{\tilde{r}}\times\mathbb{R}^{\tilde{m}}$ may have multiple stationary points, and hence $q^{\dag}(\tilde{\boldsymbol{\mu}}, \tilde{\boldsymbol{\lambda}})$ may change with the selected stationary points, which is different from $q(\tilde{\boldsymbol{\mu}}, \tilde{\boldsymbol{\lambda}})$.
In addition, Problem $\mathcal{P}^{dd,\dag}_{mas}$ is 
generally nonconvex (or cannot be shown to be convex), in contrast to the convex master problem $\mathcal{P}^{dd}_{mas}$.
We also call Problem $\mathcal{P}^{dd,\dag}_{mas}$ the master problem.
The proposed DD-M for Problem $\mathcal{P}^{dd}$ is given by the subproblems $\mathcal{P}_{sub, i}^{dd}$, $i\in\mathcal{I}$ and the master problem $\mathcal{P}^{dd,\dag}_{mas}$.

Next, we verify the effectiveness of DD-M by analyzing the relationship between the stationary points of the master problem $\mathcal{P}^{dd,\dag}_{mas}$ and Problem $\mathcal{P}^{dd}$ under the following assumptions.

\begin{assumption}\label{dual assump-stationary-1}
	There exists an open set $\mathcal{N}\subseteq\mathbb{R}_{+}^{\tilde{r}}\times\mathbb{R}^{\tilde{m}}$ and single-valued continuously differentiable functions $\mathbf{X}^{\dag}_i:\mathcal{N}\rightarrow \mathcal{U}_{i}$, $\mathbf{M}_i:\mathcal{N}\rightarrow \mathbb{R}_+^{r_i}$, $i\in\mathcal{I}$ such that for all $i\in\mathcal{I}$ and $(\tilde{\boldsymbol{\mu}}, \tilde{\boldsymbol{\lambda}})\in \mathcal{N}$, $\mathbf{X}^{\dag}_i(\tilde{\boldsymbol{\mu}}, \tilde{\boldsymbol{\lambda}})\in\operatorname{int}(\mathcal{X}_{i})$ is a stationary point of the subproblem $\mathcal{P}_{sub, i}^{dd}$, and $\mathbf{M}_i(\tilde{\boldsymbol{\mu}}, \tilde{\boldsymbol{\lambda}})$ is the Lagrange multiplier associated with the constraint in \eqref{dual sub-decoup-ineq}.
\end{assumption}

\begin{myassumptiondual}\label{dual assump-stationary-2}
1) $f_i$, $\tilde{\mathbf{g}}_i$, $\tilde{\mathbf{h}}_i$, and $\mathbf{g}_i$, $i\in\mathcal{I}$ are twice continuously differentiable;
2) There exists a stationary point $\mathbf{x}^{\ddagger}\in \prod_{i\in\mathcal{I}}\operatorname{int}(\mathcal{X}_{i})$ of Problem $\mathcal{P}^{dd}$ with Lagrange multipliers $\tilde{\boldsymbol{\mu}}^{\ddagger}$, $\tilde{\boldsymbol{\lambda}}^{\ddagger}$,  and $\boldsymbol{\mu}^{\ddagger}_i$, $i\in\mathcal{I}$ associated with the constraints in \eqref{dual coup-ineq-cons}, \eqref{dual coup-eq-cons}, and \eqref{dual decoup-ineq-cons}, respectively, such that for all $i\in\mathcal{I}$, the gradient of the subproblem $\mathcal{P}^{dd}_{sub, i}$'s KKT function $\mathbf{k}_{i}$ for $(\tilde{\boldsymbol{\mu}}, \tilde{\boldsymbol{\lambda}})\in\mathbb{R}_{+}^{\tilde{r}}\times\mathbb{R}^{\tilde{m}}$ w.r.t. $(\mathbf{x}_i,\boldsymbol{\mu}_i)$ at $(\mathbf{x}_i^{\ddagger}, \boldsymbol{\mu}^{\ddagger}_i, \tilde{\boldsymbol{\mu}}^{\ddagger}, \tilde{\boldsymbol{\lambda}}^{\ddagger})$, i.e., $\nabla_{(\mathbf{x}_i,\boldsymbol{\mu}_i)} \mathbf{k}_i(\mathbf{x}_i^{\ddagger},  \boldsymbol{\mu}^{\ddagger}_i, \tilde{\boldsymbol{\mu}}^{\ddagger}, \tilde{\boldsymbol{\lambda}}^{\ddagger})$, is invertible.
\end{myassumptiondual}

\begin{lemma}\label{dual lem1}
	Suppose that Assumptions~\ref{dual assump-p} and \ref{dual assump-stationary-2} are satisfied.
	Then there exists a neighborhood $\mathcal{N}_{(\tilde{\boldsymbol{\mu}}^{\ddagger}, \tilde{\boldsymbol{\lambda}}^{\ddagger})}\subseteq\mathbb{R}_{+}^{\tilde{r}}\times\mathbb{R}^{\tilde{m}}$ of $(\tilde{\boldsymbol{\mu}}^{\ddagger}, \tilde{\boldsymbol{\lambda}}^{\ddagger})$ and neighborhoods $\mathcal{N}_{\mathbf{x}_i^{\ddagger}}$,
	$\mathcal{N}_{\boldsymbol{\mu}^{\ddagger}_i}$, $i\in\mathcal{I}$ of $\mathbf{x}_i^{\ddagger}$, $\boldsymbol{\mu}^{\ddagger}_i$, $i\in\mathcal{I}$, respectively, such that: 
	\begin{enumerate}
		\item Assumption~\ref{dual assump-stationary-1} holds with $\mathcal{N}$ being $\mathcal{N}_{(\tilde{\boldsymbol{\mu}}^{\ddagger}, \tilde{\boldsymbol{\lambda}}^{\ddagger})}$;

		\item For all $i\in\mathcal{I}$ and $(\tilde{\boldsymbol{\mu}}, \tilde{\boldsymbol{\lambda}})\in \mathcal{N}_{(\tilde{\boldsymbol{\mu}}^{\ddagger}, \tilde{\boldsymbol{\lambda}}^{\ddagger})}$, the subproblem $\mathcal{P}^{dd}_{sub,i}$ has a unique stationary point in $\mathcal{N}_{\mathbf{x}_{i}^{\ddagger}}$ that has Lagrange multiplier associated with the constraint in \eqref{dual sub-decoup-ineq} in  $\mathcal{N}_{\boldsymbol{\mu}_{i}^{\ddagger}}$, and the Lagrange multiplier is unique.
	\end{enumerate}
\end{lemma}
\begin{IEEEproof}
	The proof is similar to that of Lemma~\ref{primal lem1}. 
\end{IEEEproof}

\begin{theorem}\label{dual thm1}
	Suppose that Assumptions~\ref{dual assump-p} and \ref{dual assump-stationary-1} (or \ref{dual assump-stationary-2}) are satisfied.
	Let $\mathcal{N}$ (or $\mathcal{N}_{(\tilde{\boldsymbol{\mu}}^{\ddag}, \tilde{\boldsymbol{\lambda}}^{\ddagger})}$) denote the open set specified in Assumption~\ref{dual assump-stationary-1} (or Lemma~\ref{dual lem1}.1).
	Then the following results~hold.
	\begin{enumerate}
		\item $q^{\dag}$, where $\mathbf{x}^{\dag}_i(\tilde{\boldsymbol{\mu}}, \tilde{\boldsymbol{\lambda}})=\mathbf{X}^{\dag}_i(\tilde{\boldsymbol{\mu}}, \tilde{\boldsymbol{\lambda}})$ with $\mathbf{X}_{i}^{\dag}$ specified in Assumption \ref{dual assump-stationary-1} (or Lemma~\ref{dual lem1}.1), is continuously differentiable on $\mathcal{N}$ (or $\mathcal{N}_{(\tilde{\boldsymbol{\mu}}^{\ddag}, \tilde{\boldsymbol{\lambda}}^{\ddagger})}$), and for all $(\tilde{\boldsymbol{\mu}}, \tilde{\boldsymbol{\lambda}})\in\mathcal{N}$ (or $\mathcal{N}_{(\tilde{\boldsymbol{\mu}}^{\ddag}, \tilde{\boldsymbol{\lambda}}^{\ddagger})}$), 
		the gradients  
		$\nabla_{\tilde{\boldsymbol{\mu}}}q^{\dag}(\tilde{\boldsymbol{\mu}}, \tilde{\boldsymbol{\lambda}})$ and $\nabla_{\tilde{\boldsymbol{\lambda}}}q^{\dag}(\tilde{\boldsymbol{\mu}}, \tilde{\boldsymbol{\lambda}})$
		are given by: 
		\begin{align}\label{dual mas-obj-grad-mu}
			&\nabla_{\tilde{\boldsymbol{\mu}}}q^{\dag}(\tilde{\boldsymbol{\mu}}, \tilde{\boldsymbol{\lambda}})=\sum_{i\in \mathcal{I}}
			\tilde{\mathbf{g}}_{i}(\mathbf{X}_{i}^{\dagger}(\tilde{\boldsymbol{\mu}}, \tilde{\boldsymbol{\lambda}})),\\
		\label{dual mas-obj-grad-lambda}
			&\nabla_{\tilde{\boldsymbol{\lambda}}}q^{\dag}(\tilde{\boldsymbol{\mu}}, \tilde{\boldsymbol{\lambda}})=\sum_{i\in \mathcal{I}}
			\tilde{\mathbf{h}}_{i}(\mathbf{X}_{i}^{\dagger}(\tilde{\boldsymbol{\mu}}, \tilde{\boldsymbol{\lambda}}));
		\end{align}
		
		\item If $(\tilde{\boldsymbol{\mu}}^{\dag}, \tilde{\boldsymbol{\lambda}}^{\dag})\in\mathcal{N}$ (or $\mathcal{N}_{(\tilde{\boldsymbol{\mu}}^{\ddag}, \tilde{\boldsymbol{\lambda}}^{\ddagger})}$) is a stationary point of the master problem $\mathcal{P}^{dd,\dag}_{mas}$, then $\mathbf{X}^{\dag}(\tilde{\boldsymbol{\mu}}^{\dag}, \tilde{\boldsymbol{\lambda}}^{\dag})$ is a stationary point of Problem $\mathcal{P}^{dd}$, where $\mathbf{X}^{\dag}(\tilde{\boldsymbol{\mu}}^{\dag}, \tilde{\boldsymbol{\lambda}}^{\dag})=(\mathbf{X}_1^{\dag}(\tilde{\boldsymbol{\mu}}^{\dag}, \tilde{\boldsymbol{\lambda}}^{\dag}), \cdots, \mathbf{X}^{\dag}_I(\tilde{\boldsymbol{\mu}}^{\dag}, \tilde{\boldsymbol{\lambda}}^{\dag}))$ with  $\mathbf{X}_{i}^{\dag}(\tilde{\boldsymbol{\mu}}^{\dag}, \tilde{\boldsymbol{\lambda}}^{\dag})$ being the stationary point of the subproblem $\mathcal{P}^{dd}_{sub, i}$ and $\mathbf{X}_{i}^{\dag}$ specified in Assumption \ref{dual assump-stationary-1} (or Lemma~\ref{dual lem1}.1).
	\end{enumerate}
\end{theorem}
\begin{IEEEproof}
	The proof is similar to that of Theorem~\ref{primal thm1}. 
\end{IEEEproof}

Assumptions~\ref{dual assump-stationary-1} and \ref{dual assump-stationary-2}, Lemma~\ref{dual lem1}, and Theorem~\ref{dual thm1} are similar to Assumptions~\ref{primal assump-stationary-1} and \ref{primal assump-stationary-2}, Lemma~\ref{primal lem1}, and Theorem~\ref{primal thm1}, respectively. We omit similar discussions for briefness.

\subsection{DD-A and Its Local Convergence Analysis}\label{dual section-alg}
By Theorem~\ref{dual thm1}, the key to designing a dual decomposition algorithm for Problem $\mathcal{P}^{dd}$ is to design a master algorithm for the master problem $\mathcal{P}^{dd, \dag}_{mas}$.
Generally, the closed-form expression for $q^{\dag}$ is hard to obtain, and only local information about $q^{\dag}$ (e.g., its value and gradient at a point) is accessible.
Thus, we need to iteratively solve the subproblems $\mathcal{P}_{sub, i}^{dd}$, $i\in\mathcal{I}$ and update the dual variables $\tilde{\boldsymbol{\mu}}$ and $\tilde{\boldsymbol{\lambda}}$ based on the accessible local information of $q^{\dag}$.

Based on this, we propose DD-A for Problem $\mathcal{P}^{dd}$, which utilizes the SCA-based algorithm \cite{scutari2018parallel} to solve the master problem $\mathcal{P}^{dd, \dag}_{mas}$. 
Specifically, at iteration $k$, for fixed $(\tilde{\boldsymbol{\mu}}^{(k)}, \tilde{\boldsymbol{\lambda}}^{(k)})$ (which is obtained at iteration $k-1$), we can divide the algorithm process into two parts.
In the first part, we focus on solving the subproblem $\mathcal{P}_{sub, i}^{dd}$, $i\in\mathcal{I}$ separately. 
First, we obtain an arbitrary stationary point $\mathbf{x}^{\dag}_i(\tilde{\boldsymbol{\mu}}^{(k)}, \tilde{\boldsymbol{\lambda}}^{(k)})$ of the subproblem $\mathcal{P}_{sub, i}^{dd}$ using algorithms for generally nonconvex problems (e.g., MM \cite{sun2017mm} and SCA algorithm \cite{scutari2018parallel}).
Then, we can obtain the corresponding Lagrange multiplier $\boldsymbol{\mu}_i(\tilde{\boldsymbol{\mu}}^{(k)}, \tilde{\boldsymbol{\lambda}}^{(k)})$ by solving the KKT system or KKT conditions.
Note that $\boldsymbol{\mu}_i(\tilde{\boldsymbol{\mu}}^{(k)}, \tilde{\boldsymbol{\lambda}}^{(k)})$ is not utilized in the followed algorithm process but is essential for the local convergence analysis.
Next, we compute $\tilde{\mathbf{g}}_{i}(\mathbf{x}_{i}^{\dagger}(\tilde{\boldsymbol{\mu}}^{(k)}, \tilde{\boldsymbol{\lambda}}^{(k)}))$ and $\tilde{\mathbf{h}}_{i}(\mathbf{x}_{i}^{\dagger}(\tilde{\boldsymbol{\mu}}^{(k)}, \tilde{\boldsymbol{\lambda}}^{(k)}))$.

\begin{algorithm}[t]
	\caption{DD-A}\label{dual alg:alg1}
	\begin{algorithmic}[1]\label{alg1}\small
		\STATE \textbf{initialization}: Set $k=0$ and choose any point $(\tilde{\boldsymbol{\mu}}^{(0)}\succeq \mathbf{0}, \tilde{\boldsymbol{\lambda}}^{(0)})$, $\tau>0$, and $\{\gamma^{(k)}\}_{k\in\mathbb{N}}\subseteq(0,1]$ satisfying (\ref{sca step size}). 
		\REPEAT
		\FOR{all $i\in\mathcal{I}$}
		\STATE Get an arbitrary stationary point $\mathbf{x}^{\dagger}_{i}(\tilde{\boldsymbol{\mu}}^{(k)}, \tilde{\boldsymbol{\lambda}}^{(k)})$ of the generally nonconvex subproblem $\mathcal{P}_{sub, i}^{dd}$. 
		\STATE Obtain the corresponding Lagrange multiplier $\boldsymbol{\mu}_i(\tilde{\boldsymbol{\mu}}^{(k)}, \tilde{\boldsymbol{\lambda}}^{(k)})$.
		\STATE Compute $\tilde{\mathbf{g}}_{i}(\mathbf{x}_{i}^{\dagger}(\tilde{\boldsymbol{\mu}}^{(k)}, \tilde{\boldsymbol{\lambda}}^{(k)}))$ and $\tilde{\mathbf{h}}_{i}(\mathbf{x}_{i}^{\dagger}(\tilde{\boldsymbol{\mu}}^{(k)}, \tilde{\boldsymbol{\lambda}}^{(k)}))$.
		\ENDFOR
		\STATE Compute $\nabla_{\tilde{\boldsymbol{\mu}}} q^{\dag}(\tilde{\boldsymbol{\mu}}^{(k)}, \tilde{\boldsymbol{\lambda}}^{(k)})$ and $\nabla_{\tilde{\boldsymbol{\lambda}}} q^{\dag}(\tilde{\boldsymbol{\mu}}^{(k)}, \tilde{\boldsymbol{\lambda}}^{(k)})$ according to (\ref{dual mas-obj-grad-mu}) and (\ref{dual mas-obj-grad-lambda}), respectively.
		\STATE Compute $\tilde{\boldsymbol{\mu}}^{*}(\tilde{\boldsymbol{\mu}}^{(k)}, \tilde{\boldsymbol{\lambda}}^{(k)})$ and $\tilde{\boldsymbol{\lambda}}^{*}(\tilde{\boldsymbol{\mu}}^{(k)}, \tilde{\boldsymbol{\lambda}}^{(k)})$
		according to (\ref{dual mas-appro-appro-solution-mu}) and (\ref{dual mas-appro-appro-solution-lambda}), respectively.
		\STATE Update $\tilde{\boldsymbol{\mu}}$ and  $\tilde{\boldsymbol{\lambda}}$ according to (\ref{dual smooth-mu}) and (\ref{dual smooth-lambda}), respectively.
		\STATE Set $k \gets k+1$.
		\UNTIL{Some termination criterion is met.}
	\end{algorithmic}
	\label{alg1}
\end{algorithm}

In the second part, we focus on solving the master problem $\mathcal{P}^{dd, \dag}_{mas}$ by the SCA-based algorithm \cite{scutari2018parallel}. 
First, we compute $\nabla_{\tilde{\boldsymbol{\mu}}}q^{\dag}(\tilde{\boldsymbol{\mu}}^{(k)}, \tilde{\boldsymbol{\lambda}}^{(k)})$ and $\nabla_{\tilde{\boldsymbol{\lambda}}}q^{\dag}(\tilde{\boldsymbol{\mu}}^{(k)}, \tilde{\boldsymbol{\lambda}}^{(k)})$ according to \eqref{dual mas-obj-grad-mu} and \eqref{dual mas-obj-grad-lambda}, respectively, where $\mathbf{X}^{\dag}_i(\tilde{\boldsymbol{\mu}}^{(k)}, \tilde{\boldsymbol{\lambda}}^{(k)})=\mathbf{x}^{\dag}_i(\tilde{\boldsymbol{\mu}}^{(k)}, \tilde{\boldsymbol{\lambda}}^{(k)})$, and choose the following approximate function of $q^{\dag}$ at $(\tilde{\boldsymbol{\mu}}^{(k)}, \tilde{\boldsymbol{\lambda}}^{(k)})$ (i.e., the second-order Taylor approximation of $q^{\dag}$ near $(\tilde{\boldsymbol{\mu}}^{(k)}, \tilde{\boldsymbol{\lambda}}^{(k)})$) \cite{scutari2018parallel}:
\begin{equation*}\label{dual mas-appro-appro-object}
	\begin{aligned}
		&Q^{\dag}(\tilde{\boldsymbol{\mu}}, \tilde{\boldsymbol{\lambda}}; \tilde{\boldsymbol{\mu}}^{(k)}, \tilde{\boldsymbol{\lambda}}^{(k)})\triangleq\\
		&-\frac{\tau_{\tilde{\boldsymbol{\mu}}}}{2}\|\tilde{\boldsymbol{\mu}}-\tilde{\boldsymbol{\mu}}^{(k)}\|_{2}^2 + 		\nabla_{\tilde{\boldsymbol{\mu}}} q^{\dag}(\tilde{\boldsymbol{\mu}}^{(k)}, \tilde{\boldsymbol{\lambda}}^{(k)})^{T}(\tilde{\boldsymbol{\mu}}-\tilde{\boldsymbol{\mu}}^{(k)})\\
		&-\frac{\tau_{\tilde{\boldsymbol{\lambda}}}}{2}\|\tilde{\boldsymbol{\lambda}}- \tilde{\boldsymbol{\lambda}}^{(k)}\|_{2}^2 + 	\nabla_{\tilde{\boldsymbol{\lambda}}} q^{\dag}(\tilde{\boldsymbol{\mu}}^{(k)}, \tilde{\boldsymbol{\lambda}}^{(k)})^{T}(\tilde{\boldsymbol{\lambda}}- \tilde{\boldsymbol{\lambda}}^{(k)}),
	\end{aligned}
\end{equation*}
with $\tau_{\tilde{\boldsymbol{\mu}}},\tau_{\tilde{\boldsymbol{\lambda}}}> 0$.
Note that $-Q^{\dagger}$ satisfies Assumption~\ref{SCA-assump-surrogate-object}.
The corresponding approximate problem of the master problem $\mathcal{P}^{dd, \dag}_{mas}$ is given by:
\begin{equation*}\label{dual mas-appro-appro}
	\begin{aligned}
		\mathcal{P}^{dd,\dag,(k)}_{mas}:\ \max_{\tilde{\boldsymbol{\mu}},\tilde{\boldsymbol{\lambda}}} \ &Q^{\dag}(\tilde{\boldsymbol{\mu}}, \tilde{\boldsymbol{\lambda}}; \tilde{\boldsymbol{\mu}}^{(k)}, \tilde{\boldsymbol{\lambda}}^{(k)})\\  
		s.t. \ 
		&\text{\eqref{dual mas-decoup}}. 
	\end{aligned} 
\end{equation*}
This problem is strongly convex since $Q^{\dagger}$ is strongly concave. 
Then, by the KKT conditions, we obtain the unique optimal point of Problem $\mathcal{P}^{dd,\dag,(k)}_{mas}$ in closed form:
\begin{align}
	&\hspace{-0.2cm}\tilde{\boldsymbol{\mu}}^{*}(\tilde{\boldsymbol{\mu}}^{(k)}, \tilde{\boldsymbol{\lambda}}^{(k)})= \left[\tilde{\boldsymbol{\mu}}^{(k)}+\frac{ \nabla_{\tilde{\boldsymbol{\mu}}}q^{\dag}(\tilde{\boldsymbol{\mu}}^{(k)}, \tilde{\boldsymbol{\lambda}}^{(k)})}{\tau_{\tilde{\boldsymbol{\mu}}}}\right]_{+}, \label{dual mas-appro-appro-solution-mu}\\
	&\hspace{-0.2cm}\tilde{\boldsymbol{\lambda}}^{*}(\tilde{\boldsymbol{\mu}}^{(k)}, \tilde{\boldsymbol{\lambda}}^{(k)})=\tilde{\boldsymbol{\lambda}}^{(k)}+\frac{\nabla_{\tilde{\boldsymbol{\lambda}}}q^{\dag}(\tilde{\boldsymbol{\mu}}^{(k)}, \tilde{\boldsymbol{\lambda}}^{(k)})}{\tau_{\tilde{\boldsymbol{\lambda}}}}. \label{dual mas-appro-appro-solution-lambda}
\end{align}
Next, 
the dual variables  
$\tilde{\boldsymbol{\mu}}$ and  $\tilde{\boldsymbol{\lambda}}$ are updated according to: 
\begin{align}
	\tilde{\boldsymbol{\mu}}^{(k+1)}=\tilde{\boldsymbol{\mu}}^{(k)}+\gamma^{(k)}(\tilde{\boldsymbol{\mu}}^{*}(\tilde{\boldsymbol{\mu}}^{(k)}, \tilde{\boldsymbol{\lambda}}^{(k)}) - \tilde{\boldsymbol{\mu}}^{(k)}),\label{dual smooth-mu}\\
	\tilde{\boldsymbol{\lambda}}^{(k+1)}=\tilde{\boldsymbol{\lambda}}^{(k)}+\gamma^{(k)}(\tilde{\boldsymbol{\lambda}}^{*}(\tilde{\boldsymbol{\mu}}^{(k)}, \tilde{\boldsymbol{\lambda}}^{(k)}) - \tilde{\boldsymbol{\lambda}}^{(k)}),\label{dual smooth-lambda}
\end{align}
where $\gamma^{(k)}$ is a step size satisfying \eqref{sca step size}.
The detailed procedure is summarized in Algorithm~\ref{dual alg:alg1}.\footnote{Similar to \cite{boyd2007notes,palomar2006tutorial},   $\mathbf{x}^{\dagger}(\tilde{\boldsymbol{\mu}}^{(k)},\tilde{\boldsymbol{\lambda}}^{(k)})$ obtained in Step~4 may be infeasible, 
i.e., $\sum_{i\in\mathcal{I}}\tilde{\mathbf{g}}_i(\mathbf{x}_{i}^{\dagger}(\tilde{\boldsymbol{\mu}}^{(k)}, \tilde{\boldsymbol{\lambda}}^{(k)}))\!\npreceq\!\mathbf{0}$~or~$\sum_{i\in\mathcal{I}}\tilde{\mathbf{h}}_i(\mathbf{x}_{i}^{\dagger}(\tilde{\boldsymbol{\mu}}^{(k)}, \tilde{\boldsymbol{\lambda}}^{(k)}))\!\neq\!\mathbf{0}$ for some $k$, 
and the feasibility is only guaranteed at the limit point of $\{\mathbf{x}^{\dagger}(\tilde{\boldsymbol{\mu}}^{(k)},\tilde{\boldsymbol{\lambda}}^{(k)})\}_{k\in\mathbb{N}}$. In Steps~8-10, parallel computations can be applied for conducting projections, matrix (vector) multiplications and additions.}

To prove the local convergence of Algorithm~\ref{dual alg:alg1} under Assumptions~\ref{dual assump-p} and \ref{dual assump-stationary-2}, we make the following assumption on $\nabla q^{\dag}(\tilde{\boldsymbol{\mu}}, \tilde{\boldsymbol{\lambda}})$ which exists by Theorem~\ref{dual thm1}.1.

\begin{assumption}\label{dual assump-mas-appro-object}
	$q^{\dag}(\tilde{\boldsymbol{\mu}}, \tilde{\boldsymbol{\lambda}})$ is bounded upper and the gradient $\nabla q^{\dag}(\tilde{\boldsymbol{\mu}}, \tilde{\boldsymbol{\lambda}})$ is Lipschitz continuous on a neighborhood of $(\tilde{\boldsymbol{\mu}}^{\ddagger}, \tilde{\boldsymbol{\lambda}}^{\ddagger})$ with $(\tilde{\boldsymbol{\mu}}^{\ddagger}, \tilde{\boldsymbol{\lambda}}^{\ddagger})$ specified in Assumption~\ref{dual assump-stationary-2}.2.
\end{assumption}

Then, based on Theorem~\ref{dual thm1} and \cite[Theorem~3.8]{scutari2018parallel}, we can show the local convergence of Algorithm~\ref{dual alg:alg1}.

\begin{theorem}\label{dual thm2}
	Suppose that Assumptions~\ref{regularity}, \ref{dual assump-p}, \ref{dual assump-stationary-2}, and \ref{dual assump-mas-appro-object} are satisfied. 
	The points $\tilde{\boldsymbol{\mu}}^{\ddagger}$, $\tilde{\boldsymbol{\lambda}}^{\ddagger}$, $\mathbf{x}_{i}^{\ddagger}$,  $\boldsymbol{\mu}_{i}^{\ddagger}$, $i\in\mathcal{I}$ are specified in Assumption~\ref{dual assump-stationary-2}.2.
	The sequences $\{(\tilde{\boldsymbol{\mu}}^{(k)}, \tilde{\boldsymbol{\lambda}}^{(k)})\}_{k\in\mathbb{N}}$ $\{\mathbf{x}^{\dagger}_{i}(\tilde{\boldsymbol{\mu}}^{(k)}, \tilde{\boldsymbol{\lambda}}^{(k)})\}_{k\in\mathbb{N}}$, $\{\boldsymbol{\mu}_i(\tilde{\boldsymbol{\mu}}^{(k)}, \tilde{\boldsymbol{\lambda}}^{(k)})\}_{k\in\mathbb{N}}$, $i\in\mathcal{I}$, are generated by Algorithm~\ref{dual alg:alg1}.
	Then, there exists a neighborhood $\mathcal{N}_{(\tilde{\boldsymbol{\mu}}^{\ddagger}, \tilde{\boldsymbol{\lambda}}^{\ddagger})}$ of $(\tilde{\boldsymbol{\mu}}^{\ddagger}, \tilde{\boldsymbol{\lambda}}^{\ddagger})$ and neighborhoods $\mathcal{N}_{\mathbf{x}_{i}^{\ddagger}}$,
	$\mathcal{N}_{\boldsymbol{\mu}_{i}^{\ddagger}}$, $i\in\mathcal{I}$ of $\mathbf{x}_{i}^{\ddagger}$, $\boldsymbol{\mu}_{i}^{\ddagger}$, $i\in\mathcal{I}$, respectively, such that if $\{(\tilde{\boldsymbol{\mu}}^{(k)}, \tilde{\boldsymbol{\lambda}}^{(k)})\}_{k\in\mathbb{N}}\subseteq\mathcal{N}_{(\tilde{\boldsymbol{\mu}}^{\ddagger}, \tilde{\boldsymbol{\lambda}}^{\ddagger})}$, $\{\mathbf{x}^{\dagger}_{i}(\tilde{\boldsymbol{\mu}}^{(k)}, \tilde{\boldsymbol{\lambda}}^{(k)})\}_{k\in\mathbb{N}}\subseteq\mathcal{N}_{\mathbf{x}_{i}^{\ddagger}}$,  $\{\boldsymbol{\mu}_i(\tilde{\boldsymbol{\mu}}^{(k)}, \tilde{\boldsymbol{\lambda}}^{(k)})\}_{k\in\mathbb{N}}\subseteq\mathcal{N}_{\boldsymbol{\mu}_{i}^{\ddagger}}$, $i \in\mathcal{I}$,
	then the following results hold:
	\begin{enumerate}
		\item The limit point of any convergent subsequence (at least one exists) $\{(\tilde{\boldsymbol{\mu}}^{(k)}, \tilde{\boldsymbol{\lambda}}^{(k)})\}_{k\in\mathcal{K}}$ with $\mathcal{K}\subseteq\mathbb{N}$, denoted by $(\tilde{\boldsymbol{\mu}}^{(\infty)}, \tilde{\boldsymbol{\lambda}}^{(\infty)})$, is a stationary point of the master problem~$\mathcal{P}^{dd, \dag}_{mas}$;
		
		\item For all $i\in\mathcal{I}$, the subsequence $\{\mathbf{x}_{i}^{\dag}(\tilde{\boldsymbol{\mu}}^{(k)}, \tilde{\boldsymbol{\lambda}}^{(k)})\}_{k\in\mathcal{K}}$ with $\mathcal{K}$ specified in 1) converges, and $\mathbf{x}^{(\infty)}=(\mathbf{x}^{(\infty)}_{1}, \cdots, \mathbf{x}^{(\infty)}_{I})$ is a stationary point of Problem $\mathcal{P}^{dd}$, where $\mathbf{x}^{(\infty)}_{i}\in \operatorname{int}(\mathcal{X}_{i})$ is the limit point of $\{\mathbf{x}_{i}^{\dag}(\tilde{\boldsymbol{\mu}}^{(k)}, \tilde{\boldsymbol{\lambda}}^{(k)})\}_{k\in\mathcal{K}}$.
	\end{enumerate}
\end{theorem}
\begin{IEEEproof}
	The proof is similar to that of Theorem~\ref{primal thm2}.
\end{IEEEproof}

\section{Successive Dual Decomposition}\label{sd section}
In this section, we consider a nonconvex problem with coupling constraints and present a successive dual decomposition method and a corresponding algorithm for this problem, referred to as SDD-M and \mbox{SDD-A}, respectively.

\subsection{Problem Formulation}
Consider the following nonconvex problem:
\begin{align}
	\mathcal{P}^{sdd}:\ \min_{\mathbf{x}} \ &f(\mathbf{x})\triangleq \sum_{i\in\mathcal{I}}f_i(\mathbf{x}_i)\nonumber\\  
	s.t. \ 
	&\sum_{i\in \mathcal{I}}\mathbf{A}_{i}\mathbf{x}_{i}+\mathbf{b}=\mathbf{0}, \label{sd coup-eq-cons}\\
	&\text{\eqref{dual coup-ineq-cons}, \eqref{dual decoup-ineq-cons}, \eqref{dual decoup-cvx-cons}},    \nonumber 
\end{align}
where $\mathbf{x}$, $f_i$, $\tilde{\mathbf{g}}_i$, and $\mathbf{g}_i$, $i\in\mathcal{I}$ are the same as in Problem $\mathcal{P}^{dd}$,  $\mathbf{A}_{i}\in\mathbb{R}^{\tilde{m}\times n_{i}}$, $i\in\mathcal{I}$, and $\mathbf{b}\in\mathbb{R}^{\tilde{m}}$.
In contrast, the coupling linear equality constraint of Problem $\mathcal{P}^{sdd}$ in \eqref{sd coup-eq-cons} is less general than the coupling (generally) nonlinear equality constraint of Problem $\mathcal{P}^{dd}$ in \eqref{dual coup-eq-cons}.
Thus, Problem $\mathcal{P}^{sdd}$ is a special case of Problem $\mathcal{P}^{dd}$. 
Accordingly, by removing  4) in Assumption~\ref{dual assump-p}, we have the following assumptions on Problem $\mathcal{P}^{sdd}$.

\begin{assumption}[Assumptions on Problem $\mathcal{P}^{sdd}$]\label{sd p}
	Assumptions~\ref{dual assump-p}.1-\ref{dual assump-p}.3 and \ref{dual assump-p}.5 hold true.
\end{assumption}

As discussed in Section \ref{dual section-problem}, the basic dual decomposition method and algorithm for convex problems in \cite{boyd2007notes,palomar2006tutorial} also cannot produce a stationary point of nonconvex Problem $\mathcal{P}^{sdd}$.

\subsection{SDD-M and Its Theoretical Analysis}\label{sd section-method}
Now, we present SDD-M. 
Specifically, first, we approximate nonconvex Problem $\mathcal{P}^{sdd}$ with a sequence of successively refined convex approximate problems $\mathcal{P}^{sdd,(k)}, k\in\mathbb{N}$.
Each convex problem is obtained by approximating the objective function and constraints of Problem $\mathcal{P}^{sdd}$ at $\mathbf{x}^{(k)}$:
\begin{align}
		\mathcal{P}^{sdd,(k)}:\ 
		\min_{\mathbf{x}} \ &F(\mathbf{x}; \mathbf{x}^{(k)})\triangleq
		\sum_{i\in\mathcal{I}}F_i(\mathbf{x}_i; \mathbf{x}_i^{(k)})\nonumber\\  
		s.t. \ 
		&\sum_{i\in\mathcal{I}}\widetilde{\mathbf{G}}_i(\mathbf{x}_i; \mathbf{x}_{i}^{(k)})\preceq \mathbf{0}, \label{sd coup-ineq}\\
		&\mathbf{G}_{i}(\mathbf{x}_{i}; \mathbf{x}_{i}^{(k)})\preceq \mathbf{0}, \ i\in\mathcal{I},  \\
		&\text{\eqref{sd coup-eq-cons}, \eqref{dual decoup-cvx-cons},}\nonumber        
	\end{align}  
where $F_i$ is a strongly convex approximation of $f_i$, satisfying Assumption~\ref{SCA-assump-surrogate-object},
and $\widetilde{\mathbf{G}}_i$ and $\mathbf{G}_{i}$ are convex approximations of $\tilde{\mathbf{g}}_i$ and $\mathbf{g}_{i}$, respectively, satisfying Assumption \ref{SCA-assump-surrogate-constriant}.

Then, we apply the basic dual decomposition method for convex problems \cite{boyd2007notes,palomar2006tutorial} to Problem~$\mathcal{P}^{sdd,(k)}$. 
Specifically, we form the partial Lagrangian: 
$L(\mathbf{x},\tilde{\boldsymbol{\mu}},\tilde{\boldsymbol{\lambda}};\mathbf{x}^{(k)})\triangleq\sum_{i\in \mathcal{I}}L_{i}(\mathbf{x}_{i},\tilde{\boldsymbol{\mu}},\tilde{\boldsymbol{\lambda}};\mathbf{x}_{i}^{(k)})
+ \tilde{\boldsymbol{\lambda}}^T\mathbf{b}$, where $L_{i}(\mathbf{x}_{i},\tilde{\boldsymbol{\mu}},\tilde{\boldsymbol{\lambda}};\mathbf{x}_{i}^{(k)})\triangleq F_{i}(\mathbf{x}_{i}; \mathbf{x}_{i}^{(k)})+\tilde{\boldsymbol{\mu}}^T\widetilde{\mathbf{G}}_i(\mathbf{x}_i; \mathbf{x}_{i}^{(k)})+ \tilde{\boldsymbol{\lambda}}^T\mathbf{A}_i\mathbf{x}_i$, and the weights $\tilde{\boldsymbol{\mu}}$ and $\tilde{\boldsymbol{\lambda}}$ are the Lagrange multipliers associated with the constraints in \eqref{sd coup-ineq} and \eqref{sd coup-eq-cons}, respectively.
Then, the problem for obtaining the dual function, i.e.,
$q(\tilde{\boldsymbol{\mu}}, \tilde{\boldsymbol{\lambda}}; \mathbf{x}^{(k)})=\inf\{L(\mathbf{x},\tilde{\boldsymbol{\mu}},\tilde{\boldsymbol{\lambda}};\mathbf{x}^{(k)})\ | \ \mathbf{G}_{i}(\mathbf{x}_{i}; \mathbf{x}_{i}^{(k)})\preceq \mathbf{0}, \mathbf{x}_{i}\in\mathcal{X}_i, i\in\mathcal{I}\}
$, can be separated into $I$ (dual) subproblems, one for each $i\in\mathcal{I}$: 
	\begin{align}
		\mathcal{P}^{sdd, (k)}_{sub, i}:\ \min_{\mathbf{x}_{i}} \ &L_{i}(\mathbf{x}_{i},\tilde{\boldsymbol{\mu}},\tilde{\boldsymbol{\lambda}};\mathbf{x}_{i}^{(k)})\nonumber\\
		s.t. \ 
		&\mathbf{G}_{i}(\mathbf{x}_{i}; \mathbf{x}_{i}^{(k)})\preceq \mathbf{0},  \\
		&\mathbf{x}_{i}\in \mathcal{X}_i.      
	\end{align}  
Then, the  master (dual) problem is as follows:
	\begin{align}
		\mathcal{P}^{sdd,(k)}_{mas}: \max_{\tilde{\boldsymbol{\mu}}, \tilde{\boldsymbol{\lambda}}} \ &q(\tilde{\boldsymbol{\mu}}, \tilde{\boldsymbol{\lambda}}; \mathbf{x}^{(k)})
		\!=\!\!\sum_{i\in\mathcal{I}}L_i(\mathbf{x}^{*,(k)}_i\!(\tilde{\boldsymbol{\mu}}, \tilde{\boldsymbol{\lambda}}), \tilde{\boldsymbol{\mu}}, \tilde{\boldsymbol{\lambda}}; \mathbf{x}_{i}^{(k)}\!)\nonumber\\
		&\ \qquad\qquad\qquad +\tilde{\boldsymbol{\lambda}}^T\mathbf{b}\nonumber\\
		s.t. \ 
		&\text{\eqref{dual mas-decoup}},\nonumber
	\end{align}  
where $\mathbf{x}^{*,(k)}_i(\tilde{\boldsymbol{\mu}}, \tilde{\boldsymbol{\lambda}})$ represents the unique optimal point of the convex subproblem $\mathcal{P}_{sub, i}^{sdd, (k)}$ for $(\tilde{\boldsymbol{\mu}}, \tilde{\boldsymbol{\lambda}})\in\mathbb{R}_{+}^{\tilde{r}}\times\mathbb{R}^{\tilde{m}}$. 
Let $(\tilde{\boldsymbol{\mu}}^{*, (k)}, \tilde{\boldsymbol{\lambda}}^{*, (k)})$ denote an optimal point of the convex master problem $\mathcal{P}^{sdd, (k)}_{mas}$.
Notably, under Assumption~\ref{regularity}, $(\tilde{\boldsymbol{\mu}}^{*, (k)}, \tilde{\boldsymbol{\lambda}}^{*, (k)})$ exists, and strong duality holds for convex Problem $\mathcal{P}^{sdd, (k)}$ \cite{scutari2017parallel}.
Thus, $\mathbf{x}^{*,(k)}(\tilde{\boldsymbol{\mu}}^{*, (k)}, \tilde{\boldsymbol{\lambda}}^{*, (k)})$ is the unique optimal point of convex Problem $\mathcal{P}^{sdd, (k)}$.

Next, taking a step from the given $\mathbf{x}^{(k)}$ towards $\mathbf{x}^{*,(k)}(\tilde{\boldsymbol{\mu}}^{*, (k)}, \tilde{\boldsymbol{\lambda}}^{*, (k)})$, we set:
\begin{equation}\label{sd smooth}
	\begin{aligned}
		\mathbf{x}^{(k+1)}
		=\mathbf{x}^{(k)}+\gamma^{(k)}\left( \mathbf{x}^{*,(k)}(\tilde{\boldsymbol{\mu}}^{*, (k)}, \tilde{\boldsymbol{\lambda}}^{*, (k)})-\mathbf{x}^{(k)}\right),
	\end{aligned}
\end{equation}
where $\gamma^{(k)}$ is a step size satisfying \eqref{sca step size}.
The proposed \mbox{SDD-M} for Problem $\mathcal{P}^{sdd}$ is given by the convex approximate problems $\mathcal{P}^{sdd, (k)}$,  the subproblems $\mathcal{P}^{sdd, (k)}_{sub,i}$, $i\in\mathcal{I}$, the master problems $\mathcal{P}^{sdd, (k)}_{mas}$, and the updates in \eqref{sd smooth}, $k\in\mathbb{N}$.

The effectiveness of SDD-M can be easily concluded based on \cite[Theorem~2]{scutari2017parallel}.

\begin{statement}\label{sd thm-1}
	If Assumptions~\ref{regularity}-\ref{SCA-assump-surrogate-constriant} and \ref{sd p} are satisfied, then at least one limit point of $\{\mathbf{x}^{(k)}\}_{k\in\mathbb{N}}$ generated by SDD-M is a stationary point of Problem $\mathcal{P}^{sdd}$.
\end{statement}

\begin{algorithm}[t]
	\caption{Dual Decomposition Algorithm \cite{boyd2007notes,palomar2006tutorial} for Problem $\mathcal{P}^{sdd,(k)}$}\label{sd alg:alg1-DD}
	\begin{algorithmic}[1]\label{alg1}\small
		\STATE \textbf{initialization}: 
		Set $t=0$ and $(\tilde{\boldsymbol{\mu}}^{(0)}, \tilde{\boldsymbol{\lambda}}^{(0)})=(\tilde{\boldsymbol{\mu}}^{*, (k)}, \tilde{\boldsymbol{\lambda}}^{*, (k)})$; 
		choose $\{\gamma_{in}^{(t)}\}_{t\in\mathbb{N}}\subseteq(0,+\infty)$ satisfying  $\sum_{t=0}^{\infty}\gamma_{in}^{(t)}=\infty$ and  $\sum_{t=0}^{\infty}(\gamma_{in}^{(t)})^2<\infty$.
		\REPEAT
		\FOR{all $i\in\mathcal{I}$}
		\STATE Get the unique optimal point $\mathbf{x}_{i}^{*,(k)}(\tilde{\boldsymbol{\mu}}^{(t)}, \tilde{\boldsymbol{\lambda}}^{(t)})$ of the convex subproblem $\mathcal{P}^{sdd,(k)}_{sub, i}$.
		\STATE Compute $\widetilde{\mathbf{G}}_i(\mathbf{x}_{i}^{*,(k)}(\tilde{\boldsymbol{\mu}}^{(t)}, \tilde{\boldsymbol{\lambda}}^{(t)}))$ and $\mathbf{A}_i\mathbf{x}^{*,(k)}_i(\tilde{\boldsymbol{\mu}}^{(t)}, \tilde{\boldsymbol{\lambda}}^{(t)})$.
		\ENDFOR
		\STATE Compute~$\nabla_{\tilde{\boldsymbol{\mu}}}q(\tilde{\boldsymbol{\mu}}^{(t)}, \tilde{\boldsymbol{\lambda}}^{(t)}; \mathbf{x}^{(k)})$~and~$\nabla_{\tilde{\boldsymbol{\lambda}}}q(\tilde{\boldsymbol{\mu}}^{(t)}, \tilde{\boldsymbol{\lambda}}^{(t)}; \mathbf{x}^{(k)})$~by:
		\begin{align*}
			&\nabla_{\tilde{\boldsymbol{\mu}}}q(\tilde{\boldsymbol{\mu}}^{(t)}, \tilde{\boldsymbol{\lambda}}^{(t)}; \mathbf{x}^{(k)})\!=\! \sum_{i\in\mathcal{I}}\widetilde{\mathbf{G}}_i(\mathbf{x}_{i}^{*,(k)}(\tilde{\boldsymbol{\mu}}^{(t)}, \tilde{\boldsymbol{\lambda}}^{(t)}); \mathbf{x}_{i}^{(k)}),\\
			&\nabla_{\tilde{\boldsymbol{\lambda}}}q(\tilde{\boldsymbol{\mu}}^{(t)}, \tilde{\boldsymbol{\lambda}}^{(t)}; \mathbf{x}^{(k)})\!=\! \sum_{i\in\mathcal{I}}\mathbf{A}_i\mathbf{x}^{*,(k)}_i(\tilde{\boldsymbol{\mu}}^{(t)}, \tilde{\boldsymbol{\lambda}}^{(t)})+\mathbf{b}.
		\end{align*}
		\STATE Update $\tilde{\boldsymbol{\mu}}$ and $\tilde{\boldsymbol{\lambda}}$ according to:
		\begin{align*}
			&\tilde{\boldsymbol{\mu}}^{(t+1)}= \left[\tilde{\boldsymbol{\mu}}^{(t)} + \gamma_{in}^{(t)} \nabla_{\tilde{\boldsymbol{\mu}}}q(\tilde{\boldsymbol{\mu}}^{(t)}, \tilde{\boldsymbol{\lambda}}^{(t)}; \mathbf{x}^{(k)})\right]_{+},\\
			&\tilde{\boldsymbol{\lambda}}^{(t+1)}= \tilde{\boldsymbol{\lambda}}^{(t)} + \gamma_{in}^{(t)} \nabla_{\tilde{\boldsymbol{\lambda}}}q(\tilde{\boldsymbol{\mu}}^{(t)}, \tilde{\boldsymbol{\lambda}}^{(t)}; \mathbf{x}^{(k)}).
		\end{align*}
		\STATE Set $t \gets t+1$.
		\UNTIL{Some termination criterion is met.}
	\end{algorithmic}
	\label{alg1}
\end{algorithm}

\begin{algorithm}[t]
	\caption{SDD-A}\label{sd alg:alg1-SCA}
	\begin{algorithmic}[1]\label{alg1}\small
		\STATE \textbf{initialization}: Set $k=0$ and choose any feasible point $\mathbf{x}^{(0)}$ and $\{\gamma^{(k)}\}_{k\in\mathbb{N}}\subseteq(0,1]$ satisfying (\ref{sca step size}).
		\REPEAT
		\STATE Obtain the unique optimal point $\mathbf{x}^{*,(k)}(\tilde{\boldsymbol{\mu}}^{*, (k)}, \tilde{\boldsymbol{\lambda}}^{*, (k)})$ of the convex approximate problem $\mathcal{P}^{sdd,(k)}$ by Algorithm~\ref{sd alg:alg1-DD}.
		\STATE Update $\mathbf{x}$ by (\ref{sd smooth}).
		\STATE Set $k \gets k+1$.
		\UNTIL{Some termination criterion is met.}
	\end{algorithmic}
	\label{alg1}
\end{algorithm}

\subsection{SDD-A and Its Convergence Analysis}\label{sd section-alg}
Different from SPD-M in Section~\ref{sp section-method}, $q$ is differentiable \cite[Proposition~B.22]{bertsekas2016nonlinear}, and both its value and gradient are accessible.
Then, we present SDD-A for Problem $\mathcal{P}^{sdd}$. 
Specifically, first, for fixed $k$, we adopt the basic dual decomposition algorithm for convex problems \cite{boyd2007notes,palomar2006tutorial} to solve Problem $\mathcal{P}^{sdd,(k)}$. 
The detailed procedure is summarized in Algorithm~\ref{sd alg:alg1-DD}.\footnote{In Steps~5, 7, and 8, parallel computations can be applied for conducting projections, matrix (vector) multiplications and additions.}
In Step 4, we can solve the convex subproblem $\mathcal{P}^{sdd,(k)}_{sub, i}$ using algorithms for convex problems, e.g., interior-point methods. 
Then, we present the complete SDD-A in Algorithm~\ref{sd alg:alg1-SCA}.
Notably, SDD-A slightly extends SCADD \cite{scutari2017parallel} to nonconvex problems with coupling linear equality constraints.

Based on the convergence results of the projected gradient algorithm in \cite[Proposition 8.2.6]{bertsekas2003convex} and Statement~\ref{sd thm-1}, we can give the convergence of Algorithm~\ref{sd alg:alg1-SCA}.
\begin{statement}\label{sd thm-2}
	If Assumptions~\ref{regularity}-\ref{SCA-assump-surrogate-constriant} and \ref{sd p} are satisfied, then at least one limit point of $\{\mathbf{x}^{(k)}\}_{k\in\mathbb{N}}$ generated by Algorithm~\ref{sd alg:alg1-SCA} is a stationary point of Problem $\mathcal{P}^{sdd}$.
\end{statement}

\section{Comparisons}\label{sec compare}
\begin{table}[t]
	\centering 
	\caption{Comparison of the proposed methods and algorithms.}
	\begin{tabular}{|m{2.25cm}<{\centering}|m{2.4cm}<{\centering}|m{2.8cm}<{\centering}|} 
		\hline
		{} & Decomposition & Successive decomposition \\ 
		\hline
		Equality constraint & Generally nonlinear &  Linear  \\ 
		\hline	
		Decomposition principle  & Directly decompose the original nonconvex problems or the dual problems & Successively approximate the original nonconvex problems into convex ones and apply standard decomposition methods for convex problems\\ 
		\hline	
		Algorithm structure  & Possibly single-loop & Inherently double-loop \\ 
		\hline	
		Assumption  & Difficult to check & Easy to satisfy \\ 
		\hline	
	\end{tabular}
	\label{table comp}
\end{table}

First, we compare the proposed decomposition methods and algorithms (i.e., \mbox{PD-M/A} and \mbox{DD-M/A}) with the successive decomposition methods and algorithms (i.e., \mbox{SPD-M/A} and \mbox{SDD-M/A}) for nonconvex problems, as Table~\ref{table comp} presents. 
\begin{enumerate}
	\item \textbf{Equality constraint:} \mbox{PD-M/A} and \mbox{DD-M/A} can deal with nonconvex problems with nonlinear equality constraints, i.e., Problems $\mathcal{P}^{pd}$ and $\mathcal{P}^{dd}$, whereas \mbox{SPD-M/A} and \mbox{SDD-M/A} can only handle nonconvex problems with linear equality constraints, i.e., Problems $\mathcal{P}^{spd}$ and $\mathcal{P}^{sdd}$, due to the limitations of the SCA-based algorithm \cite{scutari2017parallel}.
	
	\item \textbf{Decomposition principle:} \mbox{PD-M} and \mbox{DD-M} directly decompose the original nonconvex problem $\mathcal{P}^{pd}$ and the dual problem of the original nonconvex problem $\mathcal{P}^{dd}$, respectively. 
	In contrast, \mbox{SPD-M} and \mbox{SDD-M} first successively approximate the original nonconvex problems $\mathcal{P}^{spd}$ and $\mathcal{P}^{sdd}$ into sequences of convex approximate problems $\mathcal{P}^{spd,(k)}$ and $\mathcal{P}^{sdd,(k)}$, $k\in\mathbb{N}$ (possibly destroying original decomposition structures) and then apply standard primal and dual decomposition methods for convex problems \cite{boyd2007notes,palomar2006tutorial} to these convex approximate problems, respectively. 
	
	\item \textbf{Algorithm structure:} \mbox{PD-M} and \mbox{DD-M} can be single-loop iterative algorithms when the stationary points of all subproblems can be obtained in closed-form, whereas \mbox{SPD-M} and \mbox{SDD-M} are inherently double-loop iterative algorithms with an outer loop for successive convex approximation and an inner loop for decomposition.
	
	\item \textbf{Assumption:} \mbox{PD-M/A} and \mbox{DD-M/A} require some strong assumptions (e.g., Assumptions~\ref{primal assump-stationary-2} and \ref{dual assump-stationary-2}) that are relatively difficult to check, whereas \mbox{SPD-M/A} and \mbox{SDD-M/A} require assumptions that are usually easy to~satisfy. 
\end{enumerate}
Therefore, \mbox{PD-M/A} and \mbox{DD-M/A} apply to a broader range of nonconvex problems, exploit more decomposition structures, and yield possibly simpler algorithm structures but require stronger conditions than \mbox{SPD-M/A} and \mbox{SDD-M/A}. 
For commonly applicable nonconvex problems, the proposed four algorithms exhibit their own gains in respective nonconvex problems, increasing the chance of achieving a better tradeoff between convergence performance and computation time.

Next, we illustrate the connections between the proposed methods and algorithms (i.e., PD-M/A, SPD-M/A, DD-M/A, and SDD-M/A) for nonconvex problems and the basic decomposition methods and algorithms for convex problems \cite{boyd2007notes,palomar2006tutorial}. 
Specifically, \mbox{PD-M} and \mbox{DD-M} can reduce to the basic ones for convex problems \cite{boyd2007notes,palomar2006tutorial}, and \mbox{PD-A} and \mbox{DD-A} extend the basic subgradient-based ones for convex problems \cite{boyd2007notes,palomar2006tutorial}.
Moreover, \mbox{SPD-M/A} and \mbox{SDD-M/A} differentiate from the basic ones for convex problems \cite{boyd2007notes,palomar2006tutorial} offering new decomposition methods and algorithms for convex problems. 
Notice that when nonconvex Problems $\mathcal{P}^{pd}$, $\mathcal{P}^{spd}$, $\mathcal{P}^{dd}$, and $\mathcal{P}^{sdd}$ reduce to convex problems, the proposed methods and algorithms naturally produce optimal points.

\section{Extensions}\label{section extension}
In this section, we extend the proposed methods and algorithms (i.e., PD-M/A, SPD-M/A, DD-M/A, and SDD-M/A) to indirect and two-level decomposition methods and algorithnms.

\subsection{Indirect Decomposition}

Like the convex case, Problem $\mathcal{P}^{pd}$ can be converted to a special case of Problem $\mathcal{P}^{dd}$, and vice versa. 
Thus, PD-M/A and DD-M/A can be indirectly applied to Problem $\mathcal{P}^{dd}$ and Problem $\mathcal{P}^{pd}$, respectively.

Problem $\mathcal{P}^{pd}$, suited for PD-M/A, can be equivalently converted to Problem $\mathcal{P}^{pd-dd}$ by replacing the optimization variables and introducing linear equality constraints:
\begin{align}
	\mathcal{P}^{pd-dd}:\ \min_{\mathbf{z}} \ &\sum_{i\in\{0\}\cup\mathcal{I}}f_i(\mathbf{z}_i)\nonumber\\  
	s.t. \ 
	&\mathbf{C}_i\mathbf{z}_{i}= \mathbf{z}_{0}, \ i\in\mathcal{I}, \label{sp-d coup-eq-cons-1}\\
	&\tilde{\mathbf{h}}_i(\mathbf{z}_i)= \mathbf{0}, \ i\in\mathcal{I},\label{sp-d coup-eq-cons-2}\\
	&\tilde{\mathbf{g}}_i(\mathbf{z}_i)\preceq \mathbf{0}, \ i\in\mathcal{I},\label{sp-d decoup-ineq-cons-1}\\
	&\mathbf{g}_i(\mathbf{C}'_i\mathbf{z}_i)\preceq \mathbf{0}, \ i\in\{0\}\cup\mathcal{I},\label{sp-d decoup-ineq-cons-2}\\
	&\mathbf{z}_{i}\in \mathcal{Z}_i, \ i\in\{0\}\cup\mathcal{I},\label{sp-d decoup-cvx-cons}
\end{align}
where $\mathbf{z}_{0}\triangleq\mathbf{y}\in\mathbb{R}^{n_0}$, 
$\mathbf{C}'_0\triangleq\mathbf{I}_{n_0}$,
$\mathcal{Z}_0\triangleq\mathcal{Y}$, 
$\mathbf{z}_{i}\triangleq(\mathbf{x}_{i}, \mathbf{y})\in\mathbb{R}^{n_i+n_0}$, $\mathbf{C}_i\triangleq\begin{bmatrix}	\mathbf{0}_{n_0\times n_i} & \mathbf{I}_{n_0}\end{bmatrix}$,
$\mathbf{C}'_i\triangleq\begin{bmatrix}
	\mathbf{I}_{n_i} & \mathbf{0}_{n_i\times n_0}
\end{bmatrix}$,
$\mathcal{Z}_i\triangleq\mathcal{X}_i\times \mathcal{Y}$,
$i\in\mathcal{I}$, and the constraints in \eqref{sp-d coup-eq-cons-1} are additionally introduced. 
Problem  $\mathcal{P}^{pd}$ is equivalent to Problem $\mathcal{P}^{pd-dd}$ (which follows immediately from Lemma~\ref{lemma equivalent-noncvx-1} in Appendix~\ref{app equivalent nonconvex prob}). 
Letting $\mathbf{A}_{0}\triangleq-\begin{bmatrix}
	\mathbf{I}_{n_0} &\cdots & \mathbf{I}_{n_0}
\end{bmatrix}^{T}\in\mathbb{R}^{n_0r\times n_0}$ and $\mathbf{A}_{i}\triangleq\begin{bmatrix}
	\mathbf{0} & \cdots & \mathbf{C}_{i}^{T} & \cdots & \mathbf{0}
\end{bmatrix}^{T}\in\mathbb{R}^{n_0r\times (n_i+n_0)}$, $i\in\mathcal{I}$, the constraints in \eqref{sp-d coup-eq-cons-1} can be rewritten as $\sum_{i\in\{0\}\cup\mathcal{I}}\mathbf{A}_{i}\mathbf{z}_{i}=\mathbf{0}$.
Let $\tilde{m}\triangleq\sum_{i\in \mathcal{I}}\tilde{m}_i$.
Denote $\tilde{\mathbf{H}}_0(\mathbf{z}_0)\triangleq\mathbf{0}$.
For all $i\in\mathcal{I}$, denote $\tilde{\mathbf{H}}_i:\mathcal{Z}_i\rightarrow\mathbb{R}^{\tilde{m}}$ as  
\begin{align*}
	&\langle\tilde{\mathbf{H}}_i(\mathbf{z}_i)\rangle_j\triangleq\\
	&\begin{cases}
		\langle\tilde{\mathbf{h}}_i(\mathbf{z}_i)\rangle_j, & j=1+\sum_{i'=1}^{i-1}\tilde{m}_{i'},\cdots,1+\sum_{i'=1}^{i}\tilde{m}_{i'}\\
		0,& \text{otherwise}
	\end{cases},
\end{align*} 
for $i\in\mathcal{I}$.
Then,   \eqref{sp-d coup-eq-cons-2} can be rewritten as $\sum_{i\in\{0\}\cup\mathcal{I}}\tilde{\mathbf{H}}_{i}(\mathbf{z}_{i})=\mathbf{0}$.
Thus, Problem $\mathcal{P}^{pd-dd}$ can be viewed as a special case of Problem $\mathcal{P}^{dd}$ without the coupling inequality constraint in \eqref{dual coup-ineq-cons}.
Consequently, Problem $\mathcal{P}^{pd}$ can be indirectly solved by applying DD-M/A to Problem $\mathcal{P}^{pd-dd}$.

In addition, Problem $\mathcal{P}^{dd}$, suited for DD-M/A, can be equivalently converted to Problem $\mathcal{P}^{dd-pd}$ by introducing new variables, inequality constraints, and equality constraints:
\begin{align}
	\mathcal{P}^{dd-pd}:\ \min_{\mathbf{x}, \mathbf{z}} \ &\sum_{i\in\mathcal{I}}f_i(\mathbf{x}_i)\nonumber\\  
	s.t. \ 
	&\tilde{\mathbf{g}}_i(\mathbf{x}_i)-\mathbf{z}_i\preceq\mathbf{0}, \ i\in\mathcal{I},\label{sd-p coup-ineq}\\
	&\tilde{\mathbf{h}}_i(\mathbf{x}_i)-\mathbf{z}_{i+I}=\mathbf{0}, \ i\in\mathcal{I},\label{sd-p coup-eq}\\
	&\mathbf{z}\in \mathcal{Z}, \label{sd-p decoup-cvx-z}\\
	&\text{\eqref{dual coup-ineq-cons}, \eqref{dual decoup-cvx-cons}},\nonumber
\end{align}
where $\mathbf{z}_{i}\in\mathbb{R}^{\tilde{r}}$, $\mathbf{z}_{i+I}\in\mathbb{R}^{\tilde{m}}$, $i\in\mathcal{I}$, 
and $\mathcal{Z}\triangleq\{\mathbf{z}\in \mathbb{R}^{\tilde{r}I+\tilde{m}I}| \sum_{i\in\mathcal{I}}\mathbf{z}_i\preceq\mathbf{0}, \sum_{i\in \mathcal{I}}\mathbf{z}_{i+I}=\mathbf{0}\}$. 
Problem  $\mathcal{P}^{dd}$ is equivalent to Problem $\mathcal{P}^{dd-pd}$ (which follows immediately from Lemma~\ref{lemma equivalent-noncvx-2} in Appendix~\ref{app equivalent nonconvex prob}).
Obviously, Problem $\mathcal{P}^{dd-pd}$ is a special case of Problem $\mathcal{P}^{spd}$ without the separable inequality constraints for $\mathbf
y$ in \eqref{primal decoup-ineq-cons-y} and the coupling variable $\mathbf{z}$ in the objective function.
Accordingly, Problem $\mathcal{P}^{dd}$ can be indirectly solved by applying PD-M/A to Problem $\mathcal{P}^{dd-pd}$.

	\subsection{Two-Level Decomposition}
	Similar to the convex counterpart, we can repeatedly apply PD-M/A and DD-M/A at different levels to decouple problems with both coupling variables and constraints:
	\begin{align}
		\mathcal{P}^{tld}:\ \min_{\mathbf{x}, \mathbf{y}} \ &
		\sum_{i\in\mathcal{I}}f_i(\mathbf{x}_i,\mathbf{y})+f_0(\mathbf{y})&&\nonumber\\  
		s.t. \ 
		&\sum_{i\in\mathcal{I}}\tilde{\mathbf{g}}_i(\mathbf{x}_i,\mathbf{y})\preceq\mathbf{0},&&\label{md coup-ineq-cons}\\ 
		&\sum_{i\in\mathcal{I}}\tilde{\mathbf{h}}_i(\mathbf{x}_i,\mathbf{y})=\mathbf{0},&&\label{md coup-eq-cons}\\
		&\text{\eqref{primal decoup-ineq-cons-x}-\eqref{primal decoup-cvx-cons-y}}, &&\nonumber
	\end{align}           
	where 
	$\mathbf{x}$, $\mathbf{y}$, $f_i$, and $\mathbf{g}_{i}$, $i\in \{0\}\cup\mathcal{I}$ are the same as in Problem $\mathcal{P}^{pd}$, 
	$\tilde{\mathbf{g}}_i:\mathcal{U}_i\times \mathcal{V}\rightarrow \mathbb{R}^{\tilde{r}}$,
	and $\tilde{\mathbf{h}}_i:\mathcal{U}_i\times \mathcal{V}\rightarrow \mathbb{R}^{\tilde{m}}$, 
	$i\in \mathcal{I}$.
	In the following, we introduce two two-level decomposition techniques for solving Problem $\mathcal{P}^{tld}$.
	
	1) At the first level, we apply DD-M/A to deal with the coupling constraints in \eqref{md coup-ineq-cons} and \eqref{md coup-eq-cons}.       
	Then, at the second level, we use the (S)PD-M/A to deal with the coupling variable $\mathbf{y}$.
	Applying PD-M/A at the second level leads to a two-level optimization decomposition with a master dual problem, a secondary master primal problem, and the subproblems; 
	applying SPD-M/A at the second level leads to a two-level optimization decomposition with a master dual problem, a sequence of secondary master primal problems, and a sequence of subproblems.\footnote{At the first level, SDD-M/A is not applicable due to the existence of the generally nonlinear equality constraint in (\ref{md coup-eq-cons}).}

	2) At the first level, we use PD-M/A to deal with the coupling variable $\mathbf{y}$.
	Then, at the second level, we apply \mbox{DD-M/A} to deal with the coupling constraints in \eqref{md coup-ineq-cons} and \eqref{md coup-eq-cons}. 
	Applying DD-M/A at the second level results in a two-level optimization decomposition with a master primal problem, a secondary master dual problem, and the subproblems.\footnote{SPD-M/A and SDD-M/A are not applicable at the first and second levels, respectively, due to the existence of the generally nonlinear equality constraint in (\ref{md coup-eq-cons}).}

\section{Examples}\label{sec exm}
In this section, we present examples of the proposed algorithms, i.e., PD-A, SPD-A, DD-A, and \mbox{SDD-A}.\footnote{PD-A has been successfully applied to zero-forcing beamforming in single-cell MIMO networks \cite{zhaizfprimal}, with numerical results demonstrating the superior advantages of its parallel implementation over existing algorithms. 
Real-world applications of SPD-A, DD-A, and SDD-A remain open and warrant further study.}

\begin{example}[Example of Problem $\mathcal{P}^{pd}$]\label{exm 1}
The constraints in \eqref{primal coup-ineq-cons}, \eqref{primal decoup-ineq-cons-x}, and \eqref{primal decoup-ineq-cons-y} are absent, $y\in\mathbb{R}$, $\mathcal{Y}=[0,1]$, $f_0(y)= a(y-y_0)^2$ with $a\in\mathbb{R}_{++}$ and $y_0\in[0,1]$, and for all $i\in\mathcal{I}$, $\mathbf{x}_{i}=(x_{i,1},x_{i,2})\in\mathbb{R}^2$,  $\mathcal{X}_{i}=[-1,1]\times\mathbb{R}$, 
\begin{align}\label{exm 1-f}
	f_i(\mathbf{x}_i, y)=&
    \sum_{j=1}^{3}a_{i,j}(y)x_{i,1}^{j}
	+\sum_{j=1}^{2}b_{i,j} x_{i,2}^{j},
\end{align}   
with  $a_{i,j}(y)\triangleq \sum_{l=0}^{2}a_{i,j,l}y^{l}$, $a_{i,j,l}\in\mathbb{R}$, $b_{i,2}\in\mathbb{R}_{++}$, and $b_{i,1}\in\mathbb{R}$, and 
\begin{equation}\label{exam 1-h}
	\begin{aligned}
		h_i(\mathbf{x}_i, y)= -c_{i,2}\frac{x_{i,1}^2}{y+1}+c_{i,1}x_{i,2}+c_{i,0},
	\end{aligned}   
\end{equation}
with $c_{i,2}\in\mathbb{R}_{+}$ and $c_{i,j}\in\mathbb{R}$.
\end{example}

PD-A can be applied to obtain stationary points of Example~\ref{exm 1}, whereas SPD-A cannot due to the nonlinear equality constraints in \eqref{primal coup-eq-cons}.
Specifically, in each iteration of PD-A, each subproblem $\mathcal{P}_{sub, i}^{pd}$ and Problem $\mathcal{P}^{pd, \dag, (k)}_{mas}$ have closed-form stationary and optimal points, respectively.

\begin{example}[Example of Problem $\mathcal{P}^{spd}$]\label{exm 2}
This example is the same as Example~\ref{exm 1}, except that the equality constraints in \eqref{primal coup-eq-cons} with $h_i(\mathbf{x}_i,y)$ given by \eqref{exam 1-h} are replaced by the inequality constraints in \eqref{primal coup-ineq-cons} with $g_i(\mathbf{x}_i,y)= h_i(\mathbf{x}_i,y)$ given by \eqref{exam 1-h}. 
\end{example}

PD-A and SPD-A can be applied to obtain stationary points of Example~\ref{exm 2}. 
Specifically, in each iteration of PD-A, each subproblem $\mathcal{P}_{sub, i}^{pd}$ and Problem $\mathcal{P}^{pd, \dag, (k)}_{mas}$ have closed-form stationary and optimal points, respectively. 
In each iteration of SPD-A, the approximate objective and constraint functions are chosen as:
\begin{align}
	&F_0(y;y^{(k)})= f_0(y)\label{exm 2-F0}\\
	&F_{i,\mathbf{x}}(\mathbf{x}_{i};\mathbf{x}_{i}^{(k)},y^{(k)})= f_{i}(x_{i,1}^{(k)},x_{i,2},y^{(k)})\nonumber\\
	&+\nabla_{x_{i,1}}f_i(\mathbf{x}_{i}^{(k)},y^{(k)}) (x_{i,1}-x_{i,1}^{(k)})+\frac{\tau_{\mathbf{x}}}{2}(x_{i,1}-x_{i,1}^{(k)})^2,\label{exm 2-F}\\ 
	&F_{i,y}(y;\mathbf{x}_{i}^{(k)},y^{(k)})=\frac{\tau_{y}}{2}(y-y^{(k)})^2\nonumber\\
	&\quad \quad \quad \quad \quad \quad \quad \quad+\nabla_{y} f_i(\mathbf{x}_{i}^{(k)},y^{(k)})(y-y^{(k)}), \label{exm 2-Fy}\\
	&G_{i}(\mathbf{x}_{i}, y;\mathbf{x}_{i}^{(k)},y^{(k)})= c_{i,1}x_{i,2}+c_{i,0}+g_i(\mathbf{x}_{i}^{(k)},y^{(k)})\nonumber\\
	&\quad\quad\quad\quad\quad\quad\quad\quad\quad+\nabla_{x_{i,1}} g_i(\mathbf{x}_{i}^{(k)},y^{(k)})(x_{i,1}-x_{i,1}^{(k)})\nonumber\\
	&\quad\quad\quad\quad\quad\quad\quad\quad\quad+\nabla_{y} g_i(\mathbf{x}_{i}^{(k)},y^{(k)})(y-y^{(k)}),\label{exm 2-G}
\end{align}
where $i\in\mathcal{I}$ and $\tau_{\mathbf{x}}>0$, and each subproblem $\mathcal{P}^{spd, (k)}_{sub, i}$ has a unique closed-form optimal point.

\begin{example}[Example of Problem $\mathcal{P}^{spd}$]\label{exm 3}
This example is the same as Example~\ref{exm 2} except that the objective function's component $f_i(\mathbf{x}_{i},y)$ given by \eqref{exm 1-f} is replaced by:
\begin{equation}\label{exm 3-f}
	\begin{aligned}
		f_i(\mathbf{x}_i, y)=\sum_{j=1}^{3}a_{i,j}(y)x_{i,1}^{j+2}	+\sum_{j=1}^{2}b_{i,j} x_{i,2}^{j},
	\end{aligned}   
\end{equation}
where $a_{i,j}(y)$ and $b_{i,j}$
are the same as in Example~\ref{exm 2} (and~\ref{exm 1}).
\end{example}

PD-A and SPD-A can be applied to obtain stationary points of Example~\ref{exm 3}. 
Specifically, in each iteration of \mbox{PD-A}, each subproblem $\mathcal{P}_{sub, i}^{pd}$ has to be solved numerically, and Problem $\mathcal{P}^{pd, \dag, (k)}_{mas}$ has a unique closed-form optimal point.
In each iteration of SPD-A, the same approximate functions in \eqref{exm 2-F0}-\eqref{exm 2-G} for Example~\ref{exm 2} can be applied, and each subproblem $\mathcal{P}^{spd, (k)}_{sub, i}$ has a unique closed-form optimal point.

\begin{example}[Example of Problem $\mathcal{P}^{dd}$]\label{exm 4}
The constraints in \eqref{dual coup-ineq-cons} and \eqref{dual decoup-ineq-cons} are absent, and for all $i\in\mathcal{I}$, $x_{i}\in\mathbb{R}$, $\mathcal{X}_{i}=[-0.05,0.05]$, and 
\begin{align}
	&f_i(\mathbf{x}_i)= \sum_{j=1}^{3}a_{i,j}x_{i}^{j},\label{exm 4-f}\\
	&h_i(\mathbf{x}_i)= \sum_{j=1}^{3}b_{i,j}x_{i}^{j}+b/I,\label{exm 4-h}
\end{align}	   
with $a_{i,j}, b_{i,j},b\in\mathbb{R}$.
\end{example}

DD-A can be applied to obtain stationary points of Example~\ref{exm 4} with closed-form stationary points for all subproblems per iteration, whereas SDD-A cannot due to the nonlinear equality constraints in \eqref{dual coup-eq-cons}.

\begin{example}[Example of Problem $\mathcal{P}^{sdd}$]\label{exm 5}
This example is the same as Example~\ref{exm 4}, except that the equality constraints in \eqref{dual coup-eq-cons} with $h_i(\mathbf{x}_i)$ given by \eqref{exm 4-h} are replaced by the inequality constraints in \eqref{dual coup-ineq-cons} with $g_i(\mathbf{x}_i)= h_i(\mathbf{x}_i)$ given by \eqref{exm 4-h}. 
\end{example}

DD-A and SDD-A can be applied to obtain stationary points of Example~\ref{exm 5}. 
Specifically, in each iteration of DD-A, each subproblem $\mathcal{P}_{sub, i}^{dd}$ has a closed-form stationary point.
In each iteration of SDD-A, the approximate objective and constraint functions are chosen as: 
\begin{equation}
	\begin{aligned}
		&F_{i}(x_{i};x_{i}^{(k)})= 3a_{i,3}(x_{i}^{(k)})^2(x_i-x_i^{(k)})+a_{i,1}x_{i}\\
		&+\begin{cases}
			\frac{\tau}{2}(x_{i}-x_{i}^{(k)})^2+2a_{i,2}x_{i}^{(k)}(x_i-x_i^{(k)}), & a_{i,2}\leq0\\
			a_{i,2}x_{i}^2, &  a_{i,2}>0
		\end{cases}, 
	\end{aligned}
\end{equation} 
\begin{equation}\label{exm 5-G}
	\begin{aligned}
		&G_{i}(x_{i};x_{i}^{(k)})=  3b_{i,3}(x_{i}^{(k)})^2(x_i-x_i^{(k)})+b_{i,3}(x_{i}^{(k)})^3 \\	
		&+\frac{L}{2}(x_{i}-x_{i}^{(k)})^2+b_{i,1}x_{i}+b/I\\
		&+\begin{cases}
			2b_{i,2}x_{i}^{(k)}(x_i-x_i^{(k)})+b_{i,2}(x_{i}^{(k)})^2, & b_{i,2}\leq0\\
			b_{i,2}x_{i}^2, &  b_{i,2}>0
		\end{cases},
	\end{aligned}
\end{equation} 
where $i\in\mathcal{I}$, $\tau>0$, and $L\geq0$, and each subproblem $\mathcal{P}^{sdd, (k)}_{sub, i}$ has a unique closed-form optimal point.

\begin{example}[Example of Problem $\mathcal{P}^{sdd}$]\label{exm 6}
This example is the same as Example~\ref{exm 5} except that the objective function's component $f_i(\mathbf{x}_{i})$ given by \eqref{exm 4-f} is replaced by: 
\begin{align}
	f_i(\mathbf{x}_i)= \sum_{j=1}^{3}a_{i,j}x_{i}^{j+2},
\end{align}	
where $a_{i,j}$ is the same as in Example~\ref{exm 5} (and \ref{exm 4}).
\end{example}

DD-A and SDD-A can be applied to obtain stationary points of Example~\ref{exm 6}. 
Specifically, in each iteration of DD-A, each subproblem $\mathcal{P}_{sub, i}^{dd}$ has to be solved numerically.
In each iteration of SDD-A, the approximate objective function is chosen as: 
\begin{equation}
	F_{i}(x_{i};x_{i}^{(k)})= \nabla f_{i}(x_{i}^{(k)})(x_i-x_i^{(k)})+	\frac{\tau}{2}(x_{i}-x_{i}^{(k)})^2,
\end{equation} 
with $\tau>0$, the approximate constraint function $G_{i}(x_{i};x_{i}^{(k)})$ is chosen by \eqref{exm 5-G} as in Example~\ref{exm 5}, and each subproblem $\mathcal{P}^{sdd, (k)}_{sub, i}$ has a unique closed-form optimal point.

\section{Numerical Results}\label{sec numerical results}
In this section, we numerically evaluate the proposed four algorithms,  i.e., PD-A, SPD-A, DD-A, and SDD-A, with three baselines, i.e., SCAPD \cite{scutari2017parallel}, SQP \cite{gill2011sequential}, and AL \cite{fernandez2012local}, on the examples presented in Section~\ref{sec exm}.  
Notably, SCADD \cite{scutari2017parallel} is not considered, since it is identical to SDD-A on Examples~\ref{exm 5}-\ref{exm 6}.  
Our experimental environment is Ubuntu 24.04.1.
For hardware, GPU is Nvidia GeForce RTX 4090, and CPU is AMD EPYC 9654 Processor.

For all examples, we set $I=1000$ and randomly and independently generate $10$ samples of the problem parameters.\footnote{In numerical optimization, a problem is typically categorized as large-scale when the number of optimization variables $n$ reaches the order of $10^3$ or higher, since at this scale, standard centralized solvers (e.g., interior-point methods) often become computationally prohibitive due to their $O(n^3)$ per-iteration computational complexity \cite{boyd2004convex}.}
Specifically, for Examples~\ref{exm 1}-\ref{exm 3}, the samples of  $a_{i,j,l}$, $j=1,2,3$, $b_{i,1}$, and $c_{i,j}$, $j=0,1$
are generated according to i.i.d. $\mathcal{N}(0,1)$, 
the samples of $c_{i,2}$ are generated according to the half-normal distribution of $\mathcal{N}(0,1)$, 
the samples of $a$ and $b_{i,2}$ are generated according to i.i.d.  $\mathcal{U}(0,5000)$, 
and the samples of $y_0$ are generated according to i.i.d. $\mathcal{U}(0,1)$.
For Examples~\ref{exm 4}-\ref{exm 6}, the samples of $a_{i,j}$ and $b_{i,j}$  
are generated according to i.i.d. $\mathcal{N}(0,1)$.
For Example~\ref{exm 4} and Examples~\ref{exm 5}-\ref{exm 6}, the samples of $b$ are generated according to i.i.d. $\mathcal{N}(0,(0.001)^2)$ and $\mathcal{U}(-0.001,0)$, respectively.
Note that for Examples~\ref{exm 4}-\ref{exm 6}, the original objective values may be negative or near-zero.
To ensure a clear and consistent visualization of the performance comparisons in Fig. \ref{fig dual}, we have applied specific constant offsets to the results of all considered algorithms within Examples~\ref{exm 4}-\ref{exm 6}. 
Specifically, constant offsets of $41$, $40$, and $0.06$ are added to the objective values for Examples~\ref{exm 4}, \ref{exm 5}, and \ref{exm 6}, respectively. 
Note that each linear shift is applied uniformly across all algorithms with in each specific example, thereby preserving the relative performance gaps.

\begin{table*}[ht]
	\centering
	\caption{Algorithm parameter.}
		\begin{tabularx}{\textwidth}{|m{1.3cm}|>{\centering\arraybackslash}X|} 
			\hline
			\centering Algorithm &  Algorithm parameter 
			\\
			\hline
			\centering PD-A
			& Ex.~1: $\tau=0$, $\gamma^{(0)}=1$, $\alpha=1$, $\beta=5$, $\epsilon=1$; 
			Ex.~2: $\tau=0$, $\gamma^{(0)}=1$, $\alpha=1$, $\beta=5$, $\epsilon=1$; 
			Ex.~3: $\tau=5$, $\gamma^{(0)}=1$, $\alpha=1$, $\beta=5$, $\epsilon=1$. 
			\\
			\hline
			\centering SPD-A
			& Ex.~2: $\tau_{\mathbf{x}}=10^8$, $\tau_{y}=0$, $\gamma^{(0)}=1$, $\alpha=0$, $\beta=1$, $\epsilon=0.1$, $\gamma_{in}^{(0)}=1$,  $\beta_{in}=0.5$, $\sigma=0.05$, $T=10$; 
			Ex.~3: $\tau_{\mathbf{x}}=10^8$, $\tau_{y}=0$, $\gamma^{(0)}=1$, $\alpha=0$, $\beta=1$, $\epsilon=0.1$, $\gamma_{in}^{(0)}=1$,  $\beta_{in}=0.5$, $\sigma=0.05$, $T=10$.
			\\
			\hline
			\centering DD-A
			& Ex.~4: $\tau=8$, $\gamma^{(0)}=0.01$, $\alpha=3$, $\beta=1$, $\epsilon=0.9$; 
			Ex.~5: $\tau=10$, $\gamma^{(0)}=0.01$, $\alpha=3$, $\beta=1$, $\epsilon=0.9$; 
			Ex.~66: $\tau=10$, $\gamma^{(0)}=1$, $\alpha=3$, $\beta=1$, $\epsilon=0.9$.
			\\
			\hline
			\centering SDD-A
			& Ex.~5: $\tau=10^{-5}$, $L=0.1$, $\gamma^{(0)}=1$, $\alpha=1$, $\beta=1$, $\epsilon=0.1$, $\gamma_{in}^{(0)}=1$,  $\beta_{in}=0.5$, $\sigma=0.05$, $T=10$;
			Ex.~6: $\tau=10^{-1}$, $L=0.1$, $\gamma^{(0)}=1$, $\alpha=1$, $\beta=1$, $\epsilon=0.1$, $\gamma_{in}^{(0)}=0.001$,  $\beta_{in}=0.9$, $\sigma=0.05$, $T=10$.
			\\
			\hline 
			\centering SCAPD 
			& Ex.~2: $\tau_{\mathbf{x}}=10^8$, $\tau_{y}=0$, $\gamma^{(0)}=1$, $\alpha=0$, $\beta=1$, $\epsilon=0.1$, $\gamma_{in}^{(0)}=1$,  $\beta_{in}=[0.1,0.2,\cdots, 0.9]$, $\sigma=0.05$, $T=10$; 
			Ex.~3: $\tau_{\mathbf{x}}=10^8$, $\tau_{y=0}$, $\gamma^{(0)}=1$, $\alpha=0$, $\beta=1$, $\epsilon=0.1$, $\gamma_{in}^{(0)}=1$,  $\beta_{in}=[0.1,0.2,\cdots, 0.9]$, $\sigma=0.05$, $T=10$
			\\
			\hline
			\centering SQP  
			& Ex.~1: $\rho=[1,10,10^2]$, $\beta=[0.06,0.08,0.1]$, $\alpha=0.1$;
			Ex.~2: $\rho=[1,10,10^2]$, $\beta=[0.06,0.08,0.1]$, $\alpha=0.1$;
			Ex.~3: $\rho=[1,10,10^2]$, $\beta=[0.06,0.08,0.1]$, $\alpha=0.1$
			Ex.~4: $\rho=[10^{-3},10^{-2},10^{-1}]$, $\beta=[0,1,0.2,0.3]$, $\alpha=0.1$
			Ex.~5: $\rho=[10^{-3},10^{-2},10^{-1}]$, $\beta=[0,1,0.2,0.3]$, $\alpha=0.1$
			Ex.~6: $\rho=[10^{-3},10^{-2},10^{-1}]$, $\beta=[0,1,0.2,0.3]$, $\alpha=0.1$
			\\
			\hline
			\centering AL 
			& Ex.~1: $\rho=[10^{10},10^{12},10^{14}]$, $\beta=[2,4,6]$, $\alpha=0.25$
			Ex.~2: $\rho=[10^{10},10^{12},10^{14}]$, $\beta=[2,4,6]$, $\alpha=0.25$
			Ex.~3: $\rho=[10^{10},10^{12},10^{14}]$, $\beta=[2,4,6]$, $\alpha=0.25$
			Ex.~4: $\rho=[10^{-1},10^{0},10^{-1}]$, $\beta=[2,3,4]$,  $\alpha=0.25$
			Ex.~5: $\rho=[10^{-4},10^{-3},10^{-2}]$, $\beta=[2,3,4]$,  $\alpha=0.25$
			Ex.~6: $\rho=[10^{-4},10^{-3},10^{-2}]$, $\beta=[2,3,4]$,  $\alpha=0.25$
			\\
			\hline
		\end{tabularx}
	\label{table para}
\end{table*}

For PD-A, SPD-A, DD-A, and SDD-A, we choose stepsizes $\gamma^{(k)}=\frac{1}{\alpha+\beta k^{\epsilon}},k\geq 1$.
For PD-A on Example~\ref{exm 3} and DD-A on Example~\ref{exm 6}, we utilize MATLAB’s built-in fmincon function to numerically solve the subproblems.
For SPD-A and SDD-A's inner algorithms, (i.e., Algorithms~\ref{sp alg:alg1-PD} and~\ref{sd alg:alg1-DD}), we choose stepsizes $\gamma_{in}^{(t)}=\gamma_{in}^{(t-1)}(1-\beta_{in} \gamma_{in}^{(t-1)}), t\geq 1$ and the termination criteria $|F(\mathbf{x}^{(t)},y^{(t)};\mathbf{x}^{(k)},y^{(k)})-F(\mathbf{x}^{(t-1)},y^{(t-1)};\mathbf{x}^{(k)},y^{(k)})|\leq \sigma|F(\mathbf{x}^{(t-1)},y^{(t-1)};\mathbf{x}^{(k)},y^{(k)})|$ with $t\leq T$ and $|q(\tilde{\mu}^{(t)};\mathbf{x}^{(k)})-q(\tilde{\mu}^{(t-1)};\mathbf{x}^{(k)})|\leq \sigma|q(\tilde{\mu}^{(t-1)};\mathbf{x}^{(k)})|$ with $t\leq T$, respectively.
For SCAPD, we choose the same approximate functions as in SPD-A, introduce coupling slack variables $\mathbf{Z}\triangleq(z_{i,j})_{i\in\mathcal{I},j=1,2}\in\mathbb{R}^{I\times 2}$ and additional constraints $\sum_{j=1,2}z_{i,j}\leq 0$, $i\in\mathcal{I}$ for the convex approximate problem per iteration, and update the coupling variables $\mathbf{Z}$ by utilizing MATLAB's built-in quadprog function to numerically solve Quadratic Program (QP). 
For SQP, we utilize MATLAB’s built-in quadprog function to solve approximate QPs, backtracking line search (with the penalty parameter $\rho$, the shrinkage factor $\beta$, and the sufficient decrease parameter $\sigma$) to determine the stepsize, and the Broyden-Fletcher-Goldfarb-Shanno (BFGS) method to update the approximate Hessian matrix.  
Note that for all implementations involving MATLAB’s built-in quadprog function, we employ the interior-point-convex algorithm with both the optimality tolerance and step tolerance set to $10^{-5}$.
For AL, we utilize MATLAB’s built-in fmincon function to solve AL problems constructed by eliminating only the coupling constraints by a penalty (with the penalty parameter $\rho$, increase factor $\beta$, and the sufficient decrease parameter $\sigma$). 
The parameters for all algorithms are listed in Table~\ref{table para}.
Notably, for all examples, we choose only one set of parameters for the proposed algorithms and $9$ different parameter configurations for the baselines.

For each example, we randomly and independently generate $10$ initial points for all applicable algorithms with details given below. 
For Examples~\ref{exm 1}-\ref{exm 3}, we randomly generate $10$ initial points of $y$ and $x_{i,1}$, $i\in\mathcal{I}$ according to i.i.d. $\mathcal{U}(0,1)$ and $\mathcal{U}(-1,1)$.
For Example~\ref{exm 1} and Examples~\ref{exm 2}-\ref{exm 3}, we randomly generate $10$ initial points of Lagrange multipliers $\tilde{\lambda}_{i}$ and $\tilde{\mu}_{i}$, $i\in\mathcal{I}$ associated with the coupling constraints according to i.i.d. $\mathcal{U}(-1,1)$ and $\mathcal{U}(0,1)$, respectively.
For Examples~\ref{exm 4}-\ref{exm 6}, we randomly generate $10$ initial points of $x_{i}$, $i\in\mathcal{I}$ according to i.i.d. $\mathcal{U}(-0.05,0.05)$.
For Example~\ref{exm 4} and Examples~\ref{exm 5}-\ref{exm 6}, we randomly generate $10$ initial points of the Lagrange multipliers $\tilde{\lambda}$ and $\tilde{\mu}$ associated with the coupling constraints according to i.i.d. $\mathcal{U}(-1,1)$ and $\mathcal{U}(0,1)$, respectively. 
Notably, the initial points of $(\mathbf{x},y)$ in Examples~\ref{exm 1}-\ref{exm 3} are feasible,
whereas the initial points of $\mathbf{x}$ in Examples~\ref{exm 4}-\ref{exm 6} are not guaranteed to satisfy the coupling constraints.

For Examples~\ref{exm 1}-\ref{exm 3}, the convergence criterion is given by: 
\begin{align*}
	&\frac{1}{I}\sum_{i\in \mathcal{I}}\max\{0,g_{i}(\mathbf{x}_{i}^{(k)},y^{(k)})\}< 10^{-6}, k=9,10,\\
	&\max_{i\in \mathcal{I}}\max\{0, g_{i}(\mathbf{x}_{i}^{(k)},y^{(k)})\}< 10^{-5}, k=9,10,\\
	&|f(\mathbf{x}^{(10)},y^{(10)})-f(\mathbf{x}^{(9)},y^{(9)})|\leq 0.05 |f(\mathbf{x}^{(9)},y^{(9)})|,
\end{align*}
and the best objective values are chosen among all $10$ iterates that satisfy:
\begin{align*}
	&\frac{1}{I}\sum_{i\in \mathcal{I}}\max\{0,g_{i}(\mathbf{x}_{i}^{(k)},y^{(k)})\}< 10^{-6},\\
	&\max_{i\in \mathcal{I}}|h_{i}(\mathbf{x}_{i}^{(k)},y^{(k)})|< 10^{-5}.
\end{align*}
For Examples~\ref{exm 4}-\ref{exm 6}, the convergence criterion is given by:
\begin{align*}
	&\max\{0,g(\mathbf{x}^{(k)})\}< 10^{-2}, k=9,10,\\
	&|f(\mathbf{x}^{(10)})-f(\mathbf{x}^{(9)})|\leq 0.05 |f(\mathbf{x}^{(9)})|,
\end{align*}
and the best objective values are chosen among all $10$ iterates that satisfy:
\begin{align*}
	\max\{0,g(\mathbf{x}^{(k)})\}< 10^{-2}.
\end{align*}

\begin{table}[t]
	\centering
	\caption{Proportion of convergent trials within $10$ iterations.}
	\resizebox{0.485\textwidth}{!}{
    \begin{tabular}{|c|c|c|c|c|c|c|c|} 
		\hline
		\multirow{2}*{}&\multirow{2}*{Algorithm} &  \multicolumn{6}{c|}{Proportion of convergent trials (\%)} 
		\\
		\cline{3-8}& & Ex.~\ref{exm 1} & Ex.~\ref{exm 2} & Ex.~\ref{exm 3} & Ex.~\ref{exm 4} & Ex.~\ref{exm 5} & Ex.~\ref{exm 6}
		\\
		\hline
		\multirow{4}*{Proposed} & PD-A
		& 100 & 100 & 100 & N/A & N/A & N/A
		\\
		\cline{2-8} &
		SPD-A
		& N/A&100& 100 & N/A& N/A& N/A
		\\
		\cline{2-8} &
		DD-A
		& N/A& N/A& N/A& 100& 100& 100
		\\
		\cline{2-8} &
		SDD-A
		& N/A& N/A& N/A& N/A& 100& 100
		\\
		\hline \multirow{3}*{Baseline}&
		SCAPD \cite{scutari2017parallel}
		& N/A& 75& 84& N/A& N/A& N/A
		\\
		\cline{2-8} &
		SQP \cite{gill2011sequential} & 53& 79& 81& 90& 100& 100
		\\
		\cline{2-8} &
		AL \cite{fernandez2012local}& 100& 97& 97& 55& 100& 12
		\\
		\hline
	\end{tabular}}
	\label{table converge}
\end{table}

\begin{figure*}[t]
	\centering
	\subfloat[Example~\ref{exm 1}.]{\includegraphics[width=0.25\textwidth]{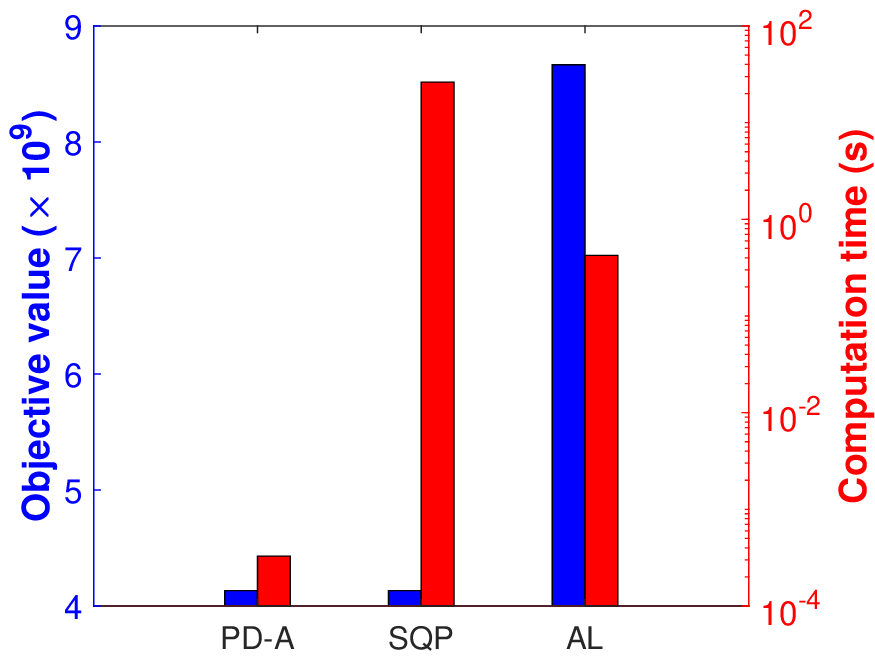}\label{fig exm 1}}
	\hspace{1pt}
	\subfloat[Example~\ref{exm 2}.]{\includegraphics[width=0.25\textwidth]{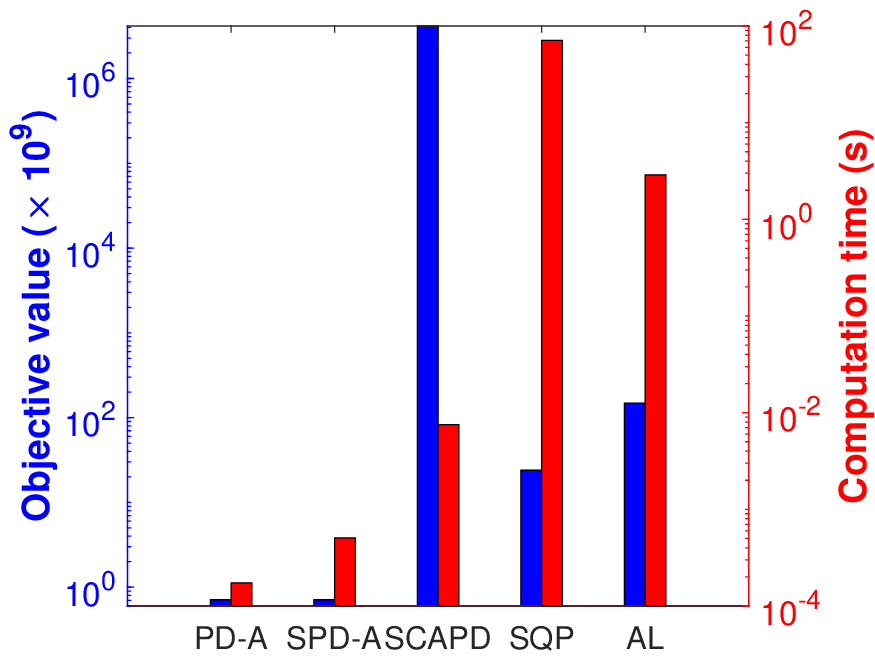}\label{fig exm 2}}
	\hspace{1pt}
	\subfloat[Example~\ref{exm 3}.]{\includegraphics[width=0.25\textwidth]{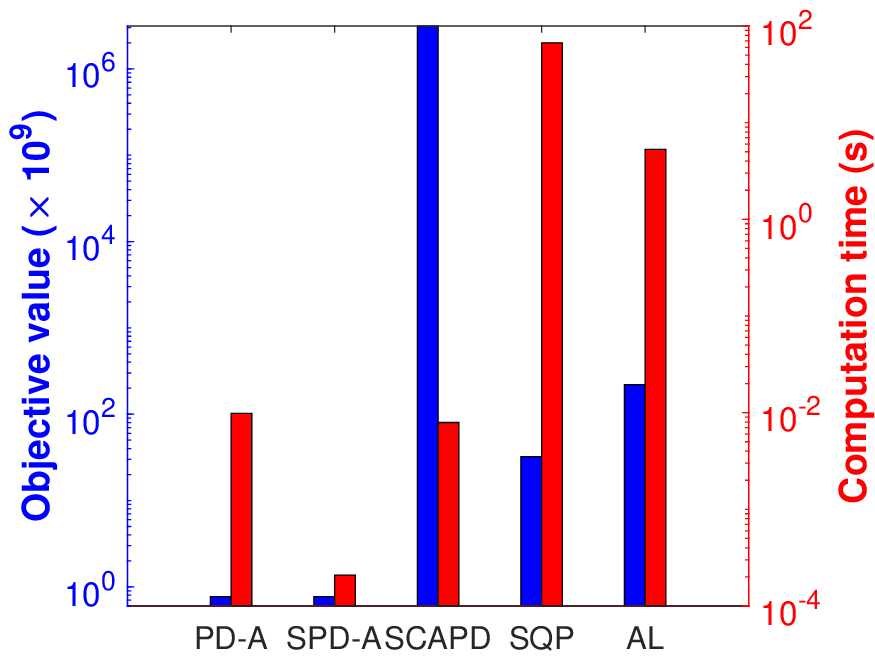}\label{fig exm 3}}
	\caption{Best objective values and minimum computation times for Examples~\ref{exm 1}-\ref{exm 3} with coupling variables.}
	\label{fig primal}
\end{figure*}

\begin{figure*}[t]
	\centering
	\subfloat[Example~\ref{exm 4}.]{\includegraphics[width=0.25\textwidth]{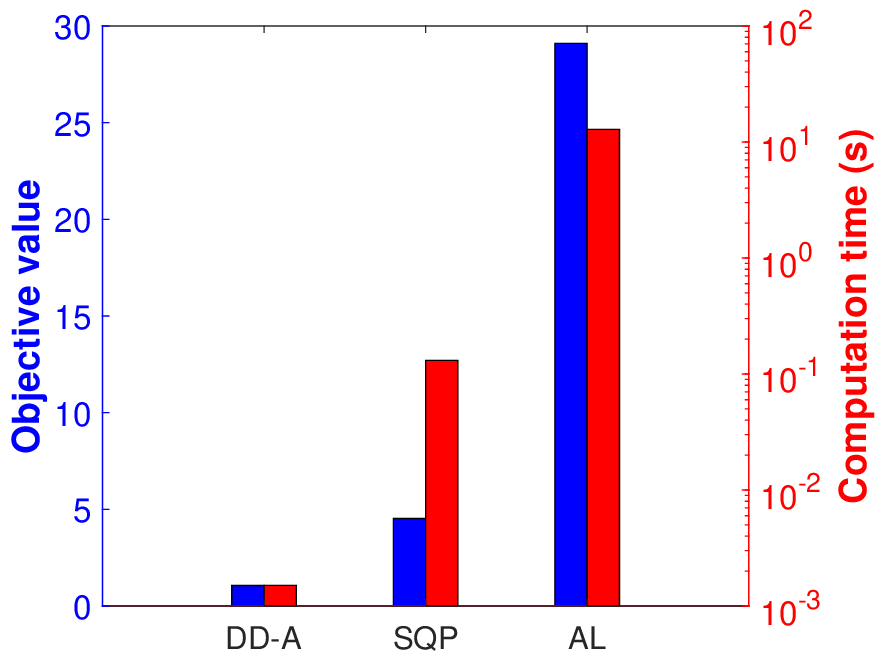}\label{fig exm 4}}
	\hspace{1pt}
	\subfloat[Example~\ref{exm 5}.]{\includegraphics[width=0.25\textwidth]{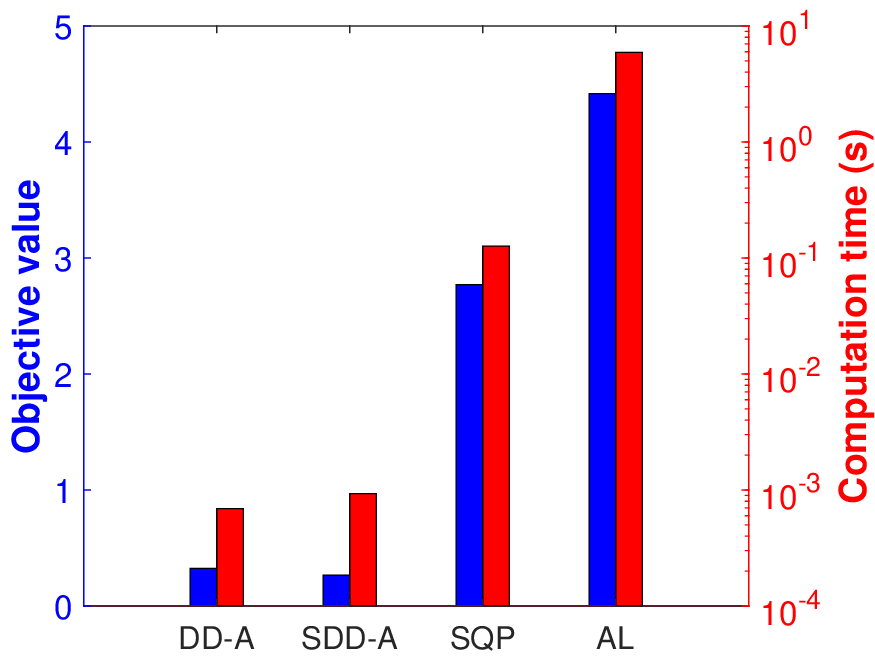}\label{fig exm 5}}
	\hspace{1pt}
	\subfloat[Example~\ref{exm 6}.]{\includegraphics[width=0.25\textwidth]{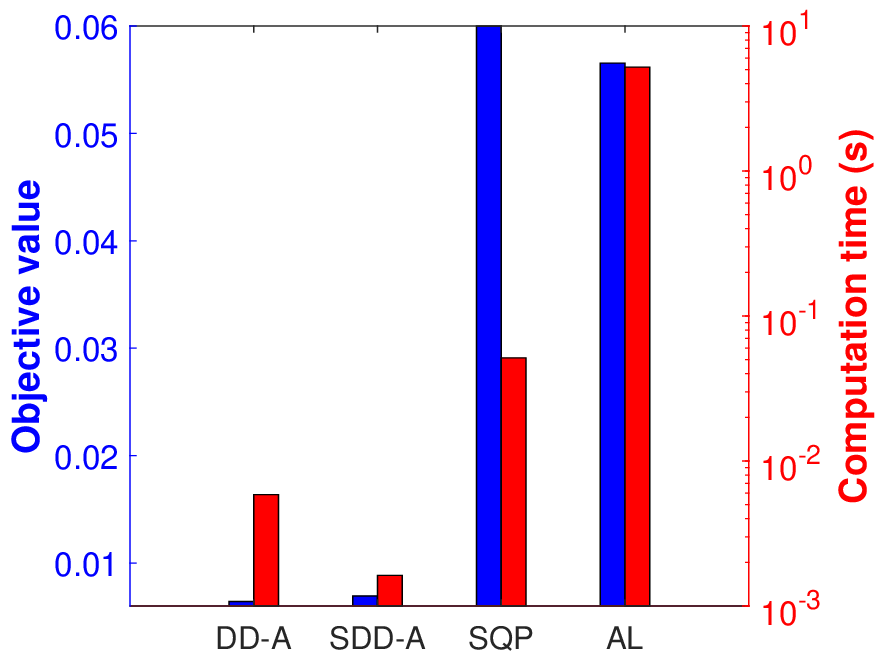}\label{fig exm 6}}
	\caption{Best objective values and minimum computation times for Examples~\ref{exm 4}-\ref{exm 6} with coupling constraints.}
	\label{fig dual}
\end{figure*}

Table~\ref{table converge} shows the proportions of convergent trials within $10$ iterations out of $100$ trials  ($10$ samples of problem parameters and $10$ initial points).
Note that the results of the baselines in Table~\ref{table converge} represent the best outcomes obtained from $9$ different parameter settings.
Fig.~\ref{fig primal} and Fig.~\ref{fig dual} illustrate the best objective values and their corresponding minimum computation times. 
For each example, the results are obtained by averaging w.r.t. the intersection of all applicable algorithms' convergent trials. 
From Table~\ref{table converge}, Fig.~\ref{fig primal}, and Fig.~\ref{fig dual}, we observe that the proposed algorithms outperform the corresponding baselines in these examples in terms of convergence performance and computation time.  
Moreover, PD-A and SPD-A (DD-A and SDD-A) exceed each other in different examples, together increasing the chance of achieving  better performance.

\section{Conclusion and Future Work}\label{sec conclusion}
This paper establishes a comprehensive decomposition framework for nonconvex problems with decomposition structures including four methods and algorithms, i.e., \mbox{PD-M/A}, \mbox{SPD-M/A}, \mbox{DD-M/A}, and \mbox{SDD-M/A}.
The proposed methods and algorithms can handle a larger set of nonconvex problems with decomposition structures than the state-of-the-art ones, exploit decomposition structures, allow for parallel and distributed implementations, produce stationary points, and offer good opportunities to achieve superior tradeoff between convergence performance and computation time. 
The proposed methods and algorithms successfully extend the classic ones for convex problems and the recent ones for nonconvex problems, enriching the decomposition theory and possibly offering approaches for large-scale  problems in practice. This paper opens up future research directions on applying the proposed decomposition framework to solve real-world problems with decomposition structures, e.g., in signal processing and wireless communications.

\appendix

\subsection{Proof of Lemma~\ref{primal lem1}}\label{appendix primal lem1}
For all $i\in\mathcal{I}$, the expression of the KKT function $\mathbf{k}_{i}:\mathcal{U}_i\times\mathbb{R}^{\tilde{r}_i}\times\mathbb{R}^{\tilde{m}_i}\times\mathbb{R}^{r_i}\times\mathcal{V}\rightarrow \mathbb{R}^{n_i+\tilde{r}_i+\tilde{m}_i+r_i}$ of the subproblem $\mathcal{P}^{pd}_{sub,i}$ is given by:
\begin{equation}\label{app lem1-kkt-function}
	\begin{aligned}
			&\mathbf{k}_i(\mathbf{x}_i,\tilde{\boldsymbol{\mu}}_i, \tilde{\boldsymbol{\lambda}}_i, \boldsymbol{\mu}_i,  \mathbf{y})\triangleq\begin{pmatrix}
					\boldsymbol{\zeta}_{i}(\mathbf{x}_i,\tilde{\boldsymbol{\mu}}_i,  \tilde{\boldsymbol{\lambda}}_i, \boldsymbol{\mu}_i, \mathbf{y})
					\\
					\boldsymbol{\eta}_{i,1}(\mathbf{x}_i,\tilde{\boldsymbol{\mu}}_i, \mathbf{y})
					\\
					\tilde{\mathbf{h}}_i(\mathbf{x}_i, \mathbf{y})
					\\
					\boldsymbol{\eta}_{i,2}(\mathbf{x}_i, \mathbf{y}),
				
					\end{pmatrix},
		\end{aligned}
\end{equation}	
where
\begin{align*}
		&\boldsymbol{\zeta}_{i}(\mathbf{x}_i,\tilde{\boldsymbol{\mu}}_i, \boldsymbol{\mu}_i, \tilde{\boldsymbol{\lambda}}_i, \mathbf{y})\triangleq
		\nabla_{\mathbf{x}_i} f_i(\mathbf{x}_i, \mathbf{y})\\
		&
		+\nabla_{\mathbf{x}_i} \tilde{\mathbf{g}}_i(\mathbf{x}_i, \mathbf{y})\tilde{\boldsymbol{\mu}}_i 
		+\nabla_{\mathbf{x}_i} \tilde{\mathbf{h}}_i(\mathbf{x}_i, \mathbf{y})\tilde{\boldsymbol{\mu}}_i
		+\nabla \mathbf{g}_i(\mathbf{x}_i)\boldsymbol{\mu}_i,\\
		&\langle\boldsymbol{\eta}_{i,1}(\mathbf{x}_i,\tilde{\boldsymbol{\mu}}_i, \mathbf{y})\rangle_j\triangleq\begin{cases}
				\langle\tilde{\mathbf{g}}_i(\mathbf{x}_i, \mathbf{y})\rangle_{j} & j\in \tilde{\mathcal{A}}(\mathbf{x}_i, \mathbf{y})\\
				\langle\tilde{\boldsymbol{\mu}}_i\rangle_{j} & j\notin \tilde{\mathcal{A}}(\mathbf{x}_i, \mathbf{y})
			\end{cases},&\\
		&\langle\boldsymbol{\eta}_{i,2}(\mathbf{x}_i,\boldsymbol{\mu}_i)\rangle_{j}\triangleq\begin{cases}
				\langle\mathbf{g}_i(\mathbf{x}_i)\rangle_{j} & j\in\mathcal{A}(\mathbf{x}_i)\\
				\langle\boldsymbol{\mu}_i\rangle_{j} & j\notin \mathcal{A}(\mathbf{x}_i)
			\end{cases},&
	\end{align*}
with $\tilde{\mathcal{A}}(\mathbf{x}_i, \mathbf{y})\triangleq\{j\in\{1, \cdots, \tilde{r}_i\}| \ \langle\tilde{\mathbf{g}}_i(\mathbf{x}_i, \mathbf{y})\rangle_j=0\}$ and $\mathcal{A}(\mathbf{x}_i)\triangleq\{j\in\{1, \cdots, r_i\}| \ \langle\mathbf{g}_i(\mathbf{x}_i)\rangle_j=0\}$.

\begin{figure*}[ht]
	\hrule
	\begin{gather}
			\nabla_{\mathbf{x}_i} f_i(\mathbf{x}^{\ddagger}_i, \mathbf{y}^{\ddagger})
			+\nabla_{\mathbf{x}_i} \tilde{\mathbf{g}}_i(\mathbf{x}^{\ddagger}_i, \mathbf{y}^{\ddagger} ) \tilde{\boldsymbol{\mu}}_{i}^{\ddagger} 
			+\nabla_{\mathbf{x}_i} \tilde{\mathbf{h}}_i(\mathbf{x}^{\ddagger}_i, \mathbf{y}^{\ddagger} ) \tilde{\boldsymbol{\lambda}}_{i}^{\ddagger} 
			+\nabla \mathbf{g}_{i}(\mathbf{x}^{\ddagger}_i) \boldsymbol{\mu}_{i}^{\ddagger}
			=\mathbf{0}, 
			\  i\in\mathcal{I},\label{app lem1 kkt1}
			\\
			\tilde{\boldsymbol{\mu}}_{i}^{\ddagger}\odot\tilde{\mathbf{g}}_{i}(\mathbf{x}^{\ddagger}_i, \mathbf{y}^{\ddagger})=\mathbf{0},\
			\boldsymbol{\mu}_{i}^{\ddagger}\odot\mathbf{g}_{i}(\mathbf{x}^{\ddagger}_i) =\mathbf{0},  \
			\tilde{\boldsymbol{\mu}}_{i}\succeq\mathbf{0}, \
			\boldsymbol{\mu}_{i}\succeq\mathbf{0}, \
			i\in\mathcal{I}, \label{app lem1 kkt2}\\
			\tilde{\mathbf{g}}_i(\mathbf{x}^{\ddagger}_i, \mathbf{y}^{\ddagger})\preceq \mathbf{0},  \
			\tilde{\mathbf{h}}_i(\mathbf{x}^{\ddagger}_i, \mathbf{y}^{\ddagger})= \mathbf{0},  \
			\mathbf{g}_{i}(\mathbf{x}^{\ddagger}_i)\preceq \mathbf{0},  \
			\mathbf{x}^{\ddagger}_i\in \mathcal{X}_i,\  
			i\in\mathcal{I}.
			\label{app lem1 kkt3}
	\end{gather}
	\hrule 
\end{figure*}

\subsubsection{Proof of Lemma~\ref{primal lem1}.1}
First, we prove that Implicit Function Theorem hold for KKT function $\mathbf{k}_{i}$, for all $i\in\mathcal{I}$.
Since $\mathbf{x}^{\ddagger}\in \prod_{i\in\mathcal{I}}\operatorname{int}(\mathcal{X}_{i})$, the stationary point $(\mathbf{x}^{\ddagger},\mathbf{y}^{\ddagger})$ together with its Lagrange multipliers $\tilde{\boldsymbol{\mu}}^{\ddagger}_i, \tilde{\boldsymbol{\lambda}}^{\ddagger}_i, \boldsymbol{\mu}^{\ddagger}_i,$, $i\in\mathcal{I}$ satisfies 
(\ref{app lem1 kkt1}), (\ref{app lem1 kkt2}), and (\ref{app lem1 kkt3}), as shown at the top of this page.
This implies that $\mathbf{k}_i(\mathbf{x}_i^{\ddagger},  \tilde{\boldsymbol{\mu}}^{\ddagger}_i,    \tilde{\boldsymbol{\lambda}}^{\ddagger}_i, \boldsymbol{\mu}^{\ddagger}_i, \mathbf{y}^{\ddagger})=\mathbf{0}$, for all $i\in\mathcal{I}$.
Besides, by Assumptions~\ref{primal assump-p} and \ref{primal assump-stationary-2}.1, there exists a neighborhood $\mathcal{N}'_{\mathbf{y}^{\ddagger}}\subseteq\operatorname{dom}f^{\dagger}$ of $\mathbf{y}^{\ddagger}$ and  neighborhoods $\mathcal{N}'_{\mathbf{x}_{i}^{\ddagger}}\subseteq\operatorname{int}(\mathcal{X}_{i})$, $i\in\mathcal{I}$ of $\mathbf{x}_{i}^{\ddagger}$, $i\in\mathcal{I}$ such that for all $i\in\mathcal{I}$, the constraints in (\ref{primal sub-coup-ineq-cons}) and (\ref{primal sub-decoup-ineq-cons-x}) are satisfied for all $\mathbf{x}_{i}\in\mathcal{N}'_{\mathbf{x}_{i}^{\ddagger}}$ and $\mathbf{y}\in\mathcal{N}'_{\mathbf{y}^{\ddagger}}$, and the sets $\tilde{\mathcal{A}}(\mathbf{x}_{i}, \mathbf{y})$ and $\mathcal{A}(\mathbf{x}_{i})$ do not change over $\mathcal{N}'_{\mathbf{x}_{i}^{\ddagger}}\times\mathcal{N}'_{\mathbf{y}^{\ddagger}}$ and $\mathcal{N}'_{\mathbf{x}_{i}^{\ddagger}}$, respectively.
This implies that for all $i\in\mathcal{I}$, $\mathbf{k}_{i}$ is continuously differentiable over $\mathcal{N}'_{\mathbf{x}_{i}^{\ddagger}}\times\mathbb{R}^{\tilde{r}_i}\times\mathbb{R}^{\tilde{m}_i}\times\mathbb{R}^{r_i}\times\mathcal{N}'_{\mathbf{y}^{\ddagger}}$. 
Based on the above results and Assumption~\ref{primal assump-stationary-2}.2, we can conclude that Implicit Function Theorem hold for KKT function $\mathbf{k}_{i}$, for all $i\in\mathcal{I}$.

Then, it is readily followed from Implicit Function Theorem that there exists a neighborhood $\mathcal{N}_{\mathbf{y}^{\ddagger}}\subseteq\mathcal{N}'_{\mathbf{y}^{\ddagger}}$ of $\mathbf{y}^{\ddagger}$,  neighborhoods $\mathcal{N}_{\mathbf{x}_{i}^{\ddagger}}\subseteq\mathcal{N}'_{\mathbf{x}_{i}^{\ddagger}}$, $\mathcal{N}_{\tilde{\boldsymbol{\mu}}_{i}^{\ddagger}}\subseteq\mathbb{R}^{\tilde{r}_{i}}_{+}$, 
$\mathcal{N}_{\tilde{\boldsymbol{\lambda}}_{i}^{\ddagger}}\subseteq\mathbb{R}^{\tilde{m}_{i}}$, 
$\mathcal{N}_{\boldsymbol{\mu}_{i}^{\ddagger}}\subseteq\mathbb{R}^{r_{i}}_{+}$, $i\in\mathcal{I}$ of $\mathbf{x}_{i}^{\ddagger}$, $\tilde{\boldsymbol{\mu}}_{i}^{\ddagger}$,    $\tilde{\boldsymbol{\lambda}}_{i}^{\ddagger}$, $\boldsymbol{\mu}_{i}^{\ddagger}$, $i\in\mathcal{I}$, respectively, and single-valued continuously differentiable functions $\mathbf{X}^{\dag}_i:\mathcal{N}_{\mathbf{y}^{\ddagger}}\rightarrow \mathcal{N}_{\mathbf{x}_{i}^{\ddagger}}$, $\widetilde{\mathbf{M}}_i:\mathcal{N}_{\mathbf{y}^{\ddagger}}\rightarrow \mathcal{N}_{\tilde{\boldsymbol{\mu}}_{i}^{\ddagger}}$,
$\tilde{\boldsymbol{\Lambda}}_i:\mathcal{N}_{\mathbf{y}^{\ddagger}}\rightarrow \mathcal{N}_{\tilde{\boldsymbol{\lambda}}_{i}^{\ddagger}}$, $\mathbf{M}_i:\mathcal{N}_{\mathbf{y}^{\ddagger}}\rightarrow \mathcal{N}_{\boldsymbol{\mu}_{i}^{\ddagger}}$, $i\in\mathcal{I}$ such that for all $i\in\mathcal{I}$ and $\mathbf{y}\in\mathcal{N}_{\mathbf{y}^{\ddagger}}$, 
$\mathbf{k}_i(\mathbf{X}^{\dag}_i(\mathbf{y}),  \widetilde{\mathbf{M}}_i(\mathbf{y}),    \tilde{\boldsymbol{\Lambda}}_i(\mathbf{y}), \mathbf{M}_i(\mathbf{y}), \mathbf{y})=\mathbf{0}$.
This implies that for all $i\in\mathcal{I}$ and  $\mathbf{y}\in\mathcal{N}_{\mathbf{y}^{\ddagger}}$, $\mathbf{X}^{\dag}_i(\mathbf{y})$ together with $\widetilde{\mathbf{M}}_i(\mathbf{y})$,    $\tilde{\boldsymbol{\Lambda}}_i(\mathbf{y})$, and $\mathbf{M}_i(\mathbf{y})$ satisfies the KKT conditions for the subproblem $\mathcal{P}_{sub, i}^{pd}$ except for (\ref{primal sub-coup-ineq-cons}) and (\ref{primal sub-decoup-ineq-cons-x}).
Furthermore, since $\mathcal{N}_{\mathbf{y}^{\ddagger}}\subseteq\mathcal{N}'_{\mathbf{y}^{\ddagger}}$, $\mathcal{N}_{\mathbf{x}_{i}^{\ddagger}}\subseteq\mathcal{N}'_{\mathbf{x}_{i}^{\ddagger}}$, $i\in\mathcal{I}$, we can get 
that for all $i\in\mathcal{I}$ and $\mathbf{y}\in\mathcal{N}_{\mathbf{y}^{\ddagger}}$, $\mathbf{X}_{i}(\mathbf{y})\in\mathcal{N}'_{\mathbf{x}_{i}^{\ddagger}}\subseteq\operatorname{int}(\mathcal{X}_{i})$ satisfies the constraints in (\ref{primal sub-coup-ineq-cons}) and (\ref{primal sub-decoup-ineq-cons-x}).
Thus, for all $i\in\mathcal{I}$ and  $\mathbf{y}\in\mathcal{N}_{\mathbf{y}^{\ddagger}}$, $\mathbf{X}^{\dag}_i(\mathbf{y})\in\operatorname{int}(\mathcal{X}_{i})$ together with $\widetilde{\mathbf{M}}_i(\mathbf{y})$,    $\tilde{\boldsymbol{\Lambda}}_i(\mathbf{y})$, and $\mathbf{M}_i(\mathbf{y})$ satisfies the KKT conditions for the subproblem $\mathcal{P}_{sub, i}^{pd}$.
Therefore, we can show Lemma~\ref{primal lem1}.1.

\subsubsection{Proof of Lemma~\ref{primal lem1}.2}
Notice that based on the proof of Lemma~\ref{primal lem1}.1, we can also get that for all $i\in\mathcal{I}$ and $\mathbf{y}\in\mathcal{N}_{\mathbf{y}^{\ddagger}}$, 
the point $(\mathbf{X}^{\dag}_i(\mathbf{y}),  \widetilde{\mathbf{M}}_i(\mathbf{y}),    \tilde{\boldsymbol{\Lambda}}_i(\mathbf{y}),\mathbf{M}_i(\mathbf{y}))$ is the unique point in $\mathcal{N}_{\mathbf{x}_{i}^{\ddagger}}\times\mathcal{N}_{\tilde{\boldsymbol{\mu}}_{i}^{\ddagger}}\times\mathcal{N}_{\tilde{\boldsymbol{\lambda}}_{i}^{\ddagger}}\times\mathcal{N}_{\boldsymbol{\mu}_{i}^{\ddagger}}$ that together with $\mathbf{y}$ satisfies $\mathbf{k}_i(\mathbf{x}_i, \tilde{\boldsymbol{\mu}}_i,    \tilde{\boldsymbol{\lambda}}_i, \boldsymbol{\mu}_i, \mathbf{y})=\mathbf{0}$.
Based on this, we prove Lemma~\ref{primal lem1}.2 in the following.
First, we prove the first statement of Lemma~\ref{primal lem1}.2.
For all $i\in\mathcal{I}$, suppose now by contradiction that the subproblem $\mathcal{P}^{pd}_{sub,i}$ for a fixed $\mathbf{y}\in \mathcal{N}_{\mathbf{y}^{\ddagger}}$ has two different stationary points in $\mathcal{N}_{\mathbf{x}_{i}^{\ddagger}}$ that have Lagrange multipliers associated with the constraints in (\ref{primal sub-coup-ineq-cons}), (\ref{primal sub-coup-eq-cons}), and (\ref{primal sub-decoup-ineq-cons-x}) in $\mathcal{N}_{\tilde{\boldsymbol{\mu}}_{i}^{\ddagger}}$,   $\mathcal{N}_{\tilde{\boldsymbol{\lambda}}_{i}^{\ddagger}}$, and $\mathcal{N}_{\boldsymbol{\mu}_{i}^{\ddagger}}$, respectively.
Denote these two different stationary points as $\mathbf{x}'_i$ and $\mathbf{x}''_i$, and denote their Lagrange multipliers as $\tilde{\boldsymbol{\mu}}'_i$,  $\tilde{\boldsymbol{\lambda}}'_i$, $\boldsymbol{\mu}'_i$ and $\tilde{\boldsymbol{\mu}}''_i$,  $\tilde{\boldsymbol{\lambda}}''_i$, $\boldsymbol{\mu}''_i$, respectively.
Then, it follows from the KKT conditions for the subproblem $\mathcal{P}^{pd}_{sub,i}$ that $\mathbf{k}_i(\mathbf{x}'_i, \tilde{\boldsymbol{\mu}}'_i,    \tilde{\boldsymbol{\lambda}}'_i, \boldsymbol{\mu}'_i, \mathbf{y})=\mathbf{k}_i(\mathbf{x}''_i, \tilde{\boldsymbol{\mu}}''_i,    \tilde{\boldsymbol{\lambda}}''_i, \boldsymbol{\mu}''_i, \mathbf{y})=\mathbf{0}$, which is in contradiction with the uniqueness of point $(\mathbf{x}_i, \tilde{\boldsymbol{\mu}}_i,    \tilde{\boldsymbol{\lambda}}_i, \boldsymbol{\mu}_i)\in\mathcal{N}_{\mathbf{x}_{i}^{\ddagger}}\times\mathcal{N}_{\tilde{\boldsymbol{\mu}}_{i}^{\ddagger}}\times\mathcal{N}_{\tilde{\boldsymbol{\lambda}}_{i}^{\ddagger}}\times\mathcal{N}_{\boldsymbol{\mu}_{i}^{\ddagger}}$ that together with $\mathbf{y}$ satisfies $\mathbf{k}_i(\mathbf{x}_i, \tilde{\boldsymbol{\mu}}_i,    \tilde{\boldsymbol{\lambda}}_i, \boldsymbol{\mu}_i, \mathbf{y})=\mathbf{0}$. 
Thus, the first statement of Lemma~\ref{primal lem1}.2 holds.
Similarly, we can prove the second statement of Lemma~\ref{primal lem1}.2.
Therefore, we can show Lemma~\ref{primal lem1}.2.

\subsection{Proof of Theorem~\ref{primal thm1}}\label{appendix primal thm1}
Based on Lemma~\ref{primal lem1}, we only need to prove Theorem~\ref{primal thm1} under Assumptions~\ref{primal assump-p} and \ref{primal assump-stationary-1}.

\subsubsection{Proof of Theorem~\ref{primal thm1}.1}
First, we prove the first statement of Theorem~\ref{primal thm1}.1.
The composite function $f^{\dagger}_{\mathcal{I}}$ is continuously differentiable on $\mathcal{N}$, since functions $f_{i}$, $\mathbf{X}^{\dagger}_{i}$, $i\in\mathcal{I}$, which compose $f^{\dagger}_{\mathcal{I}}$, are continuously differentiable.

Next, we prove the second statement of Theorem~\ref{primal thm1}.1. 
Consider any $i\in\mathcal{I}$ and $\mathbf{y}\in\mathcal{N}$. 
Define the sets $\tilde{\mathcal{A}}(\mathbf{x}_i, \mathbf{y})$ and $\mathcal{A}(\mathbf{x}_i)$ as in the proof of Lemma~\ref{primal lem1} in Appendix~\ref{appendix primal lem1}.
From Assumption~\ref{primal assump-stationary-1}, we conclude that
the stationary point $\mathbf{X}^{\dagger}_{i}(\mathbf{y})\in\operatorname{int}(\mathcal{X}_{i})$ of the subproblem $\mathcal{P}_{sub, i}^{pd}$, together with $\widetilde{\mathbf{M}}_i(\mathbf{y})$, $\tilde{\boldsymbol{\Lambda}}_i(\mathbf{y})$, and $\mathbf{M}_i(\mathbf{y})$, satisfies the following KKT conditions: 
\begin{gather}
	\begin{split}
		&\nabla_{\mathbf{x}_i} f_i(\mathbf{X}^{\dagger}_i(\mathbf{y}), \mathbf{y})
		\!+\!\nabla_{\mathbf{x}_i} \tilde{\mathbf{g}}_i(\mathbf{X}^{\dagger}_i(\mathbf{y}), \mathbf{y}) \widetilde{\mathbf{M}}_i(\mathbf{y})\\
		&+\!\nabla_{\mathbf{x}_i} \tilde{\mathbf{h}}_i(\mathbf{X}^{\dagger}_i(\mathbf{y}), \mathbf{y}) \tilde{\boldsymbol{\Lambda}}_i(\mathbf{y})\!+\!\nabla_{\mathbf{x}_i} \mathbf{g}_{i}(\mathbf{X}^{\dagger}_i(\mathbf{y})) \mathbf{M}_i(\mathbf{y})\!=\!\mathbf{0},
	\end{split}\label{app thm1-1} \\
	\hspace{-0.2cm}\widetilde{\mathbf{M}}_i(\mathbf{y})\odot\tilde{\mathbf{g}}_i(\mathbf{X}^{\dagger}_i(\mathbf{y}), \mathbf{y})=\mathbf{0},
	\mathbf{M}_i(\mathbf{y})\odot\mathbf{g}_{i}(\mathbf{X}^{\dagger}_i(\mathbf{y})) =\mathbf{0},\label{app thm1-2}\\
	\widetilde{\mathbf{M}}_i(\mathbf{y})\succeq\mathbf{0}, \ \mathbf{M}_i(\mathbf{y})\succeq\mathbf{0}, \label{app thm1-3}\\ 
	\tilde{\mathbf{g}}_i(\mathbf{X}^{\dagger}_i(\mathbf{y}), \mathbf{y})\preceq \mathbf{0}, \
	\tilde{\mathbf{h}}_i(\mathbf{X}^{\dagger}_i(\mathbf{y}), \mathbf{y})= \mathbf{0},   \label{app thm1-4}\\
	 \mathbf{g}_{i}(\mathbf{X}^{\dagger}_i(\mathbf{y}))\preceq \mathbf{0}, \
	\mathbf{X}^{\dagger}_i(\mathbf{y})\in \mathcal{X}_i.
	\label{app thm1-5}
\end{gather}
For all $k\in\tilde{\mathcal{A}}(\mathbf{X}^{\dagger}_i(\mathbf{y}), \mathbf{y})^c$ and $l\in\mathcal{A}(\mathbf{X}^{\dagger}_i(\mathbf{y}))^c$,
the inequalities in (\ref{app thm1-4}) are strictly satisfied at $\mathbf{X}^{\dagger}_i(\mathbf{y})$ in $k$-th coordinate and $l$-th coordinate, respectively, i.e., $\langle\tilde{\mathbf{g}}_i(\mathbf{X}^{\dagger}_i(\mathbf{y}), \mathbf{y})\rangle_{k}< 0$ and $\langle\mathbf{g}_i(\mathbf{X}^{\dagger}_i(\mathbf{y}))\rangle_{l}< 0$.
Due to the continuity of $\tilde{\mathbf{g}}_i(\mathbf{X}^{\dagger}_i(\cdot), \cdot)$ and $\mathbf{g}_i(\mathbf{X}^{\dagger}_i(\cdot))$, the two inequalities hold true in a neighborhood $\mathcal{N}_{\mathbf{y}}\subseteq\mathcal{N}$ of $\mathbf{y}$.
That is, for all $\mathbf{y}'\in\mathcal{N}_{\mathbf{y}}$, we have $\langle\tilde{\mathbf{g}}_i(\mathbf{X}^{\dagger}_i(\mathbf{y}'), \mathbf{y}')\rangle_{k}< 0$ and $\langle\mathbf{g}_i(\mathbf{X}^{\dagger}_i(\mathbf{y}'))\rangle_{l}< 0$. 
Combining with the equalities in (\ref{app thm1-2}), we have that for all $\mathbf{y}'\in\mathcal{N}_{\mathbf{y}}$, $\langle\widetilde{\mathbf{M}}_i(\mathbf{y}')\rangle_{k}= 0$ and $\langle\mathbf{M}_i(\mathbf{y}')\rangle_{l}= 0$.
This leads to the fact that the $k$-th column of the gradient $\nabla\widetilde{\mathbf{M}}_i(\mathbf{y})$ is the zero vector, i.e., $\nabla\widetilde{\mathbf{M}}_i(\mathbf{y})[1:n_0, k]=\mathbf{0}$,
and the $l$-th column of the gradient $\nabla\mathbf{M}_i(\mathbf{y})$ is the zero vector, i.e., $\nabla\mathbf{M}_i(\mathbf{y})[1:n_0, l] =\mathbf{0}$.
For all $k\in\tilde{\mathcal{A}}(\mathbf{X}_i(\mathbf{y}), \mathbf{y})$ and $l\in\mathcal{A}(\mathbf{X}_i(\mathbf{y}))$, we have $\langle\tilde{\mathbf{g}}_i(\mathbf{X}^{\dagger}_i(\mathbf{y}), \mathbf{y})\rangle_{k}=0$ and $\langle\mathbf{g}_{i}(\mathbf{X}^{\dagger}_i(\mathbf{y}))\rangle_{l}=0$.
Therefore, combining these two cases, we have 
\begin{align}\label{primal thm1-8}
	&\begin{aligned}
		&\nabla\widetilde{\mathbf{M}}_i(\mathbf{y})\tilde{\mathbf{g}}_i(\mathbf{X}^{\dagger}_i(\mathbf{y}), \mathbf{y})\\
		=&\!\!\!\!\!\!\!\!\!\!\!\sum_{k\in\tilde{\mathcal{A}}(\mathbf{X}_i(\mathbf{y}), \mathbf{y})^c}\!\!\!\!\!\!\!\!\!\!\nabla\widetilde{\mathbf{M}}_i(\mathbf{y})[1:n_0, k]\langle\tilde{\mathbf{g}}_i(\mathbf{X}^{\dagger}_i(\mathbf{y}), \mathbf{y})\rangle_{k}
		\\
		&\!\!\!+\!\!\!\!\!\!\!\!\!\!\!\!\sum_{k\in\tilde{\mathcal{A}}(\mathbf{X}_i(\mathbf{y}), \mathbf{y})}\!\!\!\!\!\!\!\!\!\nabla\widetilde{\mathbf{M}}_i(\mathbf{y})[1:n_0, k]\langle\tilde{\mathbf{g}}_i(\mathbf{X}^{\dagger}_i(\mathbf{y}), \mathbf{y})\rangle_{k} 
		=\mathbf{0},
	\end{aligned}\\
	&\begin{aligned}\label{primal thm1-9}
		&\nabla\mathbf{M}_i(\mathbf{y})\mathbf{g}_{i}(\mathbf{X}^{\dagger}_i(\mathbf{y}))\\
		=&\!\!\!\!\!\!\!\!\sum_{l\in\mathcal{A}(\mathbf{X}_i(\mathbf{y}))^c}\!\!\!\!\!\!\!\!\nabla\mathbf{M}_i(\mathbf{y})[1:n_0; l]\langle\mathbf{g}_{i}(\mathbf{X}^{\dagger}_i(\mathbf{y}))\rangle_{l}
		\\
		&\!\!\!+\!\!\!\!\!\!\!\!\!\sum_{l\in\mathcal{A}(\mathbf{X}_i(\mathbf{y}))}\!\!\!\!\!\!\!\nabla\mathbf{M}_i(\mathbf{y})[1:n_0; l]\langle\mathbf{g}_{i}(\mathbf{X}^{\dagger}_i(\mathbf{y}))\rangle_{l} 
		=\mathbf{0}.
	\end{aligned}
\end{align}
Now, for all $i\in\mathcal{I}$, we rewrite $f^{\dagger}_{i}$ as follows:
\begin{equation}\label{primal thm1-10}
	\begin{aligned}
		f^{\dagger}_{i}(\mathbf{y})\!=\!&
		f_i(\mathbf{X}^{\dagger}_i(\mathbf{y}), \mathbf{y})+\!\widetilde{\mathbf{M}}_{i}(\mathbf{y})^{T}\tilde{\mathbf{g}}_i(\mathbf{X}^{\dagger}_i(\mathbf{y}), \mathbf{y})\\
		&\!+\!\tilde{\boldsymbol{\Lambda}}_{i}(\mathbf{y})^{T}\tilde{\mathbf{h}}_i(\mathbf{X}^{\dagger}_i(\mathbf{y}), \mathbf{y})\!+\!\mathbf{M}_{i}(\mathbf{y})^{T}\mathbf{g}_{i}(\mathbf{X}^{\dagger}_i(\mathbf{y})),
	\end{aligned}
\end{equation}
where the equality holds true since the additional terms are equal to zero according to  (\ref{app thm1-2}) and (\ref{app thm1-4}). 
Then, differentiating both sides of the equation in (\ref{primal thm1-10}) w.r.t. $\mathbf{y}$, we get the equation in (\ref{primal thm1-11}), as shown at the top of next page. 
By substituting (\ref{app thm1-1}), (\ref{app thm1-4}), (\ref{primal thm1-8}), and (\ref{primal thm1-9}) into (\ref{primal thm1-11}), we get  (\ref{primal mas-obj-grad-y-i}). 
Therefore, we complete the proof of Theorem~\ref{primal thm1}.1.

\begin{figure*}[ht]
	\begin{equation}\label{primal thm1-11}
		\begin{aligned}
			\nabla f^{\dagger}_{i}(\mathbf{y})
			\!=\!&\nabla\mathbf{X}^{\dagger}_i(\mathbf{y})\left(\nabla_{\mathbf{x}_i} f_i(\mathbf{X}^{\dagger}_i(\mathbf{y}), \mathbf{y})
			+\nabla_{\mathbf{x}_i} \tilde{\mathbf{g}}_i(\mathbf{X}^{\dagger}_i(\mathbf{y}), \mathbf{y}) \widetilde{\mathbf{M}}_i(\mathbf{y})
			+\nabla_{\mathbf{x}_i} \tilde{\mathbf{h}}_i(\mathbf{X}^{\dagger}_i(\mathbf{y}), \mathbf{y}) \tilde{\boldsymbol{\Lambda}}_i(\mathbf{y})
			+\nabla \mathbf{g}_{i}(\mathbf{X}^{\dagger}_i(\mathbf{y})) \mathbf{M}_i(\mathbf{y})
			\right)\\
			&+\nabla\widetilde{\mathbf{M}}_{i}(\mathbf{y})\tilde{\mathbf{g}}_i(\mathbf{X}^{\dagger}_i(\mathbf{y}), \mathbf{y})
			+\nabla\tilde{\boldsymbol{\Lambda}}_{i}(\mathbf{y})\tilde{\mathbf{h}}_i(\mathbf{X}^{\dagger}_i(\mathbf{y}), \mathbf{y})
			+\nabla\mathbf{M}_{i}(\mathbf{y})\mathbf{g}_{i}(\mathbf{X}^{\dagger}_i(\mathbf{y}))\\
			&+
			\nabla_{\mathbf{y}}f_i(\mathbf{X}^{\dagger}_i(\mathbf{y}), \mathbf{y})
			+\nabla_{\mathbf{y}}\tilde{\mathbf{g}}_i(\mathbf{X}^{\dagger}_i(\mathbf{y}), \mathbf{y})\widetilde{\mathbf{M}}_{i}(\mathbf{y})
			+\nabla_{\mathbf{y}}\tilde{\mathbf{h}}_i(\mathbf{X}^{\dagger}_i(\mathbf{y}), \mathbf{y})\tilde{\boldsymbol{\Lambda}}_{i}(\mathbf{y}).
		\end{aligned}
	\end{equation}
		\noindent\hrulefill
	\begin{gather}
		\nabla_{\mathbf{x}_i} f_i(\mathbf{X}^{\dagger}_i(\mathbf{y}^{\dagger}), \mathbf{y}^{\dagger})
		+\nabla_{\mathbf{x}_i} \tilde{\mathbf{g}}_i(\mathbf{X}^{\dagger}_i(\mathbf{y}^{\dagger}), \mathbf{y}^{\dagger}) \widetilde{\mathbf{M}}_i(\mathbf{y}^{\dagger}) 
		+\nabla_{\mathbf{x}_i} \tilde{\mathbf{h}}_i(\mathbf{X}^{\dagger}_i(\mathbf{y}^{\dagger}), \mathbf{y}^{\dagger}) \tilde{\boldsymbol{\Lambda}}_i(\mathbf{y}^{\dagger})
		+\nabla \mathbf{g}_{i}(\mathbf{X}^{\dagger}_i(\mathbf{y}^{\dagger} )) \mathbf{M}_i(\mathbf{y}^{\dagger})
		=\mathbf{0}, \ 
		i\in\mathcal{I},
		\label{kkt-state-1}\\
		\begin{gathered}
			\sum_{i\in\mathcal{I}}\left(
			\nabla_{\mathbf{y}}f_i(\mathbf{X}^{\dagger}_i(\mathbf{y}^{\dagger}), \mathbf{y}^{\dagger} )
			+\nabla_{\mathbf{y}}\tilde{\mathbf{g}}_i(\mathbf{X}^{\dagger}_i(\mathbf{y}^{\dagger}), \mathbf{y}^{\dagger})\widetilde{\mathbf{M}}_{i}(\mathbf{y}^{\dagger} )
			+\nabla_{\mathbf{y}}\tilde{\mathbf{h}}_i(\mathbf{X}^{\dagger}_i(\mathbf{y}^{\dagger}), \mathbf{y}^{\dagger})\tilde{\boldsymbol{\Lambda}}_{i}(\mathbf{y}^{\dagger} ) \right)^{T}\!\!\!(\mathbf{y}-\mathbf{y}^{\dagger}) \\ 
			+ (\nabla f_0(\mathbf{y}^{\dagger})
			+\nabla\mathbf{g}_0(\mathbf{y}^{\dagger})\boldsymbol{\mu}_{0})^{T}(\mathbf{y}-\mathbf{y}^{\dagger})\geq 0, \
			\forall \mathbf{y}\in\mathcal{Y},
		\end{gathered}\label{kkt-state-2}\\
		\widetilde{\mathbf{M}}_i(\mathbf{y}^{\dagger})\odot\tilde{\mathbf{g}}_i(\mathbf{X}^{\dagger}_i(\mathbf{y}^{\dagger}), \mathbf{y}^{\dagger})=\mathbf{0},\  
		\mathbf{M}_i(\mathbf{y}^{\dagger})\odot\mathbf{g}_{i}(\mathbf{X}^{\dagger}_i(\mathbf{y}^{\dagger})) =\mathbf{0},  \
		\widetilde{\mathbf{M}}_i(\mathbf{y}^{\dagger})\succeq\mathbf{0}, \
		\mathbf{M}_i(\mathbf{y}^{\dagger})\succeq\mathbf{0}, \ i\in\mathcal{I}, 
		\label{kkt-state-3}\\
		\tilde{\mathbf{g}}_i(\mathbf{X}^{\dagger}_i(\mathbf{y}^{\dagger}), \mathbf{y}^{\dagger})\preceq \mathbf{0},  \
		\tilde{\mathbf{h}}_i(\mathbf{X}^{\dagger}_i(\mathbf{y}^{\dagger}), \mathbf{y}^{\dagger})= \mathbf{0},  \
		\mathbf{g}_{i}(\mathbf{X}^{\dagger}_i(\mathbf{y}^{\dagger}))\preceq \mathbf{0},  \
		\mathbf{X}^{\dagger}_i(\mathbf{y}^{\dagger})\in \mathcal{X}_i,\  
		i\in\mathcal{I},
		\label{kkt-state-4}\\
		\boldsymbol{\mu}_{0}\odot\mathbf{g}_0(\mathbf{y}^{\dagger})=\mathbf{0},\
		\boldsymbol{\mu}_{0}\succeq\mathbf{0},\ 
		\mathbf{g}_0(\mathbf{y}^{\dagger})\preceq \mathbf{0},\
		\mathbf{y}^{\dagger}\in\mathcal{Y}.
		\label{kkt-state-5}
	\end{gather}
	\hrule 
	\begin{align}\label{app thm1.2}
		\nabla f^{\dagger}(\mathbf{y}^{\dagger})=\nabla f_0(\mathbf{y})+\sum_{i\in \mathcal{I}}\left(\nabla_{\mathbf{y}}f_i(\mathbf{X}^{\dagger}_i(\mathbf{y}), \mathbf{y})
		+\nabla_{\mathbf{y}}\tilde{\mathbf{g}}_i(\mathbf{X}^{\dagger}_i(\mathbf{y}), \mathbf{y})\widetilde{\mathbf{M}}_{i}(\mathbf{y})
		+\nabla_{\mathbf{y}}\tilde{\mathbf{h}}_i(\mathbf{X}^{\dagger}_i(\mathbf{y}), \mathbf{y})\tilde{\boldsymbol{\Lambda}}_{i}(\mathbf{y})\right).
	\end{align}
	\noindent\hrulefill
\end{figure*}

\subsubsection{Proof of Theorem~\ref{primal thm1}.2}
Since $\mathbf{y}^{\dagger}\in\mathcal{N}$ is a stationary point of the master problem $\mathcal{P}^{pd,\dagger}_{mas}$, there exists Lagrange multiplier $\boldsymbol{\mu}_{0}$ which together with $\mathbf{y}^{\dagger}$ satisfies the following KKT conditions:
\begin{gather}
	(\nabla f^{\dagger}(\mathbf{y}^{\dagger})+\nabla\mathbf{g}_0(\mathbf{y}^{\dagger})\boldsymbol{\mu}_{0})^{T}(\mathbf{y}-\mathbf{y}^{\dagger})\geq 0,\ \forall \mathbf{y}\in\mathcal{Y},\label{primal thm1-13}\\
	\boldsymbol{\mu}_{0}\odot\mathbf{g}_0(\mathbf{y}^{\dagger})=\mathbf{0}, \ \boldsymbol{\mu}_{0}\succeq\mathbf{0}, \
	\mathbf{g}_0(\mathbf{y}^{\dagger})\preceq \mathbf{0}, \
	\mathbf{y}^{\dagger}\in\mathcal{Y}.\label{primal thm1-15}
\end{gather}
To prove that $(\mathbf{X}^{\dagger}(\mathbf{y}^{\dagger}), \mathbf{y}^{\dagger})$ is a stationary point of Problem $\mathcal{P}^{pd}$, we only need to prove that $(\mathbf{X}^{\dagger}(\mathbf{y}^{\dagger}), \mathbf{y}^{\dagger})$ together with $\boldsymbol{\mu}_0$, $\widetilde{\mathbf{M}}_i(\mathbf{y}^{\dagger})$, $\tilde{\boldsymbol{\Lambda}}_i(\mathbf{y}^{\dagger})$, and $\mathbf{M}_i(\mathbf{y}^{\dagger})$, $i\in\mathcal{I}$ satisfies the KKT conditions in (\ref{kkt-state-1})-(\ref{kkt-state-5}) for Problem $\mathcal{P}^{pd}$, as shown at the top of next page.
First, by (\ref{app thm1-1})-(\ref{app thm1-5}) and (\ref{primal thm1-15}), we can immediately show that (\ref{kkt-state-1}) and (\ref{kkt-state-3})-(\ref{kkt-state-5}) hold true.
Besides, by (\ref{primal mas-obj-grad-y-i}), we can get \eqref{app thm1.2}, as shown at the top of next page.
Then, by substituting \eqref{app thm1.2} into (\ref{primal thm1-13}), we can get (\ref{kkt-state-2}).
Therefore, we complete the proof of Theorem~\ref{primal thm1}.2.

\subsection{Proof of Theorem~\ref{primal thm2}}\label{primal proof-thm2}
Note that under Assumptions~\ref{regularity}, \ref{SCA-assump-surrogate-constriant}, \ref{primal assump-p}, \ref{primal assump-stationary-2}, and \ref{primal assump-mas-appro-object}, Lemma~\ref{primal lem1} and Theorem~\ref{primal thm1} hold.
Let $\mathcal{N}_{\mathbf{y}^{\ddagger}}$, $\mathcal{N}_{\mathbf{x}_{i}^{\ddagger}}$, $\mathcal{N}_{\tilde{\boldsymbol{\mu}}_{i}^{\ddagger}}$, 
$\mathcal{N}_{\boldsymbol{\mu}_{i}^{\ddagger}}$, $\mathcal{N}_{\tilde{\boldsymbol{\lambda}}_{i}^{\ddagger}}$, $i\in\mathcal{I}$ denote the neighborhoods of $\mathbf{y}^{\ddagger}$, $\mathbf{x}_{i}^{\ddagger}$, $\tilde{\boldsymbol{\mu}}_{i}^{\ddagger}$,   $\boldsymbol{\mu}_{i}^{\ddagger}$, $\tilde{\boldsymbol{\lambda}}_{i}^{\ddagger}$, $i\in\mathcal{I}$, respectively, specified in Lemma~\ref{primal lem1}.
Suppose that $\{\mathbf{y}^{(k)}\}_{k\in\mathbb{N}}\subseteq\mathcal{N}_{\mathbf{y}^{\ddagger}}$, $\{\mathbf{x}^{\dag}_i(\mathbf{y}^{(k)})\}_{k\in\mathbb{N}}\subseteq\mathcal{N}_{\mathbf{x}_{i}^{\ddagger}}$,
$\{\tilde{\boldsymbol{\mu}}_i(\mathbf{y}^{(k)})\}_{k\in\mathbb{N}}\subseteq\mathcal{N}_{\tilde{\boldsymbol{\mu}}_{i}^{\ddagger}}$,   $\{\boldsymbol{\mu}_i(\mathbf{y}^{(k)})\}_{k\in\mathbb{N}}\subseteq\mathcal{N}_{\boldsymbol{\mu}_{i}^{\ddagger}}$, $\{\tilde{\boldsymbol{\lambda}}_i(\mathbf{y}^{(k)})\}_{k\in\mathbb{N}}\subseteq\mathcal{N}_{\tilde{\boldsymbol{\lambda}}_{i}^{\ddagger}}$, $i\in\mathcal{I}$. 
Then, we can conclude from Lemma~\ref{primal lem1}.2 that for all $i\in\mathcal{I}$ and $k$, the stationary point $\mathbf{x}^{\dag}_i(\mathbf{y}^{(k)})$ and its Lagrange multipliers $\tilde{\boldsymbol{\mu}}_i(\mathbf{y}^{(k)})$,   $\tilde{\boldsymbol{\lambda}}_i(\mathbf{y}^{(k)})$, and $\boldsymbol{\mu}_i(\mathbf{y}^{(k)})$ (generated by Algorithm~\ref{primal alg:alg1}) are the same as the ones determined by the functions $\mathbf{X}^{\dag}_i$, $\widetilde{\mathbf{M}}_i$,   $\tilde{\boldsymbol{\Lambda}}_i$, and $\mathbf{M}_i$ at $\mathbf{y}^{(k)}$,  specified in Lemma~\ref{primal lem1}.1.
Thus, in the following, for ease of exposition, for all $i\in\mathcal{I}$, we use $\mathbf{x}^{\dag}_i$, $\tilde{\boldsymbol{\mu}}_i$,   $\tilde{\boldsymbol{\lambda}}_i$, and $\boldsymbol{\mu}_i$ to represent the functions $\mathbf{X}^{\dag}_i$, $\widetilde{\mathbf{M}}_i$,   $\tilde{\boldsymbol{\Lambda}}_i$, and $\mathbf{M}_i$, respectively. 
Accordingly, the function $f^{\dagger}_{\mathcal{I}}(\mathbf{y})$ 
determined by $\mathbf{x}^{\dag}_i(\mathbf{y})$ for $i\in\mathcal{I}$ and $\mathbf{y}\in\mathcal{N}_{\mathbf{y}^{\ddagger}}$  is a single-valued continuously differentiable function on $\mathcal{N}_{\mathbf{y}^{\ddagger}}$, which implies that $f^{\dagger}(\mathbf{y})$ is a single-valued continuously differentiable function on $\mathcal{N}_{\mathbf{y}^{\ddagger}}$. 
Furthermore, without loss of generality, assume that the neighborhood of $\mathbf{y}^{\ddagger}$ specified in Assumption~\ref{primal assump-mas-appro-object} is also $\mathcal{N}_{\mathbf{y}^{\ddagger}}$.
To prove Theorem~\ref{primal thm2}, we only need to prove that  Theorem~\ref{primal thm2}.1 and Theorem~\ref{primal thm2}.2 hold under Assumptions~\ref{regularity}, \ref{SCA-assump-surrogate-constriant}, \ref{primal assump-p}, \ref{primal assump-stationary-2}, and \ref{primal assump-mas-appro-object} and $\{\mathbf{y}^{(k)}\}_{k\in\mathbb{N}}\subseteq\mathcal{N}_{\mathbf{y}^{\ddagger}}$, $\{\mathbf{x}^{\dag}_i(\mathbf{y}^{(k)})\}_{k\in\mathbb{N}}\subseteq\mathcal{N}_{\mathbf{x}_{i}^{\ddagger}}$,
$\{\tilde{\boldsymbol{\mu}}_i(\mathbf{y}^{(k)})\}_{k\in\mathbb{N}}\subseteq\mathcal{N}_{\tilde{\boldsymbol{\mu}}_{i}^{\ddagger}}$,   $\{\boldsymbol{\mu}_i(\mathbf{y}^{(k)})\}_{k\in\mathbb{N}}\subseteq\mathcal{N}_{\boldsymbol{\mu}_{i}^{\ddagger}}$, $\{\tilde{\boldsymbol{\lambda}}_i(\mathbf{y}^{(k)})\}_{k\in\mathbb{N}}\subseteq\mathcal{N}_{\tilde{\boldsymbol{\lambda}}_{i}^{\ddagger}}$, $i\in\mathcal{I}$.

\subsubsection{Proof of Theorem~\ref{primal thm2}.1}
Steps~8-10 of Algorithm~\ref{primal alg:alg1} can be viewed as solving the master problem $\mathcal{P}^{pd,\dagger}_{mas}$ with the objective function $f^{\dagger}(\mathbf{y})$ by the SCA-based algorithm in \cite{scutari2017parallel}.
Moreover, Theorem~\ref{primal thm2}.1 is actually the same as the convergence theorem of the SCA-based algorithm in \cite[Theorem~2]{scutari2017parallel}.
Therefore, we only need to prove that the assumptions  in \cite[Theorem~2]{scutari2017parallel} are satisfied for the SCA-based algorithm used for solving the master problem $\mathcal{P}^{pd,\dagger}_{mas}$.

First, to show the first statement of Theorem~\ref{primal thm2}.1, we need to prove that \cite[Assumptions~1-3 and 5]{scutari2017parallel} are satisfied.
Based on the above discussion, we can extend $f^{\dagger}$ to a continuously differentiable function on $\mathcal{V}$ whose gradient is Lipschitz continuous on $\mathcal{Y}$.
We still denote the extended function as $f^{\dagger}$. 
Then, together with Assumptions~\ref{regularity} and \ref{primal assump-p}, we can conclude that the master problem $\mathcal{P}^{pd,\dagger}_{mas}$ satisfies \cite[Assumptions~1 and 5]{scutari2017parallel}.
By Assumption~\ref{SCA-assump-surrogate-constriant}, the approximate function $\mathbf{G}_{0}$ satisfies \cite[Assumptions~3]{scutari2017parallel}.
Besides, it is already known that the approximate function $F^{\dagger}$ satisfies \cite[Assumption~2]{scutari2017parallel}.
Then, the first statement of Theorem~\ref{primal thm2}.1 follows immediately from \cite[Theorem~2.(c)]{scutari2017parallel}.

Next, to show the second statement of Theorem~\ref{primal thm2}.1, it remains to prove that \cite[Assumptions~B4, B5, and C7]{scutari2017parallel} are satisfied.
First, it is easy to show that the approximate function $F^{\dagger}_{\mathcal{I}}$, which is constructed by the second-order Taylor approximation of $f^{\dagger}_{\mathcal{I}}$, satisfies \cite[Assumptions B4 and B5]{scutari2017parallel}.
Moreover, $F_0$ satisfies \cite[Assumptions B4 and B5]{scutari2017parallel} since $\nabla_{\mathbf{y}}F_0(\cdot;\mathbf{z})$ and $\nabla_{\mathbf{y}}F_0(\mathbf{y};\cdot)$ are both uniformly Lipschitz continuous on $\mathcal{Y}$.
Therefore, the approximate function $F^{\dagger}$ satisfies \cite[Assumptions B4 and B5]{scutari2017parallel}.
Second, $\mathbf{G}_0(\cdot; \cdot)$ is Lipschitz continuous on $\mathcal{Y}\times \mathcal{Y}$ and hence satisfies \cite[Assumptions C4]{scutari2017parallel}.
Thus, the second statement of Theorem~\ref{primal thm2}.1 follows immediately from \cite[Theorem~2.(c)]{scutari2017parallel}.

\subsubsection{Proof of Theorem~\ref{primal thm2}.2}
For all $i\in\mathcal{I}$, the subsequence $\{\mathbf{x}^{\dag}_i(\mathbf{y}^{(k)})\}_{k\in\mathcal{K}}$ converges to $\mathbf{x}^{\dag}_i(\mathbf{y}^{(\infty)})$, since $\mathbf{x}^{\dag}_i(\cdot)$ is continuously differentiable, and $\{\mathbf{y}^{(k)}\}_{k\in\mathcal{K}}$ converges to $\mathbf{y}^{(\infty)}$.
That is, $\mathbf{x}^{(\infty)}_i=\mathbf{x}^{\dag}_i(\mathbf{y}^{(\infty)})$, $i\in\mathcal{I}$.
Then, by Theorem~\ref{primal thm1}.2 and Theorem~\ref{primal thm2}.1, we can immediately get Theorem~\ref{primal thm2}.2.


\subsection{Equivalent Transformations for Nonconvex Problems}\label{app equivalent nonconvex prob}
We describe some general transformations that yield equivalent nonconvex problems.

\textbf{Introducing Linear Equality Constraints:}
Consider the following nonconvex problem:
\begin{align*}
	\mathcal{P}_{1}: \ \min_{\mathbf{x}_{1}, \mathbf{x}_{2}} \ &f(\mathbf{x}_{1}, \mathbf{x}_{2})\\  
	s.t. \ 
	&\tilde{\mathbf{g}}(\mathbf{x}_{1}, \mathbf{x}_2)\preceq \mathbf{0},\\
	&\tilde{\mathbf{h}}(\mathbf{x}_1, \mathbf{x}_{2})= \mathbf{0},\\
	&\mathbf{g}_{i}(\mathbf{x}_{i})\preceq \mathbf{0},\ i=1,2,\\
	&\mathbf{x}_{i}\in \mathcal{X}_{i}\ i=1,2,
\end{align*}
where $\mathbf{x}_{i}\in\mathbb{R}^{n_{i}}$, $f: \mathbb{R}^{n_1+n_2}\rightarrow \mathbb{R}$,
$\tilde{\mathbf{g}}: \mathbb{R}^{n_1+n_2}\rightarrow \mathbb{R}^{\tilde{r}}$,
$\tilde{\mathbf{h}}: \mathbb{R}^{n_1+n_2}\rightarrow \mathbb{R}^{\tilde{m}}$,
and $\mathbf{g}_{i}: \mathbb{R}^{n_i}\rightarrow \mathbb{R}^{r}$ are continuously differentiable.
and $\mathcal{X}_{i}$ is nonempty, closed, and convex, for $i=1,2$.
We replace $\mathbf{x}_{1}$ with new variable  $\mathbf{z}\triangleq(\mathbf{x}_{1},\mathbf{x}_{2})\in\mathbb{R}^{n_1+n_{2}}$ and introduce new linear equality constraint $\begin{bmatrix}
	\mathbf{0}_{n_{2}\times n_{1}} & \mathbf{I}_{n_{2}}
\end{bmatrix}\mathbf{z}=\mathbf{x}_{2}$ and form the following nonconvex problem:
\begin{align*}
	\mathcal{P}_{1}': \ \min_{\mathbf{z}, \mathbf{x}_{2}} \ &f(\mathbf{z})\\  
	s.t. \ 
	&\mathbf{C}\mathbf{z}=\mathbf{x}_{2},\\
	&\tilde{\mathbf{g}}(\mathbf{z})\preceq \mathbf{0},\\
	&\tilde{\mathbf{h}}(\mathbf{z})= \mathbf{0},\\
	&\mathbf{g}_{1}(\mathbf{A}\mathbf{z})\preceq \mathbf{0},\\
	&\mathbf{g}_{2}(\mathbf{x}_{2})\preceq \mathbf{0},\\
	&\mathbf{z}\in \mathcal{Z}, \\
	&\mathbf{x}_{2}\in \mathcal{X}_{2},
\end{align*}
where $\mathbf{A}\triangleq\begin{bmatrix}
	\mathbf{I}_{n_1} & \mathbf{0}_{n_1\times n_{2}}
\end{bmatrix}$, $\mathbf{C}\triangleq\begin{bmatrix}
	\mathbf{0}_{n_{2}\times n_{1}} & \mathbf{I}_{n_{2}}
\end{bmatrix}$, and $\mathcal{Z}\triangleq\mathcal{X}_1\times\mathcal{X}_{2}$.

\begin{figure*}[ht]
	\begin{gather}
		\sum_{i=1}^{2}(\nabla_{\mathbf{x}_{i}} f(\mathbf{x}_{1}^{\dagger}, \mathbf{x}_{2}^{\dagger})
		\!+\!\nabla_{\mathbf{x}_{i}}\tilde{\mathbf{g}}(\mathbf{x}_{1}^{\dagger}, \mathbf{x}_{2}^{\dagger})\tilde{\boldsymbol{\mu}} 
		\!+\!\nabla_{\mathbf{x}_{i}}\tilde{\mathbf{h}}(\mathbf{x}_{1}^{\dagger}, \mathbf{x}_{2}^{\dagger})\tilde{\boldsymbol{\lambda}} \!+\!\nabla\mathbf{g}_{i}(\mathbf{x}_{i}^{\dagger})\boldsymbol{\mu}_{i})^{T}(\mathbf{x}_{i}\!-\!\mathbf{x}_{i}^{\dagger})\geq\mathbf{0},
		\ \forall (\mathbf{x}_{1},\mathbf{x}_{2})\in\mathcal{X}_{1}\times\mathcal{X}_{2},\label{proof lemma-equivalent-noncvx-1-kkt-1}
		\\
		\tilde{\boldsymbol{\mu}}\odot\tilde{\mathbf{g}}(\mathbf{x}_{1}^{\dagger}, \mathbf{x}_{2}^{\dagger})=\mathbf{0}, \
		\tilde{\mathbf{g}}(\mathbf{x}_{1}^{\dagger}, \mathbf{x}_{2}^{\dagger})\preceq\mathbf{0}, \
		\tilde{\mathbf{h}}(\mathbf{x}_{1}^{\dagger}, \mathbf{x}_{2}^{\dagger})=\mathbf{0}, \
		\boldsymbol{\mu}_{i}\odot\mathbf{g}_{i}(\mathbf{x}_{i}^{\dagger})=\mathbf{0}, \ 
		\mathbf{g}_{i}(\mathbf{x}_{i}^{\dagger})\preceq\mathbf{0},  \
		\mathbf{x}_{i}^{\dagger}\in\mathcal{X}_{i}, i=1,2.\label{proof lemma-equivalent-noncvx-1-kkt-2}
	\end{gather}
	\noindent\hrule
	\begin{gather}
		(\nabla f(\mathbf{z}^{\dagger})
		+\nabla\tilde{\mathbf{g}}(\mathbf{z}^{\dagger})\tilde{\boldsymbol{\mu}}
		+\nabla\tilde{\mathbf{h}}(\mathbf{z}^{\dagger})\tilde{\boldsymbol{\lambda}}
		+\mathbf{A}^{T}\nabla\mathbf{g}_{1}(\mathbf{A}\mathbf{z}^{\dagger})\boldsymbol{\mu}_{1}
		+\mathbf{C}^{T}\boldsymbol{\lambda})^{T}(\mathbf{z}-\mathbf{z}^{\dagger})\nonumber\\
		+(-\mathbf{I}_{n_2}\boldsymbol{\lambda}+\nabla\mathbf{g}_{2}(\mathbf{x}_{2})\boldsymbol{\mu}_{2})^{T}(\mathbf{x}_{2}-\mathbf{x}_{2}^{\dagger})\geq\mathbf{0}, \ \forall \mathbf{z}\in\mathcal{Z}, \ \mathbf{x}_{2}\in\mathcal{X}_{2},\label{proof lemma-equivalent-noncvx-1-kkt-3}
		\\
		\tilde{\boldsymbol{\mu}}\odot\tilde{\mathbf{g}}(\mathbf{z}^{\dagger})=\mathbf{0}, \
		\boldsymbol{\mu}_{1}\odot\mathbf{g}_{1}(\mathbf{A}\mathbf{z}^{\dagger})=\mathbf{0}, \
		\boldsymbol{\mu}_{2}\odot\mathbf{g}_{2}(\mathbf{x}_{2}^{\dagger})=\mathbf{0}, \nonumber\\ 
		\mathbf{C}\mathbf{z}^{\dagger}=\mathbf{x}_{2}^{\dagger},\
		\tilde{\mathbf{g}}(\mathbf{z}^{\dagger})\preceq\mathbf{0}, \
		\tilde{\mathbf{h}}(\mathbf{z}^{\dagger})=\mathbf{0}, \
		\mathbf{g}_{1}(\mathbf{A}\mathbf{z}^{\dagger})\preceq\mathbf{0}, \
		\mathbf{g}_{2}(\mathbf{x}_{2}^{\dagger})\preceq\mathbf{0}, \
		\mathbf{z}^{\dagger}\in\mathcal{Z}, \mathbf{x}_{2}^{\dagger}\in\mathcal{X}_{2}.\label{proof lemma-equivalent-noncvx-1-kkt-4}
	\end{gather}
	\noindent\hrule
\end{figure*}

\begin{lemma}\label{lemma equivalent-noncvx-1}
	Problem $\mathcal{P}_{1}$ and Problem $\mathcal{P}_{1}'$ are equivalent.
\end{lemma}
\begin{IEEEproof}
	If $(\mathbf{x}^{\dagger}_{1}, \mathbf{x}^{\dagger}_{2})$ is a stationary point of Problem $\mathcal{P}_{1}$, then there exist Lagrange multipliers $\tilde{\boldsymbol{\mu}}\succeq\mathbf{0}$, $\tilde{\boldsymbol{\lambda}}$ and $\boldsymbol{\mu}_{i}\succeq\mathbf{0}$, $i=1,2$ that together with $(\mathbf{x}^{\dagger}_{1}, \mathbf{x}^{\dagger}_{2})$ satisfy (\ref{proof lemma-equivalent-noncvx-1-kkt-1}) and (\ref{proof lemma-equivalent-noncvx-1-kkt-2}), as shown at the top of this page.
	If $(\mathbf{z}^{\dagger}, \mathbf{x}_{2}^{\dagger})$ is a stationary point of Problem $\mathcal{P}_{1}'$, then there exist Lagrange multipliers
	$\boldsymbol{\lambda}$, $\tilde{\boldsymbol{\mu}}\succeq\mathbf{0}$, $\tilde{\boldsymbol{\lambda}}$, $\boldsymbol{\mu}_{1}\succeq\mathbf{0}$, and $\boldsymbol{\mu}_{2}\succeq\mathbf{0}$ that together with $(\mathbf{z}^{\dagger}, \mathbf{x}^{\dagger}_{2})$ satisfy (\ref{proof lemma-equivalent-noncvx-1-kkt-3}) and (\ref{proof lemma-equivalent-noncvx-1-kkt-4}), as shown at the top of this page.
	First, suppose that $(\mathbf{x}^{\dagger}_{1}, \mathbf{x}^{\dagger}_{2})$ is a stationary point of Problem $\mathcal{P}_{1}$, and $\tilde{\boldsymbol{\mu}}\succeq\mathbf{0}$, $\tilde{\boldsymbol{\lambda}}$ and $\boldsymbol{\mu}_{i}\succeq\mathbf{0}$, $i=1,2$ are the corresponding Lagrange multipliers.
	Let $\mathbf{z}^{\dagger}=(\mathbf{x}^{\dagger}_{1}, \mathbf{x}_{2}^{\dagger})$ and
	$\boldsymbol{\lambda}=\mathbf{0}$.
	From (\ref{proof lemma-equivalent-noncvx-1-kkt-1}) and (\ref{proof lemma-equivalent-noncvx-1-kkt-2}), we can show that (\ref{proof lemma-equivalent-noncvx-1-kkt-3}) and (\ref{proof lemma-equivalent-noncvx-1-kkt-4}) hold true for $(\mathbf{z}^{\dagger}, \mathbf{x}_{2}^{\dagger})$, $\tilde{\boldsymbol{\mu}}$, $\tilde{\boldsymbol{\lambda}}$, $\boldsymbol{\mu}_{i}$, $i=1,2$, and $\boldsymbol{\lambda}$, which implies that $(\mathbf{z}^{\dagger}, \mathbf{x}_{2}^{\dagger})$ is a stationary point of Problem $\mathcal{P}_{1}'$.
	Next, suppose that $(\mathbf{z}^{\dagger}, \mathbf{x}^{\dagger}_{2})$ is a stationary point of Problem $\mathcal{P}_{1}'$, and $\boldsymbol{\lambda}$, $\tilde{\boldsymbol{\mu}}\succeq\mathbf{0}$, $\tilde{\boldsymbol{\lambda}}$, $\boldsymbol{\mu}_{1}\succeq\mathbf{0}$, and $\boldsymbol{\mu}_{2}\succeq\mathbf{0}$ are the corresponding Lagrange multipliers.
	From the definition of $\mathbf{C}$, we can get $(\mathbf{C}^{T}\boldsymbol{\lambda})^{T}(\mathbf{z}-\mathbf{z}^{\dagger})+(-\mathbf{I}_{n_2}\boldsymbol{\lambda})^{T}(\mathbf{x}_{2}-\mathbf{x}_{2}^{\dagger})=\mathbf{0}$ for all $\mathbf{C}\mathbf{z}=\mathbf{x}_{2}$ and $\boldsymbol{\lambda}$.
	Thus, (\ref{proof lemma-equivalent-noncvx-1-kkt-3}) can reduce to (\ref{proof lemma-equivalent-noncvx-1-kkt-1}), as shown at the top of the last page, for all $\mathbf{C}\mathbf{z}=\mathbf{x}_{2}$ and $\boldsymbol{\lambda}$.
	Let $\mathbf{x}_{1}^{\dagger}=\mathbf{A}\mathbf{z}^{\dagger}$.
	Then, from (\ref{proof lemma-equivalent-noncvx-1-kkt-3}) and (\ref{proof lemma-equivalent-noncvx-1-kkt-4}), we can show that (\ref{proof lemma-equivalent-noncvx-1-kkt-1}) and (\ref{proof lemma-equivalent-noncvx-1-kkt-2}) hold true for $(\mathbf{x}_{1}^{\dagger}, \mathbf{x}_{2}^{\dagger})$,
	$\tilde{\boldsymbol{\mu}}$, $\tilde{\boldsymbol{\lambda}}$, and $\boldsymbol{\mu}_{i}\succeq\mathbf{0}$, $i=1,2$, which implies that $(\mathbf{x}_{1}^{\dagger}, \mathbf{x}_{2}^{\dagger})$ is a stationary point of Problem $\mathcal{P}_{1}$.
\end{IEEEproof}

\textbf{Introducing separable Constraints:}
We consider the following nonconvex problem:
\begin{align*}
	\mathcal{P}_{2}: \ \min_{\mathbf{x}} \ &f(\mathbf{x})\\  
	s.t. \ 
	&\sum_{i\in \mathcal{I}}\mathbf{g}_{i}(\mathbf{x}_{i})\preceq \mathbf{0},\\
	&\sum_{i\in \mathcal{I}}\mathbf{h}_{i}(\mathbf{x}_{i})= \mathbf{0},\\
	&\mathbf{x}\in \mathcal{X},
\end{align*}
where $\mathbf{x}\triangleq(\mathbf{x}_{1},\cdots,\mathbf{x}_{I})\in\mathbb{R}^{n}$ with $\mathbf{x}_{i}\in\mathbb{R}^{n_{i}}$, $f: \mathbb{R}^{n}\rightarrow \mathbb{R}$, $\mathbf{g}_{i}: \mathbb{R}^{n_{i}}\rightarrow \mathbb{R}^{r}$, and $\mathbf{h}_{i}: \mathbb{R}^{n_{i}}\rightarrow \mathbb{R}^{m}$, $i\in\mathcal{I}$ are continuously differentiable, 
and $\mathcal{X}$ is nonempty, closed, and convex.
We introduce new variables $\mathbf{z}_{i}\in\mathbb{R}^{r}$ and $\mathbf{z}_{i+I}\in\mathbb{R}^{m}$, $i\in\mathcal{I}$, new inequality constraints $\mathbf{g}_{i}(\mathbf{x}_{i})-\mathbf{z}_{i}\preceq \mathbf{0}$, $i\in\mathcal{I}$, and new equality constraints $\mathbf{h}_{i}(\mathbf{x}_{i})-\mathbf{z}_{i+I}=\mathbf{0}$, $i\in\mathcal{I}$,  
and form the following nonconvex problem:
\begin{align*}
	\mathcal{P}_{2}': \ \min_{\mathbf{x},\mathbf{z}} \ &f(\mathbf{x})\\  
	s.t. \ 
	&\mathbf{g}_{i}(\mathbf{x}_{i})-\mathbf{z}_{i}\preceq \mathbf{0},\ i\in\mathcal{I}\\
	&\sum_{i\in \mathcal{I}}\mathbf{z}_{i}\preceq\mathbf{0},\\
	&\mathbf{h}_{i}(\mathbf{x}_{i})-\mathbf{z}_{i+I}=\mathbf{0}, \ i\in\mathcal{I}\\
	&\sum_{i\in \mathcal{I}}\mathbf{z}_{i+I}=\mathbf{0},\\
	&\mathbf{x}\in \mathcal{X}.
\end{align*}

\begin{figure*}[ht]

	\begin{gather}
		\sum_{i\in \mathcal{I}}\left(\nabla_{\mathbf{x}_{i}} f(\mathbf{x}^{\dagger})
		+\nabla\mathbf{g}_{i}(\mathbf{x}^{\dagger}_{i})\boldsymbol{\mu}
		+\nabla\mathbf{h}_{i}(\mathbf{x}^{\dagger}_{i})\boldsymbol{\lambda}
		\right)^{T}(\mathbf{x}_{i}-\mathbf{x}_{i}^{\dagger}), 
		\ \forall \mathbf{x}\in\mathcal{X},\label{proof lemma-equivalent-noncvx-2-kkt-1}\\
		\sum_{i\in \mathcal{I}}\boldsymbol{\mu}\odot\mathbf{g}_{i}(\mathbf{x}^{\dagger}_{i})=\mathbf{0}, \
		\sum_{i\in \mathcal{I}}\mathbf{g}_{i}(\mathbf{x}^{\dagger}_{i})\preceq \mathbf{0}, \
		\sum_{i\in \mathcal{I}}\mathbf{h}_{i}(\mathbf{x}^{\dagger}_{i})= \mathbf{0}, \
		\mathbf{x}^{\dagger}\in \mathcal{X}.\label{proof lemma-equivalent-noncvx-2-kkt-2}
	\end{gather}
	\noindent\hrule
	\begin{gather}
		\sum_{i\in \mathcal{I}}(\nabla_{\mathbf{x}_{i}} f(\mathbf{x}^{\dagger})
		\!+\!\nabla\mathbf{g}_{i}(\mathbf{x}^{\dagger}_{i})\boldsymbol{\mu}_{i}
		\!+\!\nabla\mathbf{h}_{i}(\mathbf{x}^{\dagger}_{i})\boldsymbol{\lambda}_{i}
		)^{T}(\mathbf{x}_{i}\!-\!\mathbf{x}_{i}^{\dagger})\nonumber
		\\
		\!+\sum_{i\in \mathcal{I}}\!(\!-\!\boldsymbol{\mu}_{i}\!+\!\boldsymbol{\mu})^{T}(\mathbf{z}_{i}\!-\!\mathbf{z}^{\dagger}_{i})
		\!+\!(\!-\!\boldsymbol{\lambda}_{i}\!+\!\boldsymbol{\lambda})^{T}(\mathbf{z}_{i+I}\!-\!\mathbf{z}^{\dagger}_{i+I})
		\geq\mathbf{0}, \
		\forall \mathbf{x}\in\mathcal{X}, \mathbf{z}\in\mathbb{R}^{rI+mI},\label{proof lemma-equivalent-noncvx-2-kkt-3}
		\\
		\hspace{-5mm}\sum_{i\in \mathcal{I}}\boldsymbol{\mu}\odot\mathbf{z}_{i}^{\dagger}\!=\!\mathbf{0},\
		\sum_{i\in \mathcal{I}}\mathbf{z}_{i}^{\dagger}\!\preceq\!\mathbf{0}, \ \sum_{i\in \mathcal{I}}\mathbf{z}^{\dagger}_{i+I}=\!\mathbf{0},\
		\mathbf{x}^{\dagger}\in \mathcal{X},\
		\boldsymbol{\mu}_{i}\odot(\mathbf{g}_{i}(\mathbf{x}^{\dagger}_{i})\!-\!\mathbf{z}_{i}^{\dagger})\!=\!\mathbf{0}, \
		\mathbf{g}_{i}(\mathbf{x}^{\dagger}_{i})\!-\!\mathbf{z}_{i}^{\dagger}\!\preceq\!\mathbf{0},\ 
		\mathbf{h}_{i}(\mathbf{x}^{\dagger}_{i})\!-\!\mathbf{z}_{i}^{\dagger}\!=\!\mathbf{0},\ 
		i\in\mathcal{I}.\label{proof lemma-equivalent-noncvx-2-kkt-4}
	\end{gather}
	\noindent\hrule
\end{figure*}

\begin{lemma}\label{lemma equivalent-noncvx-2}
	Problem $\mathcal{P}_{2}$ and Problem $\mathcal{P}_{2}'$ are equivalent.
\end{lemma}
\begin{IEEEproof}
	If $\mathbf{x}^{\dagger}$ is a stationary point of Problem $\mathcal{P}_{2}$, then there exist Lagrange multipliers $\boldsymbol{\mu}\succeq\mathbf{0}$ and $\boldsymbol{\lambda}$ that together with $\mathbf{x}^{\dagger}$ satisfy (\ref{proof lemma-equivalent-noncvx-2-kkt-1}) and (\ref{proof lemma-equivalent-noncvx-2-kkt-2}), as shown at the top of next page.
	If $(\mathbf{x}^{\dagger}, \mathbf{z}^{\dagger})$ is a stationary point of Problem $\mathcal{P}_{2}'$, then there exist Lagrange multipliers $\boldsymbol{\mu}_{i}\succeq\mathbf{0}$, $\boldsymbol{\lambda}_{i}$, $i\in\mathcal{I}$, $\boldsymbol{\mu}\succeq\mathbf{0}$, and $\boldsymbol{\lambda}$ that together with $(\mathbf{x}^{\dagger}, \mathbf{z}^{\dagger})$ satisfy (\ref{proof lemma-equivalent-noncvx-2-kkt-3}) and (\ref{proof lemma-equivalent-noncvx-2-kkt-4}), as shown at the top of next page.
	First, suppose that $\mathbf{x}^{\dagger}$ is a stationary point of Problem $\mathcal{P}_{2}$, and $\boldsymbol{\mu}$ and $\boldsymbol{\lambda}$ are the corresponding Lagrange multipliers.
	Let $\mathbf{z}_{i}^{\dagger}=\mathbf{g}_{i}(\mathbf{x}^{\dagger}_{i})$, $\mathbf{z}_{i+I}^{\dagger}=\mathbf{h}_{i}(\mathbf{x}^{\dagger}_{i})$, 
	$\boldsymbol{\mu}_{i}=\boldsymbol{\mu}$, and $\boldsymbol{\lambda}_{i}=\boldsymbol{\lambda}$,  $i\in\mathcal{I}$.
	Then, from (\ref{proof lemma-equivalent-noncvx-2-kkt-1}) and (\ref{proof lemma-equivalent-noncvx-2-kkt-2}), we can conclude that $(\mathbf{x}^{\dagger}, \mathbf{z}^{\dagger})$ together with $\boldsymbol{\mu}$, $\boldsymbol{\lambda}$, $\boldsymbol{\mu}_{i}$, and $\boldsymbol{\lambda}_{i}$, $i\in\mathcal{I}$, satisfies (\ref{proof lemma-equivalent-noncvx-2-kkt-3}) and (\ref{proof lemma-equivalent-noncvx-2-kkt-4}), which implies that $(\mathbf{x}^{\dagger}, \mathbf{z^{\dagger}})$ is a stationary point of Problem $\mathcal{P}_{2}'$.
	Next, suppose that $(\mathbf{x}^{\dagger}, \mathbf{z}^{\dagger})$ is a stationary point of Problem $\mathcal{P}_{2}'$, and $\boldsymbol{\mu}$, $\boldsymbol{\lambda}$, $\boldsymbol{\mu}_{i}$, and $\boldsymbol{\lambda}_{i}$, $i\in\mathcal{I}$ are the corresponding Lagrange multipliers.
	Substituting $\mathbf{x}=\mathbf{x}^{\dagger}$ into (\ref{proof lemma-equivalent-noncvx-2-kkt-3}), we have 
	$
	\sum_{i\in \mathcal{I}}(-\boldsymbol{\mu}_{i}'+\boldsymbol{\mu})^{T}(\mathbf{z}_{i}-\mathbf{z}^{\dagger}_{i})
	+\sum_{i\in \mathcal{I}}(-\boldsymbol{\lambda}_{i}'+\boldsymbol{\lambda})^{T}(\mathbf{z}_{i+I}-\mathbf{z}^{\dagger}_{i+I})
	\geq\mathbf{0}
	$ for all $\mathbf{z}\in\mathbb{R}^{rI+mI}$, 
	which implies $\boldsymbol{\mu}_{i}=\boldsymbol{\mu}$ and $\boldsymbol{\lambda}_{i}=\boldsymbol{\lambda}$, $i\in\mathcal{I}$.
	Thus, (\ref{proof lemma-equivalent-noncvx-2-kkt-3}) reduces to (\ref{proof lemma-equivalent-noncvx-2-kkt-1}). 
	Then, by eliminating variables $\mathbf{z}^{\dagger}_{i}$, $i\in\mathcal{I}$ in (\ref{proof lemma-equivalent-noncvx-2-kkt-4}) and replacing variables $\mathbf{z}^{\dagger}_{i+I}$ in (\ref{proof lemma-equivalent-noncvx-2-kkt-4}) with $\mathbf{h}_{i}(\mathbf{x}^{\dagger}_{i})$ based on the equalities $\mathbf{h}_{i}(\mathbf{x}^{\dagger}_{i})-\mathbf{z}_{i+I}^{\dagger}=\mathbf{0}$, $i\in\mathcal{I}$, (\ref{proof lemma-equivalent-noncvx-2-kkt-4}) reduces to (\ref{proof lemma-equivalent-noncvx-2-kkt-2}).
	Thus, (\ref{proof lemma-equivalent-noncvx-2-kkt-1}) and (\ref{proof lemma-equivalent-noncvx-2-kkt-2}) hold true for $\mathbf{x}^{\dagger}$,  $\boldsymbol{\mu}$, and $\boldsymbol{\lambda}$, which implies that $\mathbf{x}^{\dagger}$ is a stationary point of Problem $\mathcal{P}_{2}$.
\end{IEEEproof}

%
%

\bibliographystyle{IEEEtran}
\bibliography{IEEEabrv,references}

\vfill

\end{document}